\documentclass{puthesis-UG}
\usepackage{amsthm,amssymb,amsmath,thmtools,graphicx}
\usepackage{etoolbox}
\usepackage[utf8]{inputenc}
\usepackage[T1]{fontenc}
\AtBeginEnvironment{quote}{\par\singlespacing\small}
\newcommand{\introthmname}{}

\DeclareMathOperator*{\bigsquare}{\raisebox{-0.9ex}{\scalebox{2}{$\Box$}}}
\DeclareMathOperator*{\bigxsquare}{\raisebox{-0.9ex}{\scalebox{2}{$\boxtimes$}}}

\author{I.M.J. McInnis}
\adviser{Noga Alon}
\title{Slavic Techniques for Hat Guessing Algorithms}
\abstract{
    We study a deterministic hat game. Sages stand on the vertices of a ``visibility (di)graph'' $D$. A crafty Adversary assigns to each sage $k$ a hat of color $c(k)\in [h(k)]$, i.e. from among $h(k)$ predetermined colors. Then each sage, based on the hat colors of her out-neighbors, guesses $g(k)$ colors for her own hat. They can strategize beforehand but not communicate during. They win if they can guarantee that at least one correct guess occurs. For which games $(D,g,h)$ can they guarantee victory? Parameters related to such questions have been well-studied, especially the ``hat guessing number'' $\mu(D)$. 
    
    We open with a casual, riddle-based invitation to the topic and its motivation. Rigor enters with a game taxonomy and a survey of all previous 
    work. After introduction and definitions, we prove warm-up propositions, including fundamental lemmata and the solution to most directed-cycle games. In the remaining chapters, we formalize, generalize, and utilize techniques from the literature.

    We take a ``Hats as Hints'' perspective and use the ``admissible paths'' of \cite{Szc17} to prove the conjecture of \cite{KLR21b} and go farther to settle all single-guess games on cycles. We generalize several ``constructors,'' central theorems that build games while preserving (un)winnableness. Basic ones simplify \cite{KL19}
    and algorithmically solve single-guess games on trees (two ways). A ``combinatorial prism'' approach gives rules for strategy design, a characterization of deletable vertices, and packing/covering conditions for games on $K_{n,m}$. Then we apply the probabilistic method, defining ``Ratio games'' to crystallize work on dependency (di)graphs, graph polynomials, and unwinnable games. We show that the visibility digraph is dual to a dependency digraph and that Shearer's Lemma probably can't extend to digraphs. 

    The final chapter collects open problems, including all that have been explicitly stated in the literature. Others are first posed here. It includes partial or minor results, such as curiosities about $\overrightarrow{C}_k$ and some computational complexity facts.}
\acknowledgements{First and most, I thank my brilliant advisor, Noga Alon, for his generosity and wisdom. Obviously, this thesis could not exist without him. I thank Matija Buci\'c, who––poor fellow––learned that it would be $\approx$thrice the standard length only after agreeing to be the second reader. For feedback, I thank Lidia, Beatrix, and especially  Bracklinn, who checked every word and every proof.

All other thanks are personal. Within categories, the order is roughly chronological. 

Marie Beaurline, Zuming Feng, Jeff Ibbotson, John Mosely, Clifford Gunning, Hans Halvorson, Charles Johnson, Alan Chang, Maria Chudnovsky, June Huh, and Corina Tarnita: thanks for inspiring me toward math. 

For teaching me how to write, think, and live, I thank Elisabeth Barr, Tolly Merrick, Faye Cunningham, Charlotte Wood, Todd Hearon, Michele Chapman, Townley Chisholm, Rob Richards, Mercy Carbonell, Efthymia Rentzou, John McPhee, and A.M. Homes.

Without the blessings of my family, I'd be a sorry case indeed. Mom, Dad, Kaki, Gran, Stephen, William, David, Debby, Doug, Fiona, Juan José, Angelina, Matthew, Sophia, Rona, and Angus, thank y'all. 

I thank dear friends: Sarah, Kevin, Sophie, Abby, Emilio, Traxler, Miguel, Ori, Billy, Aly, Joanna, Gordon, Madison, Nick, Andrew, Tim, Franny, Alexander, Ahmed, Lexi, Katie, Gideon, Tyler, Juan, Sumanth, Vinny, Tyler, Erik, Nia, Anthony, Megan, Elle, Aditya, Hannah, Joshua, Ben, Josh, Jay, Noah, Conor, Nik, Jello, David, Bracklinn, Jo, Simone, Zander, Marty, Haller, Tobi, Mason, Danica, Colby, Nishant, David, Musab, and Yousuf.

I thank Father Allen and all of ECP; I thank Ryan and all of NOVA. Insofar as it can serve such an end, all I write is \textit{ad maiorem Dei gloriam}. 
}
\date{April 28, 2023}
\usepackage{amsthm,amsmath,amssymb,enumitem,epigraph}
\usepackage[utf8]{inputenc}
\newtheorem{thm}{Theorem}[section]
\newtheorem*{thm*}{Theorem}
\newtheorem{prop}[thm]{Proposition}
\newtheorem{proposition}[thm]{Proposition}
\newtheorem{coro}[thm]{Corollary}
\newtheorem{corollary}[thm]{Corollary}
\newtheorem{lemma}[thm]{Lemma}
\newtheorem{lmm}[thm]{Lemma}
\theoremstyle{definition}
\newtheorem{defn}[thm]{Definition}
\theoremstyle{definition}
\newtheorem{game}[thm]{Game}
\theoremstyle{definition}
\newtheorem{ques}[thm]{Question}
\theoremstyle{definition}

\theoremstyle{definition}
\newtheorem{remark}[thm]{Remark}

\begin{document}
\setcounter{chapter}{-1}
\chapter{Invitation} 
A mathematically focused reader may skip this chapter, which is (intended to be) accessible and fun for the general public.\footnote{Its inclusion is substantially motivated by the requirement that I present a chapter of my thesis to a colloquium of European Cultural Studies students, and I didn't figure they'd enjoy a taxonomy of low-hatness Latvian cycles. Also N.B., there's no relation between the ECS requirement and the ``Slavic'' of this paper. It just happens that my favorite hat guessing work comes from Eastern Europe.} We introduce the topic with a riddle and walk the reader through its solution. Then we explain why someone might care.

\section{A fun riddle}
At a Goodwill in Texas some years back, I bought a book of riddles \cite{Gui14} for two dollars and ninety-nine cents plus tax. I was trying to build my riddle repertoire. See, each fall, I lead a dozen incoming Princeton frosh on a backpacking trip to help them adjust psychologically to college. They tend to like riddles. Hence the book. Its best was the seventy-seventh out of seventy-eight; no frosh ever solved it. Here it is (heavily paraphrased).

\begin{game}\label{impossiblerainbow}
Several sages––Agatha, Brigid, Catherine, Demetria, Edwen, Felicity, and Gaudentia––are seated in a circle, planning. In a moment, their Adversary will place upon each sage's head a hat from his Chest of Hats. The prosperous, eccentric Adversary has hundreds of red hats, hundreds of orange hats, and so on for yellow, green, blue, indigo, and violet each––more hats than any reasonable person could want––he is unconstrained by any scarcity of hats.
So every arrangement and combination is possible, from ``one of each'' to ``all violet.'' Once the hats are placed––dextrously and instantly––the sages may look around. Each sage can see the color of every other sage's hat, but she cannot see the color of her own. (These are comically small bowler hats.) Then each sage must guess the color of her own hat, based only on what she sees on whom. This is done without coordination––each sage 
scratches her guess on a private ostracon, or maybe they all say their guesses out loud on the count of three. The point is, once the hats have been placed, no communication can occur at all. No blinking in Morse code, or whatever. The sages win if at least one of them guesses her own hat color correctly. But the game hasn't started yet. For now, they're discussing strategy. If they lose, the consequences will be unspeakable. On top of that, the Adversary is eavesdropping, and if there's a single hole in the sages' strategy––a single hat arrangement where no-one guesses right––he'll find that gap and exploit it. So the sages must \textit{guarantee} victory! Can they? How? Pause and see whether you can solve it. 
\end{game}

The book's author called this ``The Impossible Rainbow.'' The name's apt. Most folks (including me, my frosh, my friends, and ChatGPT) insist it's impossible. They confidently explain why: ``For someone to know her hat color, she'd have to get that information somehow. She can't get it through communication, because that's forbidden. The only data she gets is the others' colors, but that doesn't help at all. She might e.g. see nothing but red and want to guess red, but the Adversary can do any arrangement he wants. It could be seven reds––but it could also be six reds and one orange, yellow, green, blue, indigo, or violet! So there's no way to actually get information.'' Can you see why this is wrong? (If you can, by the way, or solved the riddle from the outset, you're smarter than I. Or more patient, or possibly both.)

What that argument actually shows is that we can't guarantee that \textit{any particular} sage guesses right. It's impossible to guarantee that Agatha guesses right, or to guarantee that Brigid guesses right, etc.––but it turns out you can guarantee that \textit{someone} will guess right, never mind who. That's what hat guessing is all about: coordinating how different actors process information to ensure they capture something true. 

The answer seemed like magic when I saw it first, but I'll show you how anybody could have come up with it.\footnote{Though in practice ``anybody'' is only Jacob Neis, a grad student in Classics.} (If, at some point along the reasoning, you realize the answer, feel free to skip ahead to Proposition \ref{proofofrainbow}.) 
To understand the principles, it
helps to consider a smaller version. 

\begin{game}\label{smallrainbow}
    Agatha and Brigid are the only sages; red and orange are the only hat colors. Otherwise all is as above in Game \ref{impossiblerainbow}.
\end{game}
If you don't want to be clever, you can solve this version of the riddle simply by getting out a pencil and writing down the 16 possible strategies.\footnote{Relish that. For precious few games is such ``bashing'' feasible.} (Agatha can guess the color she sees, guess the color she doesn't see, always guess red, or always guess orange. Brigid has the same options. $4\cdot 4=16$.) Then test each against the 4 possible hat assignments. (Agatha can have a red or orange hat; Brigid can have a red or orange hat; $2\cdot 2=4$.) This identifies immediately which strategies always win.\footnote{The analogous procedure for the seven-sages version would require you to write $(7^{7^6})^7$ strategies and test each against $7^7$ hat assignments. Taking one second to check each hat assignment, this will take a little more than $10^{94352}$ years.} You'll realize that (up to symmetry) there's only one solution. Agatha guesses whatever she doesn't see, and Brigid guesses whatever she does. 
This can be rephrased as, ``A assumes their colors are different, and B assumes their colors are the same." 

What are the general principles in that guessing strategy? For a sage to guess according to some assumption about the world, that assumption (plus what she sees) ought to always yield a unique guess. For instance, if Agatha assumed ``we both have blue hats'' and she saw a red hat on Brigid, she couldn't make any guess consistent with her assumption. If Agatha assumed ``we don't both have red hats,'' and she saw a blue hat on Brigid, she wouldn't know whether to guess red or blue. If a sage's assumption doesn't yield a unique guess, then that sage isn't very useful to the collective endeavor. 

Furthermore, all the sages' assumptions together must cover the space of possibilities––i.e., \textit{someone} has to be right. Since a sage's assumption gives her a unique guess, her guess is correct if and only if her assumption is correct. So to win, it's necessary and sufficient to guarantee that some sage's assumption will be correct. 

How can we do all this in the seven-sage, seven-color case? To make things easy, let's assign a number to each color. Mathematicians like to start things at zero, so let's say red=0, orange=1, yellow=2, green=3, blue=4, indigo=5, violet=6. Let's say that $c(A)$ represents the color of the hat on Agatha's head. $c(B)$ represents the color of the hat on Brigid's head, and so forth. 

We can write a hat assignment for everyone as a vector, like $(1, 3, 5, 5, 5, 5, 2)$. In that assignment, Agatha gets an orange hat, Brigid gets a green hat, Gaudentia gets a yellow hat, and everyone else gets an indigo hat. Agatha would know that the hat assignment is $(c(A), 3, 5, 5, 5, 5, 2)$ for some $c(A)\in \{0,1,2,3,4,5,6\}$, but she doesn't know what $c(A)$ is. Brigid would know that the hat assignment is $(1, c(B), 5, 5, 5, 5, 2)$ for some $c(B)$, but she doesn't know what $c(B)$ is. And so on for all the other sages. What are the right assumptions?

Let's go back to the small version for inspiration. The possible hat-assignment-vectors are $(0,0)$, $(0,1)$, $(1,0)$, and $(1,1)$. Agatha is assuming that  $c(A)=c(B)$, while Brigid assumes that $c(A)\neq c(B)$. In other words, Agatha assumes that $c(A)+c(B)$ is even, while Brigid assumes that it's odd.\footnote{This might seem like the most unnatural step in the reasoning. If you're not big on math, the most likely way you'd discover it is probably by looking at the three-sage, three-color version. That's how Jacob did it, doodling in the corner of his sheet music in Chapel Choir rehearsal.} In math, we like to write ``$x$ is odd'' as $x \equiv 1 \mod 2$. (We say that as ``$x$ is congruent to one, modulo two.'') It means that the remainder of $x$, when divided by $2$, is $1$. Similarly, we write ``$x$ is even'' as ``$x \equiv 0 \mod 2$.'' In general, ``$x \equiv y \mod z$'' means that ``the remainder of $x$, when divided by $z$, is $y$,'' and this makes sense so long as $y<z$. This is called modular arithmetic.

So you might guess that the strategy for the seven-sage, seven-color version is: Agatha assumes that $c(A)+c(B)+c(C)+c(D)+c(E)+c(F)+c(G)\equiv 0 \mod 7$. Brigid assumes that $c(A)+c(B)+c(C)+c(D)+c(E)+c(F)+c(G)\equiv 1 \mod 7$. And so on. 

It's not too hard to show that these assumptions all give one unique guess. Say $q$ is the sum of everyone else's hat colors, which you know. $x$ is your own hat color, which you're trying to guess. $a$ is your assumption for what the sum is congruent to modulo 7. Since $x$ is in $\{0,1,2,3,4,5,6\}$ (but can be anything in that set) there's a unique value of $x$ such that $x+q \equiv a \mod 7$. No matter what, that sum will be congruent modulo 7 to $0,1,2,3,4,5,$ or $6$. So exactly one sage's assumption will be correct. Thus, we've met all the necessary conditions, and this strategy works! Try it for yourself and six friends. Hours of fun guaranteed. 

\section{Anything fun is fun to generalize}

You might find yourself asking: what if we adjust the number of colors and sages? (You also might not.) What if they're not equal? That drive is what leads us to real mathematics, which starts now.

\begin{prop}\label{proofofrainbow}
    Consider Game \ref{impossiblerainbow}, except with $s$ sages and $h$ hat colors. The sages can guarantee victory if and only if $s\geq h$. 
\end{prop}
\begin{proof}
    We'll prove ``if'' by showing a strategy, and we'll prove ``only if'' by randomizing the hats. 

    When $s\geq h$, the sages can win with modular arithmetic. Assign each color a unique value in $[h]\equiv \{0,1,...,h-1\}$. Assign $h$ of the sages a value in $[h]$ as well, as a sort of ID number. (The other $s-h$ sages can do whatever, and this strategy will still work.) For each sage $i$, let $c(i)$ denote the color of her hat. 
    Each sage looks around the room and calculates $\sum_{j\neq i} c(j)$ (i.e., the sum of everyone else's colors.) Then she guesses that $c(i)\equiv i-\sum_{j\neq i} c(j) \mod h$, which is just a rearrangement of her assumption that $\sum_{j\in [h]} c(j) \equiv i \mod h$ (i.e., that the sum of all colors is has remainder $i$ when divided by 7). Since $\sum_{j\in [h]} c(j) \mod h$ is some fixed $k\in [h]$, sage $k$ guesses correctly. Thus, the sages always win with this strategy. 

    When $h>s$, brute probability means some hat arrangements will slip past any strategy. Fix a strategy for the sages. Let the hat colors be chosen uniformly and independently at random. Since a sage's own hat color is independent of the hat colors she can see––which uniquely determine her guess––the probability of her guessing correctly is $h^{-1}$. Thus, the expected number of sages who guess correctly is $sh^{-1}<1$, so there is some hat arrangement with no sages guessing correctly. 
\end{proof}

If you want to keep going, there are many steps you can take to generalize further. As we just did, you can change the number of sages and the number of hat colors. You can penalize incorrect guesses and maybe permit abstention. 
You can demand that $k$ sages guess correctly for various $k>1$. You can introduce probabilistic considerations––instead of guaranteeing some right guess, you can try to maximize the chance that everyone is simultaneously correct (or pursue any of the $\approx 36$ probabilistic variations that \cite{Krz10} mentioned). 
You can have the number of possible hat colors differ from sage to sage. You can have infinitely many sages. The sages can guess in order and hear each other's guesses. You can give certain sages hints. You may require their guessing rules to be linear functions. 
The adversary may have a limited number of total hats at his disposal. Some hat assignments may be forbidden. And so on. These can be combined with the most interesting innovation: not all sages can see each other. 

\begin{game}
 Night has fallen. Our seven sages are still seated in a circle when the Adversary returns, proposing a rematch. The sages demur. ``The room is so dark,'' they say, ``that each of us can see only her immediate neighbors in the circle.'' The Adversary says that's no problem, offering a handicap of his own: he promises to use only red, orange, and yellow hats. Can the sages win? 
\end{game}

Recall that by Proposition \ref{proofofrainbow}, it takes all seven sages to win Game \ref{impossiblerainbow}. Here, seven sages lose. Yet if one sage has wandered off to pray (and the other six have closed her gap in the cirlce), the remaining six win. If another leaves, the remaining five lose. If yet another leaves, the remaining four win. In fact, the sages win this game if and only if the number $s$ of sages is equal to 4 or divisible by 3. That theorem is difficult; see Section~\ref{sec:MCP} and \cite{Szc17}. 
 

What about other patterns of visibility? What if everyone can see two spaces to either side? Or only to the right? What if everyone can see each other except for Agatha and Gaudentia? There are limitless questions of the form. We can describe a visibility pattern by a mathematical object called a \textit{directed graph} (\textit{digraph} for short). This isn't the graph you're familiar with from precalc, the graph of a function. This is a network. There are dots we call \textit{nodes} or \textit{vertices}. There are \textit{arrows} or \textit{arcs} that go from vertex to vertex.  The specific way you draw it doesn't matter, just which nodes have arrows between them and which ways those arrows point. If the arrows always point both directions, it's simply called a \textit{graph}. Graphs and digraphs matter for chemistry, linguistics, physics, social science, computing, biology, and more. (Mostly, I think they're neat.) In this case, we can draw any (di)graph we want and say that there's a sage at every vertex, and sage $x$ can see sage $y$ if and only if there's an arrow from vertex $x$ to vertex $y$.\footnote{Some authors use the opposite convention, so that arcs map the flow of information. I find ours more intuitive. Despite Alhazen's debunking \cite{Ada16}
of extramission––the Platonic/Galenic idea
 that vision proceeds by rays emitted from the eye \cite{Pla60, Fin94}––an arrow from $u$ to $v$ just feels like ``$u$ sees $v$.'' That may differ for speakers of languages in which the sight-verb takes the seen thing as subject and seer as object, but I speak no such language and can't shake the feeling. (It is perhaps noteworthy that \textit{many} college-educated American adults believe in extramission, and it's difficult to convince them otherwise \cite{WCKC96, WCKG96, WCGFB02, WRC03}.)}

The Adversary has $k$ predetermined colors at his disposal. We want to know: what's the highest $k$ for which the sages can win? For a digraph $D$, we call that value $k$ the \textit{hat guessing number of the digraph $D$}. We write it as $\mu(D)$. It's surprisingly hard to determine. 
This thesis elaborates on some techniques toward doing so developed in Poland, Czechia, and Russia. 

\section{Beyond fun}

I've had something like the following conversation several times. The Skeptical Interlocutor might be a classmate, a worried relative, a philosopher of science, or a reader of grant applications.\footnote{I haven't actually applied for any grants, so that one's hypothetical.} 
\begin{quote}
    Me: ...and we call that $\mu(D)$, and it's much harder to determine than you'd expect. So I'm building on some techniques for it developed in Poland, Czechia, and Russia.

    Skeptical Interlocutor: Very nice. \textit{[tactfully]} Who cares?

    Me: I do! The problem is challenging; its solutions require ingenuity; I get to contribute to the stock of human knowledge; ``Euclid alone has looked on beauty bare'';––

    S.I.: Okay, who else?

    Me: Several highly trained intellectual professionals, as part of their publicly funded duty.

    S.I.: \dots

    Me: Mathematicians, I mean. A few dozen of them.

    S.I.: \dots

    Me: \dots

    S.I.: Right but \textit{why?}
\end{quote}
I figured I ought to answer this in writing. The short answer is: deterministic hat guessing is an intrinsically nice problem, older hat games proved surprisingly applicable well after their introduction, 
and even this one might have some practical use.\footnote{At least, that might answer ``why should they get paid to do it?'' They \textit{actually} do it for the same reasons I do.} To keep the chapter friendly and short, there are no technical details.\footnote{To the reader who would have loved technical details, I apologize. You should follow the citations, where they exist. If there are no citations, feel free to email me with a demand for explanation at imj.mcinnis@gmail.com.}

\subsection{Related games' applications}
A few earlier hat games, somewhat different from ours, spread far and wide in the nineties and aughts as pure riddles. Along the way, they picked up applications. For instance:


\subsubsection*{Finite dynamical systems} 
Dynamical systems are mathematical objects we use to model real-life things. A dynamical system has some entities in some states, and these entities change states over time according to their current state and the state of various other entities. This is pretty much what all science studies. A \textit{finite dynamical system} is just one with finitely many entities and finitely many possible states, so the evolution from state to state (while still deterministic) happens in discrete time-steps. These have modeled neural networks, social interactions, epigenetics, the immune system, bacteriophages, and phenomena in thermodynamics and statistical mechanics \cite{Gad15, TD90, GM90}. 

When analyzing dynamical systems in general, we want to find fixed points (states where nothing will change), cycles (patterns of eternal recurrence), and chaos (confusing, unpredictable behavior that never repeats) and understand how these things are intermixed in possibility space. Because these finite dynamical systems are \textit{finite}, they must eventually repeat, so there's no true chaos, but it's possible to measure how messy (``unstable'') they are.

For any finite dynamical system, one can draw a (directed) ``interaction graph'' with a node for each entity and an arrow $\overrightarrow{xy}$ whenever thing $x$ affects thing $y$. It turns out\footnote{This is linked to network coding, for reasons you maybe can guess, but which we won't explore.} that the stability and instability of a dynamical system is constrained by that graph. Specifically, by the performance of sages on that digraph!
\cite{Gad15,Gad17}

\subsubsection*{Auctions} 
You know what an auction is: that thing charity galas, art markets, and livestock fairs have in common. The government uses them to decide who gets construction contracts, radio frequencies, and commercial licenses. Auctions on power line congestion pricing allow energy producers to better plan their finances. Every time you watch a video, Google something, or visit a website, there's a near-instantaneous automatic auction for the right to show you an ad. Blockchain guys are really into auctions, but I'm not sure what they auction. 

Every auction has at least one of two goals: maximize revenue, and get 
items to the right people––i.e., those who want it most, those for whom the item does the most good. Google only cares about the former when it auctions off ad-space. The Federal Communications Commission isn't a major government revenue source; it's more concerned with having America's airwaves occupied by radio stations whose broadcasts are consistent, high-quality, and popular. But these goals align pretty well. Market mechanisms still ain't perfect, but someone willing to pay \$100,000 for the 95.1 band in Portsmouth, NH probably has bigger plans for that frequency than someone willing to pay \$10,000, plus more ability to execute. In general, the person who wants it most is also the customer who will make you the most money.\footnote{This isn't always true or good, but it's usually good enough and true enough.} So an auctioneer just wants to reliably, accurately, and quickly learn how much each bidder is truly willing to pay, and sell to the one who's willing to pay the most.   

You're probably most familiar with a ``first-price auction.'' Suppose we're auctioning off one psaltery.  
Every bidder writes down her bid on a secret slip of paper and gives it to the Auctioneer. The Auctioneer then identifies the highest bid. Say Gertrude bid \$200, and that was more than everyone else. Then Gertrude pays \$200 and gets the item. The item is actually worth \$250 to her, so Gertrude profited \$50. There's an incentive to bid less than your honest value of the item, so that you profit if you win.

That's a problem for the Auctioneer. He wishes he could sell the item to Gertrude for \$250 (or \$249.99). He can't drive that hard of a bargain. But suppose Hildegard reveals to the auctioneer that actually, she values the item at \$225, but underbid just like Gertrude. Then the Auctioneer can credibly sell it at \$225––either to Hildegard (who's indifferent between paying and not) or to Gertrude (who, though the Auctioneer doesn't know for sure, would rather buy). Or maybe someone else values it even \textit{more} highly, but lowballed it more than Gertrude did! It's a mess that can gradually be sorted out by having people make and adjust their bids publicly, moving the price up in increments. 
That solution is clunky and grainy. The slicker way to capture the same idea is the ``second-price auction.'' This has the same procedure––everyone makes a single bid, and the highest bidder gets the item––except she doesn't pay her own bid; she pays the \textit{second-highest} bid. It isn't too hard to show that the best bidding strategy in a second-price auction is simply to write down whatever your true valuation of the item is, because the price you pay on winning depends only on other people's bids, and if you lose, you didn't want it as much as the winner did. It takes a little more doing, but you can also show that this makes just as much money for the Auctioneer as a first-price auction. 

Matters get much more complicated with multiple items. Suppose you're bidding on a bunch of different items of furniture. The mid-century modern couch and armchair are worth \$400 to you individually, but they're worth \$1000 to you together. The Louis XVI armchair is worth  \$600 to you, but only \$200 if you end up getting the mid-century modern armchair. And so on. Many bidders, many separate auctions. We'd like to design auction mechanisms that incentivize bidders to truthfully disclose their valuations. The genius of the second-price auction is that your price depends not on your bid but only on other people's bids––which depend on their valuation of the item. How can we create such a rule for more complicated auctions? It's hard.

In hat guessing, you're trying to construct a function that gives each sage her guess and depends only on the other sages' hat colors, which are not subject to choice. In auction mechanism design, you're trying to construct a function that gives each bidder her price and depends only on the other sages' bids (which, once the mechanism is well-designed, truthfully refelects their valuations, which are given by nature). That's the analogy.  

Until 2005, the best known mechanisms used randomness, which might seem unnecessary, fraught, and prone to manipulation. Then Aggarwal et al. \cite{AFGHIS05} used this analogy to show how to convert any random auction into a deterministic one using hat games, while getting nearly the same revenue. This is extended in \cite{BNW15}.

\subsubsection*{Coding theory}
The defining activity of modern life is probably the storage and transmission of bitstrings (i.e., zeroes and ones, the language of computers). However, there's always a risk that interference or error might flip a 0 to a 1 or vice versa, or delete a digit entirely, turning 0101 into  1101, 0100, 011, or 010 (among other possibilities). 

So we want to design ``codes'' that are resistant to a certain frequency of errors––a way of transmitting with protective redundancy. For instance, you could send every bit three times. Then, if you get a string that reads  00011101111 or 00011000111, you can be pretty confident that the intended message was 000111000111, which represents 0101. 

Different kinds of errors predominate in different contexts––the challenges of radio transmission differ from the challenges of hard drive storage. Depending on what kinds of errors you're protecting against, you'll want codes with different properties. It's an active area of research; the goal is to optimally balance resilience with concision. (You could protect yourself against errors almost surely by writing every bit twenty times in a row, but that's inefficient.) 

Hat guessing is also about how to process information in the face of uncertainty so as to ensure the faithful capturing of some truth. The specific connections are too various and technical to go into here, but facts about coding have been used to solve hat guessing problems and vice versa
\cite{Gad15, Ber01, LS02, CHLL97, EMV03, Uem16, Uem18, JJG19, Alo08}.

\subsubsection*{Genetic programming}
Evolution has produced amazing things. Could it produce amazing software? That's the premise of genetic programming, a field that studies the automatic creation of algorithms via reproduction with heritable variation and selection for performance. In 2002, researchers attacked a hat problem with genetic programming and found it difficult \cite{BGK02}. 
It spurred them to develop impressive techniques that still fell short; the genetic algorithms failed to find the globally optimal strategies, instead getting trapped in local optima. (This is also a concern for machine learning.) Finding the problem ``attractive to genetic programming for several reasons,'' they ``hope[d] to gain a deeper understanding of how and when genetic programming works well on difficult problems'' and ``[felt] that investigation of this problem will lead to a worthwhile examination of evolutionary computation systems'' and ``force a critical evaluation of [their] application.'' Sadly for hat enthusiasts everywhere, this hasn't come to pass.

\subsubsection*{Network coding}

Imagine a graph where the nodes are communicators––people, computers, physical processes, cell towers, whatever––and the edges are open lines of communication. (This can be a directed graph if some communication is one-way.) Each line of communication has a limited bandwidth. Pick two communicators $s,t$ somewhere in the graph. With what bandwidth can $s$ transmit information to $t$? How fast can a message get from $s$ to $t$? And how should it be done? As it turns out, this problem––which is critical for essentially all electronic communication––isn't as simple as ``pretend the information is a liquid and utilize all available pipes to their fullest extent.'' Instead, it's related to the coding theory we discussed earlier. (This also pertains to the computational power of electric circuits and the concept of graph entropy.) \cite{Rii07a, Rii07b, GR11,ACLY00}

\subsubsection*{Other}
The list above isn't quite complete. There are a few minor applications we haven't discussed, found in the survey \cite{Krz10}, which discusses (along with some of the above) cellular automata, linear programming, approximation of Boolean functions, and autoreducibility of random sequences. There are others that have occurred since, which we don't review. E.g., \cite{PWWZ18} relates hat guessing to the arrangment graph $A_{m,n}$, to Steiner systems, and to other designs. 

\subsection{Our game's applications}\label{subsec:applicationsofourgame}
Besides its relationship to the games in the applications above, our game has some uses. (``Uses.'')

\subsubsection*{100\% soothsaying}
I haven't found many papers that study our deterministic games per se instead of their close relatives. However, one thread of mathematical research (summarized in the monograph \cite{HT13}) concerns infinitary generalizations of the deterministic hat problem and uses it to study the prediction of arbitrary functions, finding astonishing success. For instance, to quote the abstract, 

\begin{quote}
    there is a method of predicting the value $f(a)$ of a function $f$ mapping the reals to the reals, based only on knowledge of $f$’s values on the interval $(a-1,a)$ and for every such function the prediction is incorrect only on a countable set that is nowhere dense.
\end{quote}

We can provocatively rephrase this as, ``there's a method of predicting the future with 100\% accuracy, even in a non-causal universe, so long as you have nonzero short-term memory.''\footnote{As you might imagine, this requires the axiom of choice.} 

\subsubsection*{Philosophy and wild speculation} 
Earlier drafts contained background for all the topics of speculation below, but the applications are sufficiently improbable that we decided full explanations were bloat. 

Examining our hat guessing strategies, one could extract the following general (somewhat perverse) moral: to maximize the chance that \textit{someone} finds a truth, you should assume that everyone you're famililar with is wrong, even about things you're not qualified to comment on. It's hard to imagine ``real-world'' or ``epistemological'' applications of hat guessing––like, one in which each sage corresponds to some reasoning agent. That's because the hat problem is: a bunch of people, each trying to answer some question they are uniquely unqualified for, on the basis of the true answer to other folks' questions, while all are helpless to communicate. Such situations are rare.\footnote{With the plausible exception of romantic love.} Perhaps there are other philosophical morals. Perhaps there's some application in blockchain or other multi-party cryptography. 

Having once seen (in lecture notes for a quantum computing class) the famous CHSH game introduced as a modification of Game \ref{smallrainbow}, I wonder whether hat guessing games could give us ever-more egregious Bell inequality violations. If hat guessing strategies are paradigmatic for contriving negative correlations among events with a given dependence graph, this could be useful for risk management broadly and portfolio theory  specifically (devising trading/investing strategies for financial assets related on a network in order to hedge). \cite{GKRT06} was hopeful that the hat guessing problem might have some things to say about gene epistasis; quantitative geneticist Julien Ayroles shared that intuition. Computer scientist Tanya Berger-Wolf sees affinities between hat guessing and faulty chip testing. 

\subsection{Aus liebe zur kunst}
Finally, we come to the \textit{actual} reason, the reason that motivates me: it's just a nice problem. 

\subsubsection*{Which matters}
Even if the problem is nice, why should it be studied? This is a pretty common thing for non-mathematicians to wonder. You might be surprised to learn that there's some division on this topic. Must mathematics ultimately have utility––an application in science, which in turn is applied to technology, which then betters the lot of the species? The obvious practical answer is ``of course it must, if the taxpayer's supporting it!'' There are some eloquent and well-known defenses of pure math.\footnote{The classic is G.H. Hardy's \textit{A Mathematician's Apology}, but I would rather direct the reader to the following Edna St. Vincent Millay sonnet \cite{Mil23}: 
\begin{quote}
    Euclid alone has looked on Beauty bare.\\
    Let all who prate of Beauty hold their peace,\\
    And lay them prone upon the earth and cease\\
    To ponder on themselves, the while they stare\\
    At nothing, intricately drawn nowhere\\
    In shapes of shifting lineage; let geese\\
    Gabble and hiss, but heroes seek release\\
    From dusty bondage into luminous air.\\
    O blinding hour, O holy, terrible day,\\
    When first the shaft into his vision shone\\
    Of light anatomized! Euclid alone\\
    Has looked on Beauty bare. Fortunate they\\
    Who, though once only and then but far away,\\
    Have heard her massive sandal set on stone.
\end{quote}} I could try to write another, but instead, I'd like to quote contemporary mathematicians on the subject. There's a famously abstract mathematical discipline called category theory and a global listserv devoted to announcing conferences, job openings, festschrifts, and the like. Every few months there's a long reply-all conversation about some nuanced bit of categorical history, language, or philosophy. Arguments are collegial and heartening to watch. One that began in early January (over the nomenclature of ``point-free topology'') eventually touched on usefulness. 

Patrik Eklund wrote: 
\begin{quote}
    Some parts of theoretical mathematics is about seductive tricks, and
    some mathematicians fall for it. Potential practicality of even ``deepest
    theoretical theory'' keeps feet on the ground, even if practicality is
    not realizable or desirable. But my view is that we must keep
    ``real-world applicability'' at least as a ``general burden'' in the sense
    that all science must useful, in one way or another. Science should
    never be just ``aus liebe zur Kunst.''
\end{quote}
Robert Dawson replied: 
\begin{quote}
    You should, of course, follow your own ethical guidance, but I
respectfully disagree that ``we'' must do so.

Most of us agree that a painting or a novel can be beautiful without
having an improving message, that fine wine is worthwhile even if its
health benefits are dubious, etc.  Why should science be required to
clear a utilitarian bar that other fields of human endeavour do not?
\end{quote}

The sentiment is common, but folks rarely write their aesthetic judgment down. It seems to be considered gauche. I'll be a little gauche. 

\subsubsection*{The problem's niceness}
Most folks––mathematical or not––to whom I've told the original riddle find it and its solution charming. I feel similarly about many hat games. It's satisfying to reach into such a huge and misshapen set of possibilities to pull out a simple yet improbable tool. When reasoning about sages in these general games, everything seems to depend on everything else in thorny ways that are impossible to systematize, but mathematicians have, over the course of $\approx 20$ papers, made remarkable progress––yet it seems we still know so little. Our intuition is especially bad––most conjectures about this game fail, often spectacularly.\footnote{The most dramatic example comes from \cite{BDFGM21}, who conjectured that every planar graph has $\mu(G)\leq 4$. Later, \cite{LK23} constructed a planar graph with $\mu(G)\geq 22$.} If you generalize the game to infinite sages and infinite colors, you can derive nearly-metaphysical results.
The problem has demanded that folks draw not only on pure combinatorics, but also discrete geometry, graph theory, probability, coding theory, tensors, and algebra––along with original tools like constructors and admissible paths, not to mention the ad hoc lines of reasoning that show up everywhere––the problem resists systemization.\footnote{Many of my theorems are attempts to introduce a little systematization––judge for yourself how well this works.} (If you are aesthetically inclined, I recommend the proofs of Theorem 3.6 from \cite{LK23}, Theorem 11 from \cite{BDO21}, and Proposition 4.3 from \cite{ABST20}.)

This infatuation with the deterministic hat guessing problem is perhaps the result of immersion and personal bias. I hope this thesis vindicates my fondness. 

\chapter{Introduction}
We define our basic objects of interest and give their history. Then we outline the thesis, its approach, and its highlights.

\section{The games and parameters}\label{DefiningGamesAndParameters}
Graphs can be directed or undirected, but they must be loopless and simple.\footnote{A loop makes the problem trivial, and parallel edges add nothing. Undirected graphs can simply be considered as directed graphs where $\overrightarrow{uv} \in E(G) \implies \overrightarrow{vu} \in E(G)$.
}

A \textit{Czech game}\footnote{This version (although specifically undirected) is the one introduced by Václav Blažej, Pavel Dvořák, and Micha\l Opler \cite{BDO21}. All are at universities in Prague.} is a triple $\mathcal{G}\equiv (D,g,h)$, where $D$ is a (di)graph and $g,h$ are functions $V(D)\rightarrow \mathbb{N}$. We call $D$ the \textit{visibility (di)graph}, $g$ the \textit{guessness} function and $h$ the \textit{hatness} function. These may also be regarded as $\mathbb{N}$-valued vectors over $V(D)$. We use $z\leq z'$ (for any functions/vectors over $V(D)$) to mean that $z(v)\leq z'(v)$ for every $v\in V(D)$. If $g$ or $h$ is constant, we may highlight that by writing $\star g$ or $\star h$, following \cite{LK22}. 

We'll refer to $\mathcal{G}$ as a \textit{game played on $D$}.
Our image is of ``sages'' standing at each vertex of the visibility graph.  
An Adversary will place upon the head of sage $v$ a hat of color $c(v) \in V_v$. $V_v$ denotes the set of possible colors for $v$; we may identify it with $[h(v)]\equiv \{0,1,...,h(v)-1\}$. The function $c$ is chosen by the Adversary; we call it a \textit{coloring}. A restriction $c_S$ of a coloring $c$ to a subset $S\subseteq V(D)$ is called a \textit{partial coloring}. Sage $v$ can see the hat color of sage $u$ if and only if there is an arc from $v$ to $u$. Consider for sage $v$ a function $f_v: \left(\prod_{\overrightarrow{vu}\in E(D) } V_u\right) \rightarrow \{M\subseteq V_v \; \text{s.t.} \; |M|=g(v)\}$, i.e., a function that considers what colors she sees on whom and outputs $g(v)$ guesses for the colors of her own hat.
We call it a \textit{plan for/of/on $v$}. A collection consisting of a plan $f_v$ for each $v\in S\subseteq V(D)$ is called a \textit{plan for/of/on $S$} and denoted $f_S$. A plan $f_{V(D)}$ for all vertices is called a \textit{strategy (for $\mathcal{G}$)} and denoted simply $f$. 

If $c(v)\in f_v(c(N^+(v))),$ we say that $v$ \textit{guesses right} or \textit{wins}. If no $v\in V(D)$ guesses right, $c$ is a \textit{disprover for $f$}, and $f$ \textit{loses} $\mathcal{G}$. We say that $f$ \textit{wins} $\mathcal{G}$  if it has no disprovers, and $\mathcal{G}$ is \textit{winnable}.\footnote{The logical formula for ``$\mathcal{G}$ is winnable'' is $\exists f \forall c \exists v (c(v)\in f_v(c(N^+(v))))$.} If $\mathcal{G}$ has no winning strategies, it is \textit{unwinnable}. The winnableness-or-not of a game is referred to as its \textit{outcome}.
The central question is, which games are winnable, and how?

We generally assume that $h>g>0$ (and so $h>1$)––if some $v$ has $g(v)=0$, it can be deleted without changing the outcome, and if it has $h(v)=g(v)>0$, all other vertices can be deleted without changing the outcome. In other words, we require $g/h \in (0,1)$. We'll sometimes be concerned just with the $g$-to-$h$ ratio, so for convenience (following \cite{BDO21}) we define $r(v)=g(v)/h(v)$ and call that the \textit{ratio of $v$.}
We're especially interested in certain special Czech games.  
If $g=1$, we call it a \textit{Latvian game}\footnote{This version was introduced explicitly and studied by Konstantin Kokhas and Aleksei Latyshev (at universities in St. Petersburg) and secondarily Vadim Retinskiy (at a university in Moscow). Ordinarily I would call these ``Russian games,'' but there's a war on; it seems like poor taste. I have the polyglossic Lidia Trippicione to thank for pointing out that ``Latyshev'' translates to ``Latvian.'' She wants me to remark that Latvia is not Slavic.} and write such games as pairs $(D, h)$.  If $g$ and $h$ are constant, we call it a \textit{Polish game}\footnote{It was introduced by Farnik \cite{Far16}.} and write it as $(D,\star g,\star h)$. We define the \textit{hat s-guessing} number of a digraph $D$ as the greatest integer $k$ such that the Polish game $(D,\star s,\star k)$ is winnable and denote it $\mu_s(D)$. The \textit{fractional hat guessing number}\footnote{First studied by \cite{BDO21}.} equals $\sup \frac{\mu_s(D)}{s}$; we denote it $\hat{\mu}(D)$. 

If a game is both Latvian and Polish, it is \textit{Classic.}\footnote{It was the first considered \cite{BHKL09}. The reason I need to coin all these names, by the way, is because every author simply calls the version they study ``the hat guessing game'' or ``the \sc{Hats}\rm game'' or something like that.} The \textit{hat guessing number} of a digraph equals $\mu_1(D)$, but it's of the most interest, so we denote it simply as $\mu(D)$. It's the greatest integer $k$ such that the Classic game $(D,\star k)$ is winnable. It is clear (or demonstrable from Proposition \ref{monotone}) that $\mu\leq s^{-1}\mu_s\leq \hat{\mu}$.

Rather than ``Latvian game on a cycle $C_k$'' or ``a Polish game on an undirected graph,'' or something like that, we'll often say things like ``a Latvian cycle'' or ``an undirected Polish game.''  

We'll also consider a new continuous variant called the Ratio game, but not till Chapter 6, so we defer the definition. 

\section{A brief history of deterministic hat games}
Nobody has thoroughly surveyed or historicized these topics. Perhaps it's premature, but I do so in hopes of providing a useful reference and motivational context. This thesis mostly studies Latvian and Czech games, but others have used those to get results about Polish and Classic games, i.e. about $\mu_s$ and $\mu$, which feature heavily in this survey (and lend themselves more easily to pithy summary). We include all known results on Czech, Polish, Latvian, and Classic games, along with all results on $\mu$, $\mu_s$, and $\mu$, except for a few very minor propositions. This mostly ignores strategy-restricted variants like ``bipolar'' or ``linear'' hat guessing number, as in \cite{BDFGM21,ABST20} et cetera. Techniques are only sketched. We usually omit the rediscovery of known results. For unfamiliar nomenclature, see (often) Section \ref{definitionsection} or (always) the papers cited. Most inequalities are not known to be tight.

Onward to hat guessing history. Although one cannot rule out the existence of primitive hat guessing among medieval or antique peoples, the oldest distant relative I can find appears in a Martin Gardner article (reprinted in \cite{Gar61}), which frankly bears a closer relation to the old ``blue-eyed islanders'' problem. A hat game that allowed passing and penalized wrong guesses (originating in \cite{Ebe98}) circulated some time around the millennium, growing popular enough for write-ups not only in the Mathematical Intelligencer \cite{Buh02},
but in abcNews \cite{Pou01}, Die Zeit \cite{Blu01}, and the New York Times \cite{Rob01}.
Our version is more closely prefigured in Peter Winkler's ``Games People Don't Play'' \cite{Win01}. \cite{GKRT06} presented a version closer still (the Classic game $(K_n,n)$) as a fun variant. 

\subsection{Early work} 
The restricted-visibility innovation was finally introduced in \cite{BHKL09}, which defined our Classic games. Its Theorem 7 proved that $\mu(K_{n,n})\in \Omega(\log \log n)$. The technique was to construct an algorithm by having a few vertices ``being looked at'' and an overwhelming number of ``spectator vertices,'' such that the spectators get it wrong rarely enough that the performers can patch those holes. Its Lemma 8 combined with its Example 1 shows that $\mu(T)\leq 2$ for any tree $T$. (In fact, they proved that a Latvian tree with hatnesses $(2,3,3,...)$ is unwinnable, though not in so many words.)  

The next entry is a monograph summarizing several papers \cite{HT13}. It's a marvelous compendium for infinite hat problems, but it spends a short chapter on the finite. It includes some basic theorems: $\mu(D)\geq 2$ if and only if $D$ has a directed cycle; the Classic game with $n$ sages and $n$ colors is winnable if and only if the graph is complete; $\mu(K_{n,m})\leq n+1$. They soon adjourn for the infinite case, writing: 

\begin{quote}
    There is an old saying, variously attributed to everyone from the French Minister Charles Alexandre de Calonne (1734–1802) to the singer Billie Holiday (1915– 1959), that goes roughly as follows: “The difficult is done at once; the impossible takes a little longer.” More to the point, Stanislaw Ulam (1909–1984) provided the adaptation that says, “The infinite we shall do right away; the finite may take a little longer.” With this in mind, we leave the finite.
\end{quote}

\cite{GG15} was the first to use specialized machinery. It defined ``distinguishable sets,'' which are based on  Hamming distance and presaged the Hamming ball phrasings of \cite{ABST20}. They used a ``blow-up lemma'' that can be rephrased as saying: if you take the strong product of $K_r$ with a digraph $D$, the resultant graph has hat guessing number $r\mu(D)$. The main result (besides those since superseded) is an asymptotic showing that $\mu(D)$ is rather unconstrained by $\omega(D)$, even for smallish graphs. 

Gadouleau \cite{Gad15} studied the game further in order to prove some properties of finite dynamical systems. His Theorem 3 proved that $\mu(G)\leq 1+\sum_{i=1}^{\tau(D)}i^i$, where $\tau(D)$ is the minimum size of a vertex set $S$ set such that $D\backslash S$ is acyclic.

\subsection{Slavic techniques}\label{Slavictechniques}
The PhD thesis \cite{Far16}, most of which was published as \cite{BDFGM21}, collects results on undirected graphs. They show probabilistically (see Chapter \ref{chap:LLL}) that $\mu(G)<e\Delta(G)$. Using convex functions, they bound $\mu$ in terms of $\chi(G)$ (the graph chromatic number) and $|V(G)|$: $\mu(G)\leq \left\lfloor (1-(1-\chi^{-1})^{\chi/|V|})^{-1}\right \rfloor$. This yields a (rather weak) set of probabilistic bounds, which are the earliest we know of––this paper is also the first to use combinatorial prisms\footnote{Upsettingly, they call them ``cubes.'' We adopt ``combinatorial prisms'' from \cite{HIP22}, since that does better justice to the discreteness, disconnectedness, and unequal dimensions of these objects.} (see Chapter 5) or examine $\mu_s(G)$. It proves:  $\mu_s(T)\leq s(s+1)$ for all trees; $\mu_s(S)\leq s(s+1)^2$ for all $S$ that are a 1-subdivision of another graph; every graph $G$ of genus $\gamma$ and sufficiently large girth (depending on $\gamma$) has $\mu_s(G)\leq (s^2+s)(s^2+s+1)$. The latter two are shown by means of the following theorem: for $G$ connected, $(A,B)$ a partition of $V(G)$, $d=\max_{v\in B} |N(v)\cap A|$, and $s_1=s(\mu(G[A])+1)^{d}$, we have $\mu_s(G)\leq \mu_{s_1}(G[B])$. (They also study some variant parameters of lesser interest.)

Still one of the most impressive papers is \cite{Szc17}, which merely set out to determine $\mu(C_k)$. He introduces admissible paths and the Warsaw graph $3*G$  (see Section \ref{AdmissibleGraphs}) along with a raft of ad hoc combinatorial constructions, through which he plods to show that $\mu(C_k)=3$ if and only if $3\mid k$ or $k=4$, otherwise equalling $2$. (He also shows that no Latvian cycle with hatnesses $(4,3,3,...,3)$ is winnable.)

Kokhas and Latshev have done the most of anyone––mostly on undirected Latvian or Czech games, though their latest publication \cite{LK23} concedes the naturality and usefulness of proper digraphs. 
In their first work, \cite{KL19}, they describe every graph for which $\mu(G)\leq 2$: graphs that are trees or whose unique cycle has length neither equal to $4$ nor divisible by 3. They use Szczechla's $3*G$ along with $L(3*G)$ to rephrase strategies as matrices, which they then test using brute-force matrix, tensor, Motzkin path, and SAT computation. It also explores hints (see Section \ref{sec:HAH}), one of which is a hint in our sense, and two of which prove weak statements in some non-classic variants. 

Subsequently, they formally inaugurate the Latvian game, which made hat guessing a suitably local phenomenon for the constructors technique (see Chapter \ref{ch:constructors}) to be devised.\footnote{A ``constructor'' is a theorem that alters a game or amalgamates games while preserving the outcome. The proto-constructor is the ``blow-up lemma'' mentioned above.} 
This was in two Russian-language papers, which they translated and published as \cite{KL21} and \cite{KLR21a}. These are combined as \cite{KLR21b}, ``Cliques and Constructors in \sc{Hats } \rm Game." 
They solve Latvian cliques––once constructively, and once by Hall's marriage theorem. They partially solve it for cliques-minus-an-edge.
They introduce the one-vertex product constructor, which proves that $h(v)=2^{\deg(v)}$ on trees is a winnable Latvian game. Their ``substitution'' constructor allows you to do the blow-up lemma one vertex at a time. There are a few ways you can attach vertices of hatness 2 and alter hatnesses to preserve winning or losing. You can also attach a short path with hatnesses 2 and 3 and alter hatnesses accordingly. Attaching a leaf of hatness $\geq 3$ does nothing. (All the preceding appear in Chapter \ref{ch:constructors}, where they get used and generalized.) Unwinnable games can be glued on a vertex of hatness 2 to yield another unwinnable game. There are more ways you can fasten games together, but they're hard to say concisely.
The next section is a long digression on ``Blind Chess,'' with subvariants called ``Rook check,'' ``Queen check,'' etc. Rook check constitutes a complete analysis of the Latvian game on $C_4$. 
Their constructors, combined with Szczechla's work, nearly\footnote{Contra the abstract's claims, they don't fully resolve it.} complete the analysis of the Latvian game on a cycle for $h(v)\leq 4$; see Section \ref{sec:MCP} for details and how we finish the job. 

They put the constructors to further work in \cite{LK21}. Their product constructor easily proves the windmill theorem of \cite{HIP22}. From products and contes, they have a few lemmata about constructing graphs to disregard a couple low-hatness vertices. Using these and more constructors, they devise an outerplanar graph with $\mu\geq 6$ and a planar graph with $\mu \geq 14$, strongly disproving conjectures from \cite{GIP20,BDFGM21}.

Constructors were picked up by \cite{BDO21}, who introduced the Czech game and solved it for cliques. Their only constructor, the ``clique-join'' of games, generalizes the product, substitution, and attaching-a-hatness-2-vertex-to-an-edge constructors.  Their other main tool, which we discuss in Chapter 6, is Shearer's lemma \cite{She85} as presented by Scott and Sokal \cite{SS06}, which motivates the (implicit) definition of Ratio games and the (explicit) definition of $\hat{\mu}(G)$.  Combining these tools with pre-existing algorithmic knowledge, they essentially solve the Ratio game for chordal graphs, showing that there exists a polynomial-time algorithm taking a Ratio game on a chordal graph as its input that outputs (if the Ratio game is winnable) a Czech game no easier than the Ratio game along with a polynomially-describable winning strategy for an associated Czech game or (otherwise) a proof that the game is unwinnable.\footnote{This is not an \textit{entirely} faithful description of what it does; I am substantially rephrasing, and certain edge cases are problematic. For discussion, see Section \ref{subsec:RatioGames}.}  They present an algorithm that takes a chordal graph $G$ with $n$ vertices and returns $\hat{\mu}(G)$ up to an additive error of $1/2^k$ in time $O(k\cdot poly(n))$. These tools also show that $\hat{\mu}(G)\in \Omega(\Delta/\log \Delta) \cap O(\Delta)$. They also point out that $\lim_{n\rightarrow \infty} \hat{\mu}(P_n)=\lim_{n\rightarrow \infty} \hat{\mu}(C_n)=4$ and that $\hat{\mu}(K_n)=n$. 

Latyshev and Kokhas picked up the clique-join constructor and independence polynomial methods from an earlier print of \cite{BDO21} and used them in \cite{LK22}. They showed in \cite{KLR21b} that the single-point product of two critical Latvian games (i.e., games that cannot be made any harder without becoming unwinnable) is not necessarily critical. Here, though, they show that \cite{BDO21}'s clique-join preserved maximality, which is a condition akin to criticality  most natural for Ratio games but also useful for Czech games and Latvian games.\footnote{For ratio games, it may in fact \textit{be} criticality. See Question \ref{PolynomialWorksForAllRatioGames}.}  Using these tools, they show that diameter and $\mu$ are independent. There are vapid instances of this (one can always attach additional leaves), but Latyshev and Kokhas find edge-critical graphs (i.e., any proper subgraph has lower $\mu$) with any odd diameter $d$ and even $\mu$; they suspect the parity restrictions are non-essential. Diameter is independent from $\hat{\mu}$ as well, they show. Most impressively, they contrive counterexamples to a conjecture of \cite{HIP22} that $\mu\leq \Delta +1$, showing that for any positive integer $N$ there exists a $G$ with $\mu(G)=\Delta(G)+N$ and that there exists a sequence of graphs $G_n$ with $\Delta(G_n)\rightarrow \infty$ and $\mu(G_n)=\frac{4}{3}\Delta(G_n)$. 

In \cite{LK23}, Kokhas and Latyshev improve the constructors for clique-joins and substitutions, adding ``reduction'' to control the increase of guessness functions and to derive Latvian facts from Czech ones. They give constructors for the deletion or addition of single arcs. Doing so, they find $\mu_s$ for paths and ``petunias,'' which are graphs of which each 2-connected block is a path plus a universal vertex. Most impressively, they find a planar graph with $\mu(G)\geq 22$. (It has 546 vertices by my count.) Analogous constructions yield planar graphs with $\mu_s(G)\geq 4(s + 1)(s + 2) - 2$ for any $s$. Along the way, they find constructions for outerplanar graphs with $\mu_s(G)\geq 4s(s + 1) - 2$. 

\subsection{The variegated recent boom}
Since about 2020, papers on deterministic hat guessing have come frequently, appearing in multiple versions in journals and on the ArXiV. Accordingly, it's difficult to chronologize recent papers, and even harder to trace whose work was influenced by whose and when. Individual papers use many techniques, so a thematic breakdown is also difficult. 

The popular \cite{HIP22} uses combinatorial prisms and substantial casework to prove that $\mu(K_{3,3})=3$. They determine the hat guessing number of some windmills and books, showing that books with enough pages attain the bound $\mu(G) \leq 1+\sum_{i=1}^{\tau(G)}i^i$. It emerged from a summer project whose final report \cite{GIP20} contained all of the above, along with a laborious determination of $\mu$ for the three 5-vertex graphs that remained unclassified by previous theorems. 

He and Li construct in \cite{HL20} an infinite family of graphs $G_d(N)$: take the complete rooted $d$-ary tree with depth $N$ and connect every vertex to all its ancestors. They show $\mu(G_d(N))=s_d-1$ for $N$ sufficiently large in terms of $d$, where $\{s_i\}$ is Sylvester's sequence. Since $G_d(N)$ is $d$-degenerate, this gives $d$-degenerate graphs with $\mu$ doubly exponential in $d$. They obtain upper bounds that depend on the degeneracy ordering, but not the degeneracy itself. 

Rooted trees of fixed height and arity appear also in  \cite{Bra21}. Bradshaw studies them in order to show that for any tree $T$, there exists a number $N(T)$ such that any graph $G$ not containing $T$ as a subgraph has $\mu(G)\leq N(T)$. He also shows that $\mu(G)<(\frac{64}{25})^{2^{\lfloor c^2/2 \rfloor-1}}+\frac{1}{2}$, where $c$ is the circumference of $G$, i.e. the size of its largest cycle. Throughout, he uses rephrased lemmata from \cite{KLR21b} and \cite{BDFGM21} (with pretty proofs), along with block decompositions, the notion of treedepth, and the parameter $\mu_2$ on subgraphs. 

In \cite{Bra22}, he bounds $\mu_s$ for outerplanar graphs and for a large class of planar graphs called ``layered planar graphs,'' which consist of concentric outerplanar graphs with non-crossing edges added. For the former, $\mu<2^{125000}$. (More generally, $\mu_s\leq (3(s+1)^6 +3(s+1)^3 +3)^{3((s+1)^6+(s+1)^3+1)^2}$.)  For the latter, $\mu\leq 2^{2^{2^{2^{2^{149}}}}}$. (More generally, $\log_2 \log_2 \log_2 \log_2 \mu_s < 2^{149}s^{35}$.) He leverages the partitioning theorem on $\mu_s$ from \cite{BDFGM21}, generalizing it to more than two parts. An additional tool depends on a ``Turán-type edge density problem.'' There's a fair bit of casework, especially for the layered planar graph. Finally, he proves that if $\mu_s$ has a constant bound for planar graphs, then so does $\mu_s$ for graphs of genus $\leq C$ for any fixed $C$. 

Using randomly chosen plans (and adjustments to cover their deficiencies), along with Hamming balls and lemmata about ``$t$-saturated matrices,'' \cite{ABST20} estimate the hat guessing number of complete $r$-partite graphs: $\mu(K_{n,n,...,n})\in \Omega(n^{\frac{r-1}{r}-o(1)})$. For bipartite graphs, this answered the popular question of whether $\mu(K_{n,n})$ grows at least polynomially in $n$. They show that $\mu(\overrightarrow{C}_{n,n,...n})\in \Omega(n^{\frac{1}{r}-o(1)})$, where $\overrightarrow{C}_{n,n,...n}$ is the complete $r$-partite directed cycle. Furthermore, $\mu(K_{n,m})=n+1$ for sufficiently large $m$. They also study linear hat guessing numbers and a connection between the hat guessing number of a graph and the hat guessing number of subgraphs on high-degree vertices. 

The greatest contribution of \cite{AC22} is to show that with high probability, $\mathcal{G}(n,1/2)$ has $\mu \geq n^{1-o(1)}$. They do so by showing that smaller book graphs than thought have large hat guessing number, and proving that $\mathcal{G}(n,1/2)$ has (with high probability) a reasonably-sized book graph. (The proof, apparently, does not depend on $1/2$ but could work for any $p$.) They also study linear hat guessing number and give a simple construction for a planar graph  $G$ with $2+2\cdot 66^{144}$ vertices such that $\mu(G)\geq 12$, as shown by giving an explicit plan with computations in $\mathbb{Z}_{12}$. They show $\mu(G)<13$ by showing that a particular Latvian game is unwinnable. 
\cite{WC23} apes \cite{AC22} to show that $\mathcal{G}(n,p)$ has with high probability hat guessing number $\geq n^{1-o(1)}$.

The question of whether $\mu$ can be bounded by degeneracy motivates \cite{KMS21}, but they don't solve it, instead introducing a notion of ``strong $d$-degeneracy.''  They use the Czech game to obtain results about the Classic game. For a strongly $d$-degenerate graph $G$, they prove $\mu(G)\leq (2d)^d$ by way of a harder Czech game and a constructor on vertex removal.\footnote{I closely read only a handful of papers––mostly Szczechla, Kokhas, and Latyshev––before embarking on my own research. It's good to see how often my own ideas were already had by others. They must be natural.} By similar means, they show that $\mu\leq 40$ for outerplanar graphs, forcefully contradicting Bradshaw's guess that it could attain 1000. They give additional characterizations of strong $d$-degeneracy, especially for graphs containing no $K_{2,s}$ subgraph for given $s$. Finally, they demonstrate that there exists a function $b$ such that a graph $G$ with $n$ vertices and average degree $C$ has $\mu(G)\leq b(C)$ with high probability as $n\rightarrow \infty$.

\section{Outline of the thesis}
First some general notes. Omitted from this summary is the context, philosophy, and general chattiness that appears alongside definitions, claims, and proofs. The reader who hates such parts can skip them. 
We phrase our propositions generally, sometimes stating the useful special case as a corollary. Our work is, in loving imitation of \cite{KLR21b}, more a garden of lemmata than a hike up to a theoremic view. More woodshop, less furniture. Also, though dependencies are not always chronological (a proposition in Chapter 3 might reference one in Chapter 4, for instance), be assured that we don't have any cyclic dependencies. 
Finally, though hat guessing on graphs can be a rather visually intuitive topic, I include no pictures or diagrams, concurring with my advisor that it's finicky work with low payoff. We encourage you to doodle in the margins. 

Chapter 0 is, as it says, an ``Invitation'' to hat guessing. We consider the Classic game $(K_7,7)$ and explain how one might figure out a winning strategy, emphasizing a ``covering assumptions'' view of guessing. We explore generalizations and define $\mu(G)$. Then we showcase applications of other hat guessing games, speculate about applications of our own, and discuss why someone might care. Nothing later depends on Chapter 0.

Chapter 1 (which you are reading) defines our objects of investigation: games, plans, and colorings. It surveys all known results in deterministic hat guessing. It summarizes the rest of the thesis (as it is doing), explains our approach, and highlights our best results.

Chapter 2 gives further definitions and proves some easy results as foundation and warm-up. Notably, it solves Latvian and Polish directed cycles. 

Chapter 3 crystallizes the ``assume your neighbors are wrong'' notion of hat guessing into the ``Hats As Hints'' theorem, which massively elaborates a notion from \cite{KL19}. We take up ``admissible paths'' from \cite{Szc17}, generalizing them in a few directions. We use these two tools to complete the task of Latvian low-hatness cycle classification from \cite{KLR21b} (which we limn in Section \ref{sec:MCP}). Then we classify all other Latvian cycles using admissible paths. 

Chapter 4 is concerned with constructors (though some appear in other chapters). We show their utility by using a few from \cite{KLR21b} to simplify proofs of \cite{KL19} and categorize all Latvian trees with a concise leaf deletion algorithm. Then we majorly generalize several constructors of \cite{BDO21} and \cite{KLR21b}. Using Hats As Hints, we generalize all of \cite{KLR21b}'s theorems that pertain to attaching a hatness-2 vertex in some manner to involve arbitrary hatness, guessness, and manner. We devise a new constructor about attaching a hatness-3 vertex and use it to generalize \cite{KLR21b}'s ``attach a short path'' constructor.
We generalize the ``clique join'' constructor from \cite{BDO21} in three complicated directions.

Chapter 5 deals with the ``combinatorial prism perspective'' that arises naturally when extending plans to independent sets. We define and discuss various forms of covering to give conditions on how to extend plans.
Leveraging an original Alon lemma, we find a broad class of vertices in Czech games  
that can be deleted or replaced by arcs. This includes a full generalization of the leaf-deletion theorem from \cite{KLR21b}, which gives an alternate proof of the bound on $\mu_s$ for trees. It allows an even slicker algorithm classifying Latvian trees. Finally we translate  complete bipartite Czech games to certain covering/packing problems in the integer lattice. 

Chapter 6 concerns hat randomization, dependency digraphs, and variants of the Lovász Local Lemma. We define the ``Ratio game.'' This embodies certain limiting properties of Czech games, implies straightforward Czech results, and is the proper object of probabilistic study.
 The central fact is that the dual of a visibility graph is a dependency digraph for the events ``$v$ guesses right,'' each of which has probability $r(v)$. We give obvious applications via extant theorems on dependency graphs. We also reply to a question of \cite{SS06} by saying that Shearer's Lemma probably doesn't generalize interestingly to dependency \textit{di}graphs.

Chapter 7 collects open questions on diverse hat guessing topics––all those posed in the literature (that are still open) and many original ones. It also includes partial results, including outcome conditions for elementary Czech games, basic analytic facts about Ratio games, and some complexity considerations. 

\section{Our approach}\label{OurApproach}
You might wonder about the title, ``Slavic Techniques for Hat Guessing Algorithms.'' Let's dissect it, starting with ``Slavic.''
I had no intention of regional specificity; it just happened that some specific papers charmed me \cite{Szc17, KL19,KLR21b,BDO21, BDFGM21}.
The connection isn't meaningless––they seem to be in close dialogue with one another, the techniques largely mesh, and it's they who've most explicitly considered variable hatness and guessness.

I say ``techniques'' because we concern ourselves less with theorems than with perspectives, approaches, tools, and ideas. (We still have theorems.) The specific techniques are mostly these: considering an additional sage as a hint to other sages, admissible paths/graphs/assignments, constructors, combinatorial prisms, and randomization.  

Finally, why ``hat guessing algorithms'' instead of ``hat guessing games'' or ``the hat guessing number''? We wish to emphasize three things. 
First, we don't talk much about the parameters $\mu,\mu_s,\hat{\mu}$ nor all that much about particular games. As we said above, we're interested in processes.

Second, the algorithmic aspect of hat guessing has been under-recognized. 
 One wants to know which games are winnable; one wants to find $\mu(G)$. For any game or graph, this can be answered by iterating over the finitely many possible choices of strategy and coloring. At even moderate size, such brute force is infeasible. At any size, it's unsatisfying. We want \textit{efficient} recognition/calculation protocols.

(Here's one standard we could set. Consider a family of games $\mathcal{F}$. Latvian cycles, maybe, or Czech stars, or games having no vertex of hatness $>100$. Anything. Treat $\mathcal{F}$ as a language in your favorite ad hoc encoding. We say $\mathcal{F}$ is weakly $C$-solvable for complexity class $C$ if there's an algorithm in $C$ whose input is a game in $\mathcal{C}$ and whose output is (correctly, of course)``winnable'' or ``unwinnable.'' Say it's strongly $C$-solvable if there's an algorithm in $C$ whose input is a game in $\mathcal{F}$ and whose output is either a winning strategy or a proof that the game is unwinnable.\footnote{Modulo an appropriate encoding scheme, the following classes have been shown strongly $P$-solvable: games with $h\leq 2$ \cite{BHKL09}, undirected games with $h\leq 3$ \cite{KL19}, Classic unicyclic graphs \cite{KL19}, Classic games on $\leq 5$ vertices \cite{GIP20}, some sporadic games that show up in constructions [nearly every paper], Czech cliques \cite{BDO21}, and some trivialities like Classic games with $\star h>|V(G)|$. I am not aware of any sharp non-constructive results. In this thesis, we add to the list Polish and Latvian directed cycles, along with several Latvian classes including cycles, trees, and unicyclic graphs.})

To fully resolve a class of hat games is to exhibit the efficient construction of algorithms or to show that no such efficient solution is possible.  The former's been done a handful of times; the latter never has.

Third, almost all of our propositions yield polynomial-time techniques for answering questions of interest to the would-be algorithm designer. Admissible assignments/graphs/paths allow you to grow disprovers step by step, making search easier. ``Hats As Hints" formalizes the explicit behavior necessary for a vertex's plan to help its neighbors, and the discrete geometry perspective explains the behavior necessary for a plan to extend usefully to an independent set. Both may inspire array-based algorithms. Constructors allow you to polynomially transform winning strategies for games into winning strategies for other games, or winnable games into other winnable games. When based on deletion, they automatically lower the computational load.  There even exist recent constructive versions of the Lovász Local Lemma and its variants, which would help identify specific failures. (Unfortunately, many LLL-and-variant-related computations are decidedly not polynomial.)

Nevertheless, we mostly suppress explicitly algorithmic language or content in the thesis. Let these tools be taken up by better wielders.

\section{Highlights}
This is rather long for an undergraduate math thesis! The reader who wants to read a few remarkable facts and depart is well within her rights.\footnote{It is perhaps optimistic to imagine that anybody is reading this whose name isn't Noga Alon or Matija Buci\'c.} Sadly, it isn't well-suited to headline takeaways. Like \cite{KLR21b}, it consists mostly of ``techniques'' and their many associated propositions. (That's part of why it's so long.) Still, we give a few candidates here.
\subsubsection*{Constructors}
Most of our constructors are, due to their generality, cumbersome statements. (Still, we're pleased with Propositions \ref{prop:GeneralAttaching}, \ref{AttachHatness3}, \ref{StrongPath}, and \ref{CliqueAndOneMoreThingProduct}. Note further that most of these constructors extend in nice, obvious fashion to ratio games.) Here we present a few weaker, cleaned sentences.

\newtheorem*{Jank12}{Parts of Corollaries \ref{AddManyHatness2} and \ref{cor:AddNearlyFullVertex}}
\begin{Jank12}
    Let $\mathcal{G}$ be a winnable game. Attach a vertex $x$ universal to $S\subseteq V(D)$ with $h(x)=2$. For every $v\in S$, increase $h(v)$ by $g(v)$ or (if $S$ is a clique) by $h(v)$. The game is still winnable.
\end{Jank12}

\newtheorem*{Jank6}{Part of Theorem \ref{ReplaceVertexWithArcs}}
\begin{Jank6}
Let $v$ be a vertex in a winnable game with $N(v)$ independent and $r(v)^{-1}>\prod_{u\in N(v)} (g(u)+1)$. Add arcs as needed to make the function $N^+(\;)$ constant on $N(v)$; delete $v$. The game is still winnable. 
\end{Jank6}

\newtheorem*{Jank7}{Part of Theorem \ref{DeletableByPeerSets} and Corollary \ref{cor:deleteCzechleaves}}
\begin{Jank7}
Let $v$ be a vertex with each member of $N(v)$ belonging to a different component of $D\backslash v$ and $r(v)^{-1}>\prod_{u\in N(v)} (g(u)+1)$. Deleting $v$ doesn't change the outcome. In particular, if $v$ is a leaf with $v\sim u$ and $r(v)^{-1}>g(u)+1$, then $v$ is deletable. (Also, replacing $>$ with $\geq$ makes these statements false.)
\end{Jank7}

\newtheorem*{Jank8}{Proposition \ref{Removing2Leaf}}
\begin{Jank8}
    Let $\ell$ be a leaf in a Latvian game with neighbor $u$. Delete $\ell$, and if $\ell$ had hatness 2, set $h(u)\mapsto \left\lceil \frac{h(u)}{2}\right\rceil$. The outcome is unchanged. 
\end{Jank8}

\newtheorem*{Jank13}{Corollary \ref{cor:redundantvisiondelete}}
\begin{Jank13}
    In a Czech or Ratio game $\mathcal{G}$, if $N^+(S) \subseteq \bigcap_{u\in N^-(S)} N^+(u)$ for some $S\subseteq V(D)$, then $S$ can be deleted without changing the outcome. 
\end{Jank13}

\subsubsection*{Solving game classes}

Our most ``final'' results are these theorems efficiently characterizing winnable games. Before this thesis, no nontrivial family of (di)graphs had been wholly solved for Latvian or Polish games.\footnote{A digraph is trivial if, after eliminating vertices according to \ref{cor:redundantvisiondelete}, all strongly connected components are cliques.} We solve Polish and Latvian games for directed cycles. We solve Latvian games for cycles (proving the conjecture of \cite{KLR21b} and much more) and trees. Furthermore, by combining these theorems and Proposition \ref{CzechOnComplete} (not stated here) with Proposition \ref{Removing2Leaf}, we have an efficient solution for Latvian games on any graph that can be obtained from a clique, cycle, or directed cycle by adding leaves––in particular, for unicyclic graphs. 

\newtheorem*{Jank3}{Theorems \ref{LatvianDirectedCycles} and \ref{LatvianPolish}}
\begin{Jank3}
    A Latvian game on a directed cycle $\overrightarrow{C}_k$ is winnable if and only if $h(v)=2$ for all $v$. Also, $\mu_s(\overrightarrow{C}_k)/s=\hat{\mu}(\overrightarrow{C}_k)=2$. 
\end{Jank3}
(Recall that we require throughout that $h>g>0$.)
\newtheorem*{Jank1}{Theorems \ref{CategorizedTwoToFour} and \ref{CategorizedFiveAndUp}}
\begin{Jank1}
    Consider a single-guess game on a cycle, $(C_{k\geq 4},h)$. It is winnable if and only if one of the following holds:
    \begin{itemize}
        \item $h=3$ and $k$ is divisible by 3 or equal to $4$
        \item or $h\leq 4$ and the hatness sequence $(3,2,3)$ or $(2,3,3)$ appears
        \item or  the hatness sequence $(2,...,2)$ appears, where no intervening number is $>4$.
    \end{itemize}
\end{Jank1}

\newtheorem*{Jank2}{Theorem \ref{thm: FirstTreeCategorization}/\ref{TreeAlgo2}}
\begin{Jank2}
    A Latvian game on a tree $T$ is winnable if and only if $T$ has a subtree $T'$ with $h(v)\leq 2^{\deg_{T'}(v)}$ for all $v\in V(T')$. 
\end{Jank2}

Our discrete geometry perspective of Chapter 5 give packing/covering conditions for Czech games on the ever-popular complete bipartite graphs, of which we present two digestible special cases. If the first is easily computable, we have the first efficient solution for a nontrivial family of Czech games. If the second is easily computable, we have the longstanding dream of calculating $\mu(K_{n,m})$. If either is provably hard, we can prove the computational hardness of large game classes.

\newtheorem*{Jank4}{Proposition \ref{lmm:WinOnCzechStars}}
\begin{Jank4}
    Consider a Czech game $\mathcal{G}$ on a star with center $v$ and leaves $\ell_1,...,\ell_n$. It is winnable if and only if you can fit $h(v)$ combinatorial prisms, each $h(\ell_1)-g(\ell_1) \times ... \times h(\ell_n)-g(\ell_n)$, into a $h(\ell_1)\times ... \times h(\ell_n)$ combinatorial prism, such that any $g(v)+1$ of the smaller prisms have empty intersection.  
\end{Jank4}

\newtheorem*{Jank5}{Corollary \ref{cor:ClassicBipartite}}
\begin{Jank5}
    $\mu(K_{m,n})\geq k$ if and only if there exist $n$ partitions $\mathcal{P}_i$ of $[k]^m$ into $k$ parts such that, whenever we select one member $P_{i,q_i}$ from each $\mathcal{P}_i$, the set $\bigcup P_{i,q_i}$ contains a $k-1\times k-1 \times ... \times k-1$ combinatorial prism.
\end{Jank5}

In either case, these provide a very simple formalism (arrays!) in which to write a computer program for assaying complete bipartite games.

\subsubsection*{Dependency digraphs}
\newtheorem*{Jank9}{Theorem \ref{thm:mainfact}}
\begin{Jank9}
    For a fixed (Czech or Ratio) game $\mathcal{G}$ and strategy $f$, the probability $\mathbf{P}(R_v)$ that sage $v$ wins is $r(v)$. The dual of the visibility digraph $D$ is a dependency digraph for these events. 
\end{Jank9}
This justifies a world of hat-randomization approaches, including the directed Lovász Local Lemma and Shearer's Lemma (as in Propositions \ref{prop:application1} and \ref{prop:shearerforratio}), which specialize nicely. 
\newtheorem*{Jank14}{Propositions \ref{prop:application2} and \ref{prop:application3}}
\begin{Jank14}
    For a digraph $D$, 
    $\hat{\mu}(D)\leq e(\Delta^-+1)$. 
    For a graph $G$, $\hat{\mu}(G)\leq (\Delta-1)^{1-\Delta} \Delta^{\Delta}<e\Delta$.
\end{Jank14}
Scott and Sokal \cite{SS06} wonder whether Shearer's Lemma, which concerns the independence polynomial of a dependency graph, might generalize to digraphs. In Section \ref{sec:DirectedShearersLemma}, we explain why Corollary \ref{thm:nodirectedshearers} constitutes an answer of ``probably not.''
\newtheorem*{Jank10}{Corollary \ref{thm:nodirectedshearers}}
\begin{Jank10}
    Let $D$ be a digraph containing an odd directed cycle subgraph. There exist events $\{R_v\}$ having dependency digraph $D$ and probability vector $\mathbf{p}$ such that the acyclicity polynomial $Q_D$ is positive at $\mathbf{p1}_S$ for all $S\subseteq V(D)$ but $\mathbf{P}(\bigcap \neg R_v)=0$.
\end{Jank10}

\subsubsection*{Other}
Definition \ref{def:RatioGames}, of a ``Ratio game,'' is a highlight at least aesthetically. Theorem \ref{FinalCzechHintTheorem}, Corollary \ref{LatvianHAH}.1, Proposition \ref{prop:winnableifcomplementhasprism}, and Proposition \ref{Hamminghowtowin} give us good views on plan extension and strategy design. Noga Alon's proof of Lemma \ref{AlonCoveringLemma} is lovely, and some might appreciate the ``Warsaw graph'' of Section \ref{AdmissibleGraphs}. Czech games on $\overrightarrow{C}_3$ are an intriguing challenge (e.g. Proposition \ref{prop:directedtrianglesweird}), and Subsection \ref{AddingEdgesAndVertices} is fun.

\chapter{Starting simple}
\section{Notation, definitions, and apposite lemmata}\label{definitionsection}
Our topic requires no background beyond set theory, though it would be tough sledding for someone who's never seen (di)graphs before. We'll rehash standard graph-theory definitions for precision's sake, then define some useful concepts specific to our work and state their obvious properties. Further definitions appear in later chapters where they're needed.

The $\equiv$ symbol indicates definitional equality; $x:=y$ indicates that ``we are actively defining $x$ to equal $y$, right now.'' $[n]$ denotes the set $\{0,1,...,n-1\}$. 
A \textit{list} can have repetition in its contents, and order matters.
Given a list of nonnegative integers $n_1,...,n_m$, a  \textit{$n_1 \times ... \times n_m$-combinatorial prism} is a set of the form $Z_1 \times ... \times Z_m$ where $Z_i\subseteq \mathbb{Z}$ and $|Z_i|=n_i$ for all $i$. ($\mathbb{N}\equiv \mathbb{Z}_{\geq 0}$ works just as well.) We may also call this a combinatorial prism of measurements $\{n_i\}$. For ease of writing, we'll write $D\backslash v$ etc. when we really mean $D\backslash \{v\}$ or $V(D)\backslash \{v\}$. Empty sums are 0, and empty products are 1. 

\subsection{(Di)graphs}\label{Digraphdefs}
A \textit{digraph} $D\equiv (V,E)$ is a (usually finite)
set of \textit{vertices} $V\equiv V(D)$ and a set of \textit{arcs} or \textit{arrows} $E(D)\equiv E\subseteq V\times V$. For an arc $(u,v)$ (which we write $\overrightarrow{uv}$), we call $u$ the \textit{start} and $v$ the \textit{end} of the arc. 
We call a digraph \textit{loopless} if $\forall v\in V(D)(v\not\rightarrow v)$. If $\forall u,v\in V(D)$, $(v\rightarrow u \implies u\rightarrow v)$, we call $D$ \textit{undirected} or simply a \textit{graph} and consider each reciprocal arc-pair as a single \textit{edge}, written $\overline{uv}$. We concern ourselves with loopless digraphs (usually denoted $D$), especially undirected ones (usually denoted $G$).

If there's an arc from $v$ to $u$, we write $v\rightarrow u$. (If there isn't, we can write $v\not \rightarrow u$.) If there's also an arc from $u$ to $v$, we write $v\sim u$. We denote the \textit{out-neighborhood} of a vertex $v$ by $N^+(v)$; it consists of all vertices $u$ such that $v\rightarrow u$. . We denote the \textit{in-neighborhood} of a vertex $v$ by $N^-(v)$; it consists of all vertices $u$ such that $u\rightarrow v$.  The \textit{outdegree} $\deg^+(v)$ of a vertex $v$ is $|N^+(v)|$; the \textit{indegree}  $\deg^-(v)$ is $|N^+(v)|$. The maximal outdegree on a graph is denoted $\Delta^+$ and the maximal indegree is $\Delta^-$. If for a vertex $v$ we have $N^+(v)=N^-(v)$, we call $v$ \textit{directionless} and simply write $N(v)$ and $\deg(v)$; the maximal degree in a graph is denoted $\Delta$. A directionless vertex of degree 1 is called a \textit{leaf}. Writing $N(v)$ or $\deg(v)$ implies that $v$ is directionless; writing $\Delta(D)$ implies that $D$ is undirected. If $v$ is directionless and $S=N(v)$, we say that $v$ is \textit{universal to $S$}. For $S\subseteq V(D),$ we define  $N^+(S)=\cup_{s\in S} N^+(s)$, $N^-(S)=\cup_{s\in S} N^-(s)$, and $N(S)=\cup_{s\in S} N(s)$.

If $I\subseteq V(D)$ has $u \not\rightarrow v$ for all $u,v\in I$, we call $I$ an \textit{independent set.}

We define some particular graphs and digraphs; let their vertex set be $[n]$.  A \textit{complete graph} $K_n$ has $i\sim j$ for all $i\neq j$. A \textit{path} $P_n$ has $i\sim j$ if and only if $|i-j|\in \{-1,1\}$; we call $1$ and $n$ its \textit{endpoints} and the others its \textit{interior vertices}. A \textit{cycle} $C_n$ has $i\sim j$ if and only if $((i-j)\mod n)\in \{-1,1\}$; we require $n\geq 3$. A \textit{directed path} $\overrightarrow{P}_n$ has $i\rightarrow j$ if and only if $j-i=1$; we call $1$ its \textit{origin} and $n$ its \textit{terminus}. A \textit{directed cycle} $\overrightarrow{C}_n$ has $i\rightarrow j$ if and only if $(j-1 \mod n) =1$. A \textit{complete bipartite graph} $K_{k, n-k}$ has $i\sim j$ if and only if $i<k$ and $j\geq k$ (or vice versa) for some fixed natural $k$; if $k=1$, we call it a \textit{star}. A \textit{book} $B_{d,n-d}$ has $i\sim j$ if and only if $i<d$ (or $j<d$, or both). A \textit{windmill} $W_{q,p}$ consists of $p$ copies of $K_q$ overlapping on a single vertex.

Although vertices are usually lowercase, we make an exception for paths and cycles. We usually label the vertices of a path $A,B,...,Y,Z$ from left to right, regardless of its length. We'll label the vertices of a cycle $...B,A,\Omega,Z,Y...$, regardless  of its length. (We imagine the cycle as being formed by bending up the ends of the path and attaching a vertex $\Omega$ to them.) 

We'll speak of deleting arcs and vertices. When we delete a vertex, we also delete all arcs starting or ending at that vertex. $D\backslash v$ denotes $D$ with vertex $v$ deleted. For $R\subseteq V(D)$, $D\backslash R$ denotes $D$ with all vertices in $R$ deleted.
If digraph $H$ is obtainable from $D$ by deleting arcs and/or vertices, we write $H\subseteq D$ and say that $H$ is a \textit{subgraph of $D$}. A $K_n$ subgraph of $G$ is called an \textit{$n$-clique}; often this refers just to the vertices of that subgraph. A \textit{forest} is a graph with no cycle subgraph. A \textit{tree} is a forest with one fewer edges than vertices. A \textit{unicyclic graph} is a graph with exactly one cycle subgraph. A subset $S\subseteq V(D)$ is called \textit{acyclic} if $D$ has no directed cycle subgraph all of whose vertices are contained in $S$. 


If for two vertices $a,b\in V(D)$ there exists a directed path subgraph of $D$ with origin $a$ and terminus $b$, we say that $a$ and $b$ are \textit{connected}. If there is also a directed path subgraph of $D$ with origin $b$ and terminus $a$, we say that $a$ and $b$ are \textit{strongly connected}. Extend (strong) connectedness to an equivalence relation; call each equivalence class a \textit{(strong) component}. A digraph is \textit{(strongly) connected} if it has exactly one (strong) component.

We say that two vertices $u,v$ are \textit{peer} if $N^+(v)=N^+(u)$. A set $X$ of pairwise peer vertices is a \textit{peer set}. A set $Y$ of vertices such that any two $u,v\in Y$ are either peer vertices or lie in different components of $D$ is called 
\textit{metapeer.} (Clearly, all metapeer sets are independent.) Let a vertex be called \textit{Lyonic} if $N(v)$ is metapeer in $D\backslash v$.

\subsection{Game-related definitions}
See Section \ref{DefiningGamesAndParameters} for the definitions of Czech/Polish/Latvian/Classic games, plans, colorings, disprovers, winning/losing, outcome, and (un)winnableness. We don't distinguish between vertices and the sages that stand at them. When we say ``game'' without modifier, we mean Czech game. 

We usually write games as $\mathcal{G}\equiv (D,g,h)$. When ``$D$,'' ``$g$,'' or ``$h$'' appear, they pertain to some implicit background game. Similarly, when ``a vertex''  appears without further context, assume that vertex is part of a graph, which is part of a game. 

Consider a hatness function $h$ on a path or cycle. 

For a hatness function $h$ on a path or cycle, we may speak of a ``hatness sequence.'' These may be read from left to right or right to left, clockwise or counterclockwise. They are simply the hatnesses one encounters while traversing the graph. If a cycle or path has (for instance) $h(B)=2,h(C)=3, h(D)=3$, it is said to contain the hatness sequence $(2,3,3)$, and it can also be said to contain the hatness sequence $(3,3,2)$, or $(3,2)$, et cetera.  

Given a subgraph $D'\subseteq D$, game $\mathcal{G}\equiv (D,g,h)$ can be restricted to $\mathcal{G}'\equiv (D',g',h')$ simply by setting $g'=g$ and $h'=h$. We'll sometimes speak of (for instance) ``a strong component of $\mathcal{G}$'' to mean ``$\mathcal{G}$ restricted to a strong component of $D$.'' For a vertex $v$, $\mathcal{G}\backslash v$ denotes the restriction of $\mathcal{G}$ to $D\backslash v$. We call $v$ \textit{deletable} if $\mathcal{G}\backslash v$ has the same outcome as $\mathcal{G}$. For $R\subseteq V(D)$, $\mathcal{G}\backslash R$ and the deletability of $R$ are defined analogously.

Suppose $c_X$ and $c_Y$ are partial colorings on disjoint sets of vertices $X,Y\subseteq V(D)$. The partial coloring on $X\cup Y$ that equals $c_X$ on $X$ and equals $c_Y$ on $Y$ may be written $c_X \cup c_Y$. Similarly, if $f_X$ and $f_Y$ are plans for disjoint sets of vertices $X,Y\subseteq V(D)$, we may write the plan for $X\cup Y$ that agrees with $f_X$ and $f_Y$ on $X$ and $Y$ respectively as $f_X\cup f_Y$. For $c_X$ the restriction of a coloring $c$, we say that $c_X$ \textit{extends to} $c$. 
For $f_X$ the restriction of a strategy $f$, we say that $f_X$ \textit{extends to} $f$. If $f$ wins, we simply say $f_X$ \textit{extends to win.} 

As stated in Section \ref{DefiningGamesAndParameters}, we follow \cite{Szc17} in writing $V_k$ for the set of possible hat colors for a sage $k$ and can speak of it as $[h(k)]$ 
(but never treat $V_k, V_j$ as ``overlappping'' for $k\neq j$).
It will be useful to speak of colorings and partial colorings as ``vectors in hat-space'' or ``color-space" if you prefer. 
\begin{defn}
    In a game $\mathcal{G}$, for any $S\subseteq V(D)$, the \emph{hat-space over $S$}, denoted $\mathbb{H}(S)$, is the collection of all possible vectors of hat assignments for the vertices in $S$. I.e., it's $\prod_{s\in S} V_s$.
\end{defn}
Whenever we speak of randomness (mostly in Chapter 6), the implicit probability space is $\mathbb{H}(V(D))$ with the uniform distribution. $R_v$ denotes the event that sage $v$ guesses right. 

While building plans, we want to know what's ``slipping through.'' The following gives two ways of making that precise. 
\begin{defn}
    Consider $S\subseteq V(D)$. Let $f_0$ be any plan that includes all vertices of $V(D)\backslash S$. Denote the \emph{admissible ends in $S$ for coloring $c$} by 
    $L'(S,f_0, c)\subseteq \mathbb{H}(S)$. It is defined as the collection of all possible partial colorings $d_S$ on $S$ such that $c_{N^+(S)}\cup d_S$ extends to a coloring causing everyone in $V(D)\backslash S$ to guess wrong when guessing according to $f_0$.

    Let $f$ be a strategy. Denote the \emph{accepted ends in $S$ for coloring $c$} by 
    $L(S,f,c)\subseteq \mathbb{H}(S)$. It is defined as the collection of all possible partial colorings $d_S$ on $S$ such that $c_{N^+(S)}\cup d_S$ extends to a coloring causing everyone to guess wrong when guessing according to $f$.

    These definitions work just as well with partial colorings, so long as the partial coloring includes colors for every member of $N^+(S)$.
\end{defn}
Here are some immediate properties of $L$ and $L'$.
\begin{proposition}\label{PropertiesofLL'}
    Fix a strategy $f$ for $\mathcal{G}$.
    \begin{enumerate}
        \item $f$ wins if and only if $\forall S \forall c (L(S,f,c)= \emptyset)$.
        \item For any $v\in V(D)$, $f_{G\backslash v}$ extends to win if and only if $\forall c (|L'(v,f_{G\backslash v},c)|\leq g(v))$.
        \item For any $S\subseteq V(G)$, if $\exists c(|L'(S,f,c)|\leq \sum_{v\in S}g(v))$, then $c$ does not disprove $f$.
        \item If $X_1,...,X_n\subseteq V(D)$ belong to $n$ different components of $D$, then $L'(X_1\cup ... \cup X_n,f_0, c)=L'(X_1,f, c)\times ... \times L'(X_n,f, c)$.
    \end{enumerate}
\end{proposition}
\begin{proof}
    The first, second, and third are all definitional. For the fourth, we need only observe (for $i\neq j$) that the value of $c_{X_i}$ is independent of anyone in $X_j$ guessing right.
\end{proof}
Sometimes comparing games' difficulty is easy even if we don't know their outcomes. 
\begin{defn}
    Consider Czech games $\mathcal{G}\equiv (D,g,h)$ and $\mathcal{G}'\equiv (D',g',h')$.
    We define $\preceq$, the ``ease relation,'' to be the minimal partial order on Czech games such that $\mathcal{G} \preceq \mathcal{G'}$ if any of the following hold:
    \begin{itemize}
        \item $D\subseteq D'$ and $(g(v),h(v))=(g'(v),h'(v))$ for all $v\in v(D)$.
        \item $D=D', g=g'$, and $h'\leq h$.
        \item $D=D', h=h'$, and $g'\geq g$.
        \item $D=D'$, and there exists a vertex $v$ and positive integer $k$ such that $k\cdot h(v)=h'(v), k\cdot g(v)=g'(v)$, with $h=h'$, $g=g'$ on other vertices.
    \end{itemize}
    I.e., $\preceq$ is the transitive closure of those four conditions. Notice that the middle two bullet points could be combined to ``$D'=D$ and $r'\geq r$.''
\end{defn}
``$\mathcal{G} \preceq \mathcal{G}'$'' is said ``$\mathcal{G}'$ is easier than $\mathcal{G}$,'' or if you want to be strictly correct, ``$\mathcal{G}'$ is no harder than $\mathcal{G}$.'' That interpretation is justified, as we show.
\begin{proposition}\label{monotone}
    If $\mathcal{G}\preceq \mathcal{G}'$ and $\mathcal{G}$ is winnable, then $\mathcal{G}'$ is winnable. (Moreover, given a winning strategy on $\mathcal{G}$, a winning strategy can be constructed for $\mathcal{G}'$ in polynomial time.)
\end{proposition}
\begin{proof}
Because of transitivity, it suffices to prove this for the case when $\mathcal{G}'$ differs from $\mathcal{G}$ in just one of the four bullet points. We address each in order. 
\begin{itemize}
    \item Let the sages in $V(D)\subseteq V(D')$ play and win however they were doing it before. The other sages can do whatever. 
    \item Let all the sages play and win however they were doing it before––it just so happens that certain colors don't show up. 
    \item Let all the sages play and win however they were doing it before––then add some extra guesses however you want. 
    \item For those viewing $v$, they should treat the $k\cdot h(v)$ colors as $k$ copies of $h(v)$ colors; they treat each copy the same as the old color. $v$ herself plays according to her old strategy, except that when it comes time to guess colors, she instead guesses all $k$ copies of each of those colors. \qedhere
\end{itemize}
\end{proof}

\section{Warm-up}
Let's get a feel for these problems. First we solve the Czech game on complete graphs, following \cite{BDO21}. Then we prove some obvious results: vertices that leave something unguessed are deletable, lone sages can't win, and games are winnable if and only if they have a winnable strong component. We apply these to solve the Latvian and Polish games for directed cycles.
\subsection{Complete graphs}
Ratio games have been solved for chordal graphs, and Classic games have been solved for many graph families, but Czech games haven't been solved for anything but complete graphs. Nor, even, have Latvian games! (Until this thesis, at least.) However, the solution for Czech games on complete graphs is elegant and easy.
\begin{prop}[Theorem 5 of \cite{BDO21}]\label{CzechOnComplete}
    The game $(K_n, g,h)$ is winning if and only if 
    \[\sum^n_i \frac{g(i)}{h(i)} \geq 1.\]
        \end{prop}
    
        This proof follows the same idea as \ref{proofofrainbow}. For ``only if,'' we show that if the condition fails, a union bound of probabilities guarantees some instance of failure. For ``if,'' we again use the fact of complete visibility to connect a sage's guess about her own color to the (weighted) sum of all hat colors. We engineer their guesses to cover the space of possibilities for that sum.     
\begin{proof}
        \emph{Only if.} The odds that sage $j$ guesses correctly is always $\frac{g(j)}{h(j)}$, so the expected number of correct guesses is $\sum^n_i \frac{g(i)}{h(i)}$. If this value is $<1$, then with positive probability there are 0 correct guesses.
    
        \emph{If.} We'll present an algorithm that wins, assuming $\sum^n_i \frac{g(i)}{h(i)} \geq 1$. (It's clearly polynomial, too, but we won't prove that.) 

        Arbitrarily label the vertices of $K_n$ as $1,2,...,n$. Let $c(i)\in [h(i)]$ be the color of $i$. Let $\pi$ denote the quotient map from $\mathbb{R}$ to $\mathbb{R}/\mathbb{Z}$. That is, $\pi: x\mapsto x-\lfloor x \rfloor$. For each vertex $j$, set $\hat{I}_j=[\sum_{i<j} \frac{g(i)}{h(i)},\sum_{i\leq j} \frac{g(i)}{h(i)})$, and set  $I_j=\pi(\hat{I}_j)$. For any coloring $c$, set $T(c)=\sum_{i=1}^n \frac{c(i)}{h(i)}.$ 
        The strategy is just this: sage $j$ guesses that $\pi(T(c))\in I_j$. 
    
        Proving that this works is a good exercise. If you don't wish to try it yourself, here's a sketch. Sage $j$ can determine $T(c)$ up to the unknown term of $\frac{c(j)}{h(j)}$. Because $\hat{I}_j$ is a half-open interval with length $\frac{g(j)}{h(j)}$ (and all $h(j)$ colors are possible for $c(j)$), there are exactly $g(j)$ possibilities for $c(j)$ such that $\pi(T(c))\in I_j$; sage $j$ guesses exactly these; she is right if and only if $\pi(T(c))\in I_j$. So if $\sum_{i=1}^n \frac{g(i)}{h(i)} \geq 1$, then the various $I_j$ cover $[0,1)$, so some sage must guess right.
    \end{proof}

    \begin{remark}
        The argument above essentially shows that $\mu(K_n)=s^{-1}\mu_s(K_n)=\hat{\mu}(K_n)=n$. All these parameters decrease with the removal of an edge from $K_n$; the idea is that two vertices without mutual visibility can with positive probability both guess right, so if the expected number of correct guesses is exactly $1$, then with positive probability there are no correct guesses.
    \end{remark}


\subsection{Strong components and clearly useless vertices}
As will be common, fix some $\mathcal{G}$ and $f$.
\begin{prop}\label{prop:firstdeletable}
    For a strategy $f$, if there's some color in $V_v$ that sage $v$ never guesses, then $v$ can be deleted; $f_{D\backslash v}$ wins $\mathcal{G}\backslash v$ if and only if $f$ wins $\mathcal{G}$.
\end{prop}
\begin{proof}
    Take the color that she never guesses (say it's mauve), and fix that color on her head. She'll never guess correctly. All who saw $v$ are now playing their ``how do I interpret everyone else when sage $v$ is wearing a mauve hat'' sub-plans. If that collection of plans wins, then it works as a strategy for $\mathcal{G}\backslash v$, and if it doesn't, then $f$ doesn't win. 
\end{proof}
\begin{prop}\label{DeleteHighRatio}
    Let $\mathcal{G}$ be a Czech game and $v$ a vertex. If $h(v)/g(v)>\prod_{u\in N^+(v)} h(u)$, then $v$ is deletable.
\end{prop}
\begin{proof}
    Sage $v$ can only see $\prod_{u\in N^+(v)} h(u)$ possible things, and for each of those possible sights, she can only guess $g(v)$ of her colors. If $h(v)>g(v)\prod_{u\in N^+(v)} h(u)$, then there is some color she never guesses. Invoke Proposition \ref{prop:firstdeletable}.
\end{proof}
\begin{corollary}\label{onevertexloses}
    Any game played on $K_1$ is unwinnable.
\end{corollary}
\begin{proof}
    The empty product has value 1, and $h>g>0$, so $h/g>1$. Apply Proposition \ref{DeleteHighRatio}, and try to do so with a straight face. 
\end{proof}

\begin{prop}\label{prop:strongcomponly}
    A Czech game $\mathcal{G}\equiv (D,g,h)$ with multiple strong components $D_1,...,D_n$ is winnable if and only if its restriction to one of the strong components is.
\end{prop}
\begin{proof}
    ``If'' follows by Proposition \ref{monotone}. For ``only if,'' assume it's unwinnable on every strong component, and fix a strategy $f$ for $\mathcal{G}$. There must be some strong component with no arcs to any other strong component, since if there were not, these strong components would not all be distinct. Without loss of generality, let that strong component with no arcs to any other be $D_n$. Since $\mathcal{G}_{D_n}$ (the game restricted to $D_n$) is unwinnable, we can color $D_n$ in such a way that none in $D_n$ guess right according to $f$. Now, define a strategy on $\mathcal{G}\backslash D_n$: everybody imagines they see this particular assignment on $D_n$ and plays according to $f$. If someone guessed correctly according to $f$ in $\mathcal{G}$, they were not in $D_n$, so they still guess correctly. Thus, if $\mathcal{G}$ is winnable, so is $\mathcal{G}\backslash D_n$. 
    $\mathcal{G}\backslash D_n$ must in turn have some strong component with no arrows to any other strong component of $\mathcal{G} \backslash D_{n}$. Call that component $D_{n-1}$. Arguing similarly, we show that $\mathcal{G}\backslash D_n  \backslash D_{n-1}$ is winnable. We carry on in this fashion until we show that $\mathcal{G}\backslash D_n \backslash ... \backslash D_2$, i.e. the game restricted to $D_1$, is winnable, which contradicts our assumption. 
\end{proof}
\begin{corollary}\label{nodirectedcycle}
    If $D$ has no directed cycle subgraph, then $\mathcal{G}$ is unwinnable. (An edge counts; it's a directed 2-cycle.)
\end{corollary}
\begin{proof}
    In a digraph with no directed cycle, the strong components are precisely the vertices. Apply Proposition \ref{prop:strongcomponly} and Corollary \ref{onevertexloses}.
\end{proof}

\subsection{Directed cycles}
\begin{thm}\label{LatvianDirectedCycles}
    A Latvian game on a directed cycle is winning if and only if $h(v)=2$ for all $v$.
\end{thm}
\begin{proof}
    \textit{If, version 1.} Let one player assume ``my hat is the color I see.'' Let all other players assume ``my hat is the color I don't see.'' If all hats are the same color, the first sage guesses right. If not all hats are the same color, there are at least two arcs of the graph that go between different-colored hats. At least one of those arcs has an opposite-guesser standing at the source; she will guess correctly.

    \textit{If, version 2.} If the cycle's length is even, do as above. If it's odd, they can win with identical plans: ``my hat is the color I see.'' See Question \ref{SameStratQuestion}. 
    
    \textit{Only if, version 1.} By Proposition \ref{monotone}, it suffices to show that it's unwinnable if one sage $v$ has hatness 3 and all others have hatness 2. But that sage $v$ will only ever see two things, so she won't guess one of her three colors. Put that color on her head. Then she's guaranteed to guess wrong. Some sage sees $v$; call her $w$. Her guess is already fully determined by the hat on $v$. So we can put a hat on $w$ that she won't guess. Supposing $u$ sees $w$, we can repeat this procedure all the way around the cycle, until we wind up back at $v$, who's already guaranteed to guess wrong. 
    
    \textit{Only if, version 2.} Again assume $v$ has hatness 3 and all others have hatness 2. By Proposition \ref{DeleteHighRatio}, $v$ is deletable. By Corollary \ref{nodirectedcycle}, the game that remains after deleting $v$ is unwinnable.

    \textit{Only if, version 3.} Apply Lemma \ref{DirectedCyclesLemma}.
\end{proof}

So what about Czech games?
\begin{lemma}\label{DirectedCyclesLemma}
    Let $\mathcal{G}$ be a Czech game on a directed cycle. Suppose sage $v$ sees sage $u$, and suppose that $1-g(v)/h(v) > g(u)/h(u)$, i.e. that $1-r(v) > r(u)$. Then $\mathcal{G}$ is unwinnable. 
\end{lemma}
\begin{proof}
    We construct a disprover. Some color on $v$ is guessed for at most $\lfloor  h(u)g(v)/h(v) \rfloor$ of the possible colors of $u$, by a counting argument. Fix that color on $v$ and extend it to a coloring on $G\backslash u$ such that nobody is guessing correctly. (This extension can be done sequentially backwards along visibility arrows.) Once in this setup, $v$ hasn't decided her guess yet, and $u$ hasn't had her hat assigned yet, though her guess is decided. There are $\leq \lfloor  h(u)g(v)/h(v) \rfloor$ colors of $u$ that, if picked, would cause $v$ to guess right. There are $g(u)$ colors of $u$ that, if picked, would cause $u$ to guess right. I.e., $\leq h(u)g(v)/h(v) + g(u)$ colors are forbidden. Thus, for there to be a non-forbidden color for $u$, it suffices to have $h(u)> h(u)g(v)/h(v) + g(u)$. That rearranges to $1-r(v)>r(u)$.
\end{proof}

\begin{thm}\label{LatvianPolish}
     $\mu(\overrightarrow{C}_k)=\frac{\mu_s(\overrightarrow{C}_k)}{s}=\hat{\mu}(\overrightarrow{C}_k)=2$.
\end{thm}
\begin{proof}
    Since $\hat{\mu}=\sup s^{-1}\mu_s$, and $\mu=\mu_1$, it suffices to prove that $s^{-1}\mu_s(\overrightarrow{C}_k)=2$ for all $s$. Because $r$ is constant in a Polish game, if $r>\frac{1}{2}$, we have $1-r>r$, so that game is unwinnable by Lemma \ref{DirectedCyclesLemma}. If $r\geq \frac{1}{2}$, however, that game is winnable by Proposition \ref{monotone} and Theorem \ref{LatvianHAH}.
\end{proof}

\chapter{Two tools and one use}
In that chapter, we take our first two ``Slavic techniques'' of interest and generalize them, demonstrating their usefulness by proving the conjecture from \cite{KLR21b} on Latvian cycles and going further to classify \textit{all} Latvian cycles. 

The first technique is ``Hats As Hints.'' It formalizes the philosophy that every sage may assume all her neighbors are wrong. This underlies all winning strategies, but it had only been explicitly deployed to show that leaves of hatness $\geq 3$ can be deleted in Latvian games \cite{KLR21b}. Our formulation reveals more precisely how and when a new vertex can help extend a plan to a winning strategy. In this way ``Hats As Hints'' operates as a ``constructor for strategies,'' allowing you to combine a suite of strategies on a game-with-a-hint into a strategy on a related game. It also guides the search for disprovers, which is how it's deployed for the Latvian cycle problem. 

That search for disprovers is conducted by our second technique, 
``admissible graphs" (or assignments, or paths). Admissible paths were developed in \cite{Szc17} to study Classic cycles of hatness 3. They are essentially partial colorings that don't yet cause any sages to guess right. The rules for expanding such partial colorings work as ``constructors for disprovers.'' We provide a couple formalisms to generalize them to Czech games, but the application is limited to certain Latvian paths. 

We combine the techniques to show that two infinite families of Latvian cycles are unwinnable, as conjectured by \cite{KLR21b}, completing the characterization of all winnable Latvian cycles with $h\leq 4$. (Before doing so, we give history for the problem.) The specific means is to turn one vertex of hatness 4 into a hint, then show that there are enough admissible paths in the remaining game for the hint to be insufficient. 

The obvious follow-up question is: what if $h(v)\geq 5$ for some $v$? We answer this too, showing that $v$ is deletable. The long and unenlightening proof is only sketched.

\section{Motivating problem}\label{sec:MCP}
We'd like a polynomially-checkable characterization of winnable and unwinnable Latvian cycles. The problem begins with Szczechla \cite{Szc17}, who made the first nontrivial determination of $\mu$ for a graph class. With great ingenuity, he proved the following.
\begin{lmm}\label{SzczechlaCycleContribution}
    Let $\mathcal{G}\equiv (C_k, h)$ be a Latvian game on a cycle. 
    \begin{itemize}
        \item If $\forall v(h(v)\geq 3)\wedge \exists u (h(u)\geq 4)$, then $\mathcal{G}$ is unwinnable. 
        \item If $\forall v(h(v)=3)$, then $\mathcal{G}$ is winnable if and only if $k=4$ or $3\mid k$. 
    \end{itemize}
\end{lmm}
Kokhas, Latyshev, and Retinskiy claim in the abstract of \cite{KLR21b} to ``solve the problem'' ``for cycles,'' they in fact only do so for $C_3$ and $C_4$. 

\begin{lmm}[Theorems 2.1 and 4.2 of \cite{KLR21b}]\label{C3C4Solved}
    A Latvian game on $C_3$ is winnable if and only if $\sum_{v\in V(C_3)}h(v)^{-1}\geq 1$. A Latvian game on $C_4$ is winnable if and only if its hatness sequence is no harder than one of the following: $(3,3,3,3)$, $(2,2,\infty,\infty)$, $(2,4,2,\infty)$, $(2,3,3,4)$, $(3,2,3,4)$.
\end{lmm}

For general $k$, they prove the following:
\begin{lmm}\label{KLRCycleContribution}
    Let $\mathcal{G}\equiv (C_k, h)$ be a Latvian game on a cycle. Suppose $2\leq h(v) \leq 4$ for all $v\in C_k$. Suppose $h(A)=2$ for vertex $A\in V(C_k)$. Then $\mathcal{G}$ is winnable in any of the following conditions.
    \begin{itemize}
        \item $k=3$.
        \item Some other vertex has hatness $2$.
        \item The hatness sequence $(3,2,3)$ appears.
        \item The hatness sequence $(2,3,3)$ appears. 
    \end{itemize}
\end{lmm}

The Latvian cycles with $h\leq 4$ that remain unaddressed by \ref{SzczechlaCycleContribution} and \ref{KLRCycleContribution} are precisely those that contain the hatness sequences $(4,2,4)$ or $(4,2,3,4)$ and contain no other vertices of hatness $2$. The easiest games of this sort are those all of whose other vertices have hatness 3, so if \textit{those} games are unwinnable, \textit{all} games of this sort are unwinnable. For $k\leq 4$, this is proven by Proposition \ref{C3C4Solved}. Thus, \cite{KLR21b} conjectured the following (claiming a proof for $k\leq 7$ by computer checking): 
\begin{lmm}\label{ConjectureSufficientForTwoToFour}
    Let $\mathcal{G}\equiv (C_k, h)$ be a Latvian game on a cycle, $k\geq 5$. The hatness sequences $(... 3,4,2,4,3, ...)$ and $(... 3,4,2,3,4,3, ...)$ represent unwinnable games.
\end{lmm}
 Together with Lemmata \ref{SzczechlaCycleContribution}, \ref{C3C4Solved}, and \ref{KLRCycleContribution},  Lemma \ref{ConjectureSufficientForTwoToFour} entails Theorem \ref{CategorizedTwoToFour}. We'll prove it. 
\begin{thm}\label{CategorizedTwoToFour}
    Let $\mathcal{G}\equiv (C_k, h)$ be a Latvian game on a cycle. Suppose $2\leq h(v) \leq 4$ for all $v\in C_k$. $\mathcal{G}$ is winnable if and only of one of the following holds.
    \begin{itemize}
        \item $k=3$ and $\exists v\in V(C_k)(h(v)=2)$.
        \item $k$ is either equal to 4 or divisible by 3, and $\forall v\in V(C_k)(h(v)\leq 3)$.
        \item For at least two $v$, $h(v)=2$. 
        \item The hatness sequence $(3,2,3)$ appears.
        \item The hatness sequence $(2,3,3)$ appears. 
    \end{itemize}
\end{thm}

The obvious question is: what if $\exists a(h(a)\geq 5)$? 
\begin{thm}\label{CategorizedFiveAndUp}
    Let $\mathcal{G}\equiv (C_k, h)$ be a Latvian cycle with $h(a)\geq 5$ for some $a\in V(C_k)$. Then $a$ is deletable. (Equivalently, $\mathcal{G}$ is winnable if and only if it contains a hatness sequence $(2,c_1,...,c_m,2)$ $m\geq 0$ and $\forall i(c_i\in \{3,4\})$.)
\end{thm}
I intuited that answer before even proving Theorem \ref{CategorizedTwoToFour}, but I mucked about with tools from Chapter 5 and Section \ref{sec:HAH} before finding a proof based mostly on admissible paths. 

\begin{remark} 
    The first bullet of Lemma \ref{SzczechlaCycleContribution} can easily be proven with admissible paths, the $C_3$ part of Lemma \ref{C3C4Solved} follows from Proposition \ref{CzechOnComplete}, the $C_4$ part is achievable with some constructors and the arguments we present in this chapter, and Lemma \ref{KLRCycleContribution} follows immediately from constructors. However, the second bullet of Lemma \ref{SzczechlaCycleContribution} is not subsumed by this thesis. Szczechla's tools are quite specialized. Also, we've checked Lemma \ref{23532lemma} brutishly, but we'd like a proof that can be written down concisely. 
\end{remark}

\section{Hats As Hints}\label{sec:HAH}
\epigraph{A prophet I, madam, and I speak the truth the next way...} {Fool, \textit{All's Well That Ends Well}, 1.3.22}

The ``Hats As Hints'' perspective systematizes the main surprise of hat guessing. Brigid's hat has nothing at all to do with Agatha's hat, or Zdislava's for that matter, yet the sages can still extract meaning from the colors they see on each other! They can do this precisely because they discuss their plans with one another beforehand. So suppose that Agatha sees a green hat on Brigid and an orange one on Clementia, and she knows––thanks to the earlier discussion––that if Brigid sees orange on Clementia and indigo on Agatha, then she'll guess green. (For instance.) So if Agatha's wearing indigo, someone (Brigid) will guess correctly, and it doesn't matter what Agatha guesses. So Agatha can freely assume that she isn't wearing indigo! It's of course a little more complicated and reciprocal than this, but it still makes sense to ask: if sage $v$ has color $x$, what does that tell everyone else? 

Such reasoning is used implicitly in many constructors of \cite{KLR21b} (e.g. Theorem 3.6), though never with such generality. It's explicit in Theorems 3.9 and 3.10, in which the authors use the unhelpfulness (shown earlier in \cite{KL19}) of a hint called $A^*$ to show the unhelpfulness in Latvian games of a leaf $\ell$ with $h(\ell)>2$.

A ``hint,'' for our purposes, is a construction of the following sort: during the game (after hat placement and before guessing), the Adversary comes up to a predetermined sage and whispers some truthful statement in her ear. (Hints can also be considered with multiple recipients, and you can consider multiple hints simultaneously.) The statement's recipients and parameters of possible content are fixed and known to all before hat placement, but its \textit{specific} content is known only to the Adversary and the recipients. 
For instance, the hint $A^*$ from \cite{KLR21b} consists of the Adversary's public promise ``during the game, I'll tell Agatha two colors, and her hat will be one of those two colors'' and its fulfillment. 
Everyone but Agatha must have a fixed plan as always (after all, they don't hear the hint), but Agatha has $\binom{h(A)}{2}$ plans, one for each possible statement she could hear from the Adversary. 

Theorem 3.9 from \cite{KLR21b} is precisely that this hint doesn't help: if the Adversary won before, he still wins. 
Theorem 3.10 from \cite{KLR21b} says that adding a leaf $\ell$ with $h(\ell)=3$ doesn't help either. It's proven by 3.9 and the argument that adding a leaf of hatness 3 is no better than giving the $A^*$ hint to whatever sage neighbors the leaf. 

Our Proposition \ref{HeadlineHints} greatly expands this argument. Proposition \ref{HintWinning} generalizes Theorem 3.9 \cite{KLR21b}, and Theorem \ref{FinalCzechHintTheorem} generalizes Theorem 3.10. (Corollary \ref{cor:deleteCzechleaves} is, in its statement, a more fun, interesting, and obvious generalization. Its proof, however, makes it an inferior analogy.)

\begin{defn}\label{def:hints}
Let $\mathcal{G}\equiv (D,g,h)$ be a Czech game and $S,R\subseteq V(D)$. A \emph{$j,k$-hint with respect to $S,R$} is a collection  $\{A_i\}_{i\in [k]}$ of $k$ subsets $A_i\subseteq \mathbb{H}(S)$ such that for each $x\in \mathbb{H}(S)$, there are exactly $j$ different $A_i$ such that $x\in A_i$.  

Here is how the game-with-hint $\mathcal{G}'$ is played. The sages play on digraph $D$ with guessness function $g$ and hatness function $h$, but during the game, the Adversary comes up to each sage in $R$ and says, truthfully, ``the vector $c(S)$ does not lie in $A_i$'' for some $i\in [k]$. The Adversary says the same $i$ to all members of $R$. We may write this datum as ``$\neg A_i$.'' 

The nomenclature from hintless games extend naturally. 
A \textit{strategy} $f'$ for $\mathcal{G}'$ is a single plan $f'_w$ for each sage  $w \in V(G)\backslash R$, along with a \textit{$k$-plan ensemble} for each sage $u\in R$, which is just a $k$-tuple of plans $f'_u=(f_u^0, f_u^1... f_u^{k-1})$. Each corresponds to some $A_i$. For a given $A_i$, use $f_R^i$ to denote the plan on $R$ corresponding to $A_i$. A \textit{disprover} for a strategy $f'$ on $\mathcal{G}'$ is a coloring $c$ such that, for some $i$, $c$ disproves  $f'_{V(G)\backslash R} \cup f_{R}^i$ and $c\notin A_i$. If a strategy has no disprovers, we call it \textit{winning}; otherwise it is \textit{losing}. If a game has a winning strategy, it is \textit{winnable}; otherwise it is \textit{unwinnable}.

When $\mathcal{G}$ is a game, $v$ is a vertex of $G$, and $H$ is a $g(v),h(v)$-hint with respect to $N^+(v),N^-(v)$, we will denote the game ``$\mathcal{G}\backslash v$-but-with-the-hint-$H$'' as $\mathcal{G}_{vH}$.
\end{defn}
\begin{remark}\label{RemarkAboutSimpleHints}
    If you want a simple-to-understand $j,k$-hint, take a partition of $\mathbb{H}(N^+(v))$ into $k$ parts $P_0,P_1,P_2...P_{k-1}$. Then consider $A_i$ to consist of $P_i \cup ... \cup P_{i+j-1}$, where addition is modulo $k$.
\end{remark}

The headline fact is Proposition \ref{HeadlineHints}.9: for any game $\mathcal{G}$ and vertex $v$, $\mathcal{G}$ is winnable if and only if $\mathcal{G}/v$ is winnable with some $g(v),h(v)$-hint with respect to $N^+(v),N^-(v)$. 
Our motto throughout is ``a plan for $v$ translates colors to hints.'' Here's how.

\begin{defn}\label{def:derived}
    Let $v$ be a vertex in $\mathcal{G}$ with plan $f_v$. We define a $g(v),h(v)$-hint with respect to $N^+(v),N^-(v)$ as follows. For each color $i\in [h(v)]$, define $A_i=f_v^{-1}(i)$. I.e., $A_i$ is the set of $x\in \mathbb{H}(N^+(v))$ such that, when $v$ sees $x$, her guesses include $i$. Call this hint $H_{f_v}$, and call the game $\mathcal{G}_{vH_{f_v}}$ the \textit{derived game}. 

    Consider any $u\in N^-(v)$ with plan $f_u$. We define a $k$-plan ensemble for $u$ in $\mathcal{G}_{vH}$ as follows. Let $f_u^i(...)=f_u(i, ...)$, where $c(v)$ is in the first place of $f_u$. That is, when $u$ is told ``$\neg A_i$'', she interprets the hats she sees as if she were in $\mathcal{G}$, seeing an $i$ hat on $v$. Call this $f'_u=(f_u^0,...,f_u^{h(v)-1})$. For any plan $f_Z$ on $Z\subseteq V(D)$, let $f'_Z$ denote $\bigcup_{u\in N^-(v)\cap Z} f'_u \cup \bigcup_{w\in Z\backslash N^-(v)} f_w$; call it the \textit{derived plan}. (If $Z=V(D)\backslash v$, call it the \textit{derived strategy}.)
\end{defn}

The following are obvious but worth writing down. 

\begin{prop}\label{HeadlineHints}
    Let $\mathcal{G}$ be a game and $v\in V(D)$ a vertex.
    \begin{enumerate}
        \item The map $f_v \mapsto H_{f_v}$ is bijective. 
        \item For $u\rightarrow v$, the map $f_u \mapsto \{f_u^0,...,f_u^{h(v)-1}\}$ is bijective.
        \item There is a bijection from plans $f_Q$ on $Q\subseteq V(D)$ with $v\in Q$ to pairs $(H,f'_{Q\backslash v})$, where $H$ is a $g(v),h(v)$-hint with respect to $N^+(v),N^-(v)$, and $f'_{Q\backslash v}$ is a plan on $Q\backslash v$.
        \item For any $H$, there's a bijection from plans $f_Q$ on $Q\subseteq V(D)$ in $\mathcal{G}$ with $v\notin Q$ to plans $f'_Q$ on $Q$ in $\mathcal{G}_{vH}$.
        \item Strategy $f$ wins $\mathcal{G}$ if and only if the derived strategy wins the derived game. 
        \item Let $f_S$ be a plan in $\mathcal{G}$ for some vertices $S$ with $v\in S$. $f_S$ extends to win $\mathcal{G}$ if and only if $f'_S$ extends to win $\mathcal{G}_{vH_{f_v}}$.  
        \item Let $f_Q$ be a $\mathcal{G}$-plan on $Q\subseteq V(D)$ with $v\notin Q$. $f_Q$ extends to a winning strategy on $\mathcal{G}$ if and only if $f'_Q$ extends to a winning strategy on $\mathcal{G}_{vH}$ for some $g(v),h(v)$-hint $H$ with respect to $N^+(v),N^-(v)$.
        \item Consider a plan $f_v$ for $v$. $f_v$ extends to a winning strategy if and only if $\mathcal{G}_{vH_{f_{v}}}$ is winnable. 
        \item $\mathcal{G}$ is winnable if and only if $\mathcal{G}_{vH}$ is winnable for some $g(v),h(v)$-hint $H$ with respect to $N^+(v),N^-(v)$.
    \end{enumerate}
\end{prop}
\begin{proof} Essentially just the application of definitions.
    \begin{enumerate}
        \item A plan for $v$ is precisely an assignment, to each element of $\mathbb{H}(N^+(v))$, exactly $g(v)$ elements of $V_v$. That is equivalent to a family $\mathcal{F}$ of $h(v)$ subsets $F_i\subseteq \mathbb{H}(N^+(v))$ such that each element of $\mathbb{H}(N^+(v))$ is contained in exactly $g(v)$ of these $F_i$, and that's a hint. 
        \item A plan for $u$ in $\mathcal{G}$ with $u\rightarrow v$ is precisely a rule that assigns, given a coloring of $v$ and the rest of $N^+(u)$, outputs a guess. That is equivalent to a family of $h(v)$ rules, each of which outputs a guess for every coloring of $N^+(u)\backslash v$, which is precisely the $h(v)$-plan ensemble obtained by this map.
        \item Follows from 1 and 2, with the additional observation that for $u\not\rightarrow v$, we have the map $f_u\mapsto f_u$, which is clearly bijective. 
        \item Follows from 2 and the observation in 3. 
        \item Suppose $f$ wins $\mathcal{G}$. A disprover on $D\backslash v$, paired with a true hint $\neg A_i$, corresponds via the above bijections to a disprover for $\mathcal{G}$.
        \item Follows from 5. So do 7, 8, and 9.  \qedhere 
    \end{enumerate} 
\end{proof}

\begin{corollary}\label{cor:redundantvisiondelete}
    For $S\subseteq V(D)$, if $N^+(S) \subseteq \bigcap_{u\in N^-(S)} N^+(u)$, then $S$ is deletable. 
\end{corollary}
\begin{proof}
    Delete any $v\in S$ and give the hint $H_{f_v}$. This hint tells nobody anything they can't see for themselves, so it can be ignored. Iterate.
\end{proof}

\begin{prop}\label{HintWinning}
    Let $\mathcal{G}\equiv (G,g,h)$ be a game. Let $\mathcal{G}_H$ be the same game with a $j,k$-hint $H$ with respect to $S,R$. The constituent sets of $H$ are $A_i$.
\begin{enumerate}
    \item A strategy $f$ for $\mathcal{G}_H$ wins if and only if, for all $i$ and all $c$, $L_{\mathcal{G}\backslash v}(S, f_R^i \cup f_{D\backslash R}, c)\subseteq A_i$. 
    \item There exists winnable $\mathcal{G}_H$ if and only if there exists a plan $f_{D\backslash R}$ on $V(D)\backslash R$ and $k$ plans $f^i_R$ on $R$ such that $\bigcap_{i\in I} L_\mathcal{G}(S, f_R^i \cup f_{D\backslash R}, c)=\emptyset$ whenever $|I|>j$. 
    \item If $\mathcal{G}$ is such that $\forall f \exists S\exists c(|L_{\mathcal{G}}(S, f, c)|>\lfloor \frac{j|\mathbb{H}(S)|}{k}\rfloor$, then $\mathcal{G}_H$ is unwinnable. 
\end{enumerate}
\end{prop}
\begin{proof} You may find this all obvious.
   \begin{enumerate}
       \item There exists a $c$ and  $i$ such that the Adversary can truthfully say ``$\neg A_i$" and nobody guesses right, if and only if $\exists i\exists c(L_{\mathcal{G}\backslash v}(S, f_R^i \cup f_{D\backslash R}, c)\not\subseteq A_i)$. By the definition of winnability for $\mathcal{G}_H$, then, the sages win if and only if no such $i,c$ exist.
       \item Set $A'_i= L_{\mathcal{G}\backslash v}(S, f_R^i \cup f_{D\backslash R}, c)$. No $j+1$ of them have nonempty intersection, so we may freely add elements to them to obtain $A_i$ satisfying the conditions of a $j,k$-hint $H$. 
       \item From 2. \qedhere
   \end{enumerate}
\end{proof}

Parts of Propositions \ref{HeadlineHints} and \ref{HintWinning} can be mixed and matched to yield useful statements that don't reference ``hints.''  Here's the consequence of combining \ref{HeadlineHints}.9 and \ref{HintWinning}.2.

\begin{thm}\label{FinalCzechHintTheorem}
    Let $\mathcal{G}$ be a game and $v$ a vertex. $\mathcal{G}$ is winnable if and only if there exists a plan $f$ on $V(D)\backslash v\backslash N^{-}(v)$ and $h(v)$ plans $f^i$ on $N^{-}(v)$ such that $\bigcap_{i\in I} L_{\mathcal{G} \backslash v}(N^+(v), f^i \cup f, c)=\emptyset$ whenever $|I|>g(v)$.
\end{thm}
It's perhaps surprising that the condition in Theorem \ref{FinalCzechHintTheorem} is satisfied for \textit{some} vertex if and only if it's satisfied by \textit{all} vertices. 

    \begin{coro}\label{LatvianHAH}
        Let $\mathcal{G}$ be an undirected Latvian game and $v$ a vertex. 
        \begin{enumerate}
            \item $\mathcal{G}$ is 
            winnable if and only if there exists a plan $f$ on $V(G)\backslash v\backslash N(v)$ and $h(v)$ plans $f^i$ on $N(v)$ such that $L_{\mathcal{G} \backslash v}(N(v), f^i \cup f, c)$ are disjoint for different $i$. 
            \item Let $f$ be a plan on $G\backslash v$. If $\exists c(|L_{\mathcal{G} \backslash v}(N(v), f, c)|>\lfloor \frac{|\mathbb{H}(N(v))|}{h(v)}\rfloor)$, then $f$ does not extend to a winning strategy on $G$.
        \end{enumerate}
    \end{coro}
    \begin{proof}
        Respectively, from specializing Theorem \ref{FinalCzechHintTheorem}, and from specializing the combination of Propositions \ref{HintWinning}.2 and \ref{HeadlineHints}.7.
    \end{proof}
Also, if you want to speak of a $1,k$-hint with respect to $R,R$ in an undirected Latvian game, you may simply say ``$k$-hint with respect to $R$.''




\section{Admissible paths}\label{AdmissibleGraphs}
\epigraph{For it is easier for a camel to go through a needle's eye, than for a rich man to enter into the kingdom of God.
}
{Luke 18:25}
\subsubsection*{The idea}
The critical tool for \cite{Szc17}'s determination of $\mu(C_k)$ for all $k$ was the idea of an \textit{admissible path}, which represents partial colorings such that every sage either guesses wrong or hasn't seen enough to make up her mind. It's a tool for visualizing and systematizing the brute computations about disprovers, either finding one or showing one needn't exist. It does so slowly and iteratively. We'll give two formalisms. Szechla's original formalism (also used in \cite{KL19} for computational models) used a sort of blown-up copy of the graph, with admissible paths particular paths in that blown-up copy. We call this the ``Warsaw graph.'' We generalize that formalism, but first we give one that's simpler to state: admissible partial colorings. (Szczechla's original employed further constructions, specific to $(C_k,3)$, called ``directed paths,'' ``colourable strategies,'' or the heroically devised ``characteristic number of a strategy.'')
Throughout, we write ``admissible'' when really we mean ``admissible with respect to a fixed strategy $f$,'' or in the language of \cite{Szc17}, ``$f$-admissible.''

\subsubsection*{Admissible partial colorings}
For a given strategy $f$, we call a partial coloring \textit{admissible on $Q$} if no sage in $Q\subseteq V(D)$ guesses correctly. 
If the domain of $c$ induces a path and $c$ is admissible on every vertex of that path, we call $c$ a \textit{admissible path}. If it is admissible on every interior vertex of that path, we call it a \textit{quasi-admissible path}.

\subsubsection*{Admissible graphs}
For visibility digraph $D$ and hatness function $h$, We define the \textit{Warsaw (di)graph}\footnote{We name it after Szczechla's home university. It was subsequently used in \cite{KL19}, as well as its line graph.} $h*G$ as follows. 
Set $V(h*D)=\coprod_{i\in V(D)}V_i $. (As a reminder, $V_i$ is the set $[h(i)]$ of possible colors for sage $i$, though in $h*D$ it functions as a set of vertices.) For distinctness, one may label the new vertices as $i_p$, where $i$ is the vertex of $D$ from which they originated, and $p\in [h(i)]$ is the relevant color. (We'll do this only sometimes.) Let there be an arc in $h*D$ from $i_p$ to $j_q$ if and only if there is an arc in $D$ from $i$ to $j$.

Warsaw graphs are the setting for \textit{admissible subgraphs.} For a plan (in practice, usually a strategy) $f$ on $D$, an admissible (sub)graph $S\subseteq D*h$ is 
\begin{itemize}
    \item a connected induced subgraph of $D*h$ such that
    \item for each $v\in D$, $V(A)$ contains at most one of $v_p,v_q\in [h(v)]$ for $p\neq q$ and
    \item if $\deg^+_S(v_p)=\deg^+_D(v)$, then $f_v(N^+(v))\neq v_p$.
\end{itemize}

 An \textit{admissible path} is an admissible subgraph that is a path. A \textit{quasi-admissible path} is a subgraph that is admissible except possibly for the failure of the third bullet point on the endpoints. Phrased thus, the mission of Szczechla was to show that $3*C_k$ was guaranteed an admissible $C_k$ (i.e., an admissible subgraph that is a $C_k$) if $k$ was greater than $4$ and not divisible by 3, but not otherwise. 

Intuitively, a vertex of $S$ is an actual (i.e. in some specific instance of the Adversary's choices) assignment of a color to a sage. Arcs of $S$ are actual beholdings. $S$ itself is a ``local,'' or ``partial'' reality––say, an intermediate stage while the Adversary is placing hats––when no sage has yet (correctly) decided her guess. (Note that a subgraph is admissible if and only if it's admissible on every subgraph induced on a set of the form $\{v\}\cup N^+(v)$. For undirected graphs, this is stars. For cycles, this is copies of $P_3$.)

\subsubsection*{Admissible continuations}
The visualization that a Warsaw graph affords us is quite nice in low-degree regions; it looks like weaving. See \cite{Szc17}.
The reasoning and computation itself is quite different. It's like threading several needles in succession, maneuvering a snake through rubble, or cracking a safe. It took me far more doodling, casework, and head-scratching than I'd care to admit before finding the arguments of this chapter. 

The main notion for building admissible paths is that of an \textit{admissible continuation.} We ape \cite{Szc17} directly. For the sake of focus on admissible paths, we will restrict our attention to an induced path $P$ in $D$. Label the vertices in this $n$-vertex path by $1,...,n$. We will speak as though the higher numbers were arrayed to the right.
Let $S$ be an admissible path in $h*D$ (really, in $h*P$) from $V_{i}$ to $V_{j}$ (left to right). A rightward admissible continuation is a vertex in $V_{j+1}$ that can be added to the path without causing to it to cease being admissible, or (equivalently) the edge between that vertex and the rightmost vertex of the $S$. A leftward admissible continuation is defined similarly, except left instead of right and $V_{i-1}$ instead of $V_{j+1}$. 
\begin{lmm}\label{admissiblecontinuationlemma}
    Given $r$ edges from $V_k$ to $V_{k+1}$ with common left end and a subset $V'\subseteq V_{k+2}$, the total number of admissible continuations of these edges into $V'$ is $\geq (r-g(v_{k+1}))|V'|$. 
    (If $r=|V_{k+1}|$, then equality holds.)
    In particular, at least one of the $r$ edges has $\geq |V'|-\left\lfloor \frac{g(v_{k+1})|V'|}{r}\right\rfloor$ admissible continuations. 
    The analogous fact holds from right to left. 
\end{lmm}
\begin{proof}
    Let $\beta\in V_k$ be the common left end. Let $\Gamma\subseteq V_{k+1}$ be the set of right ends. For each vertex $\delta\in V'$, at most $g(v_{k+1})$ of the vertices in $\Gamma$ are not the middle of an admissible $P_3$ with ends $\beta,\delta$. If $\Gamma=V_{k+1}$, then we may replace ``at most'' with ``exactly.'' Then we count.  
    The leftward proof is identical.
\end{proof}
The next fact expands on (a) and (b) of Lemma 2 in \cite{Szc17}. Notice that it works even if the various $V_k$ are replaced with some particular $V_k'\subset V_k$, so long as $|V_k'|\geq 3$ where needed.

\begin{corollary}\label{3PathLemma}
    For $j\geq i$. Let $P$ be an admissible path from $V_i$ to $V_j$ with $|V_{j+1}|,|V_{j+2}|\geq 3$ and $P$ having at least two admissible rightward continuations. Then there exists an admissible path $P'$ from $V_i$ to $V_{j+1}$ with at least two admissible rightward continuations. The analogous fact holds with directions reversed.
\end{corollary}

\section{Application to Low-Hatness Cycles}
Technically, we could reason all this out without the language of hints or admissible paths. However, those formalisms point to a relatively easy solution and help conceptualize it. 

We're only considering Latvian games, and we specify $h(Z)=2$, $h(\Omega)=4$. Games $\mathcal{P}_n^x$ and $\mathcal{C}_n^x$ are played on $P_n$ and $C_n$, respectively, with $h(X)=4$ and $h(v)=3$ for all vertices except $Z,\Omega,X$. Games $\mathcal{P}_n^y$ and $\mathcal{C}_n^y$ are played on $P_n$ and $C_n$, respectively, with $h(Y)=4$ and $h(v)=3$ for all vertices except $Z,\Omega,Y$. 

Thus, Lemma \ref{ConjectureSufficientForTwoToFour} rephrases as: $\mathcal{C}_n^x$ are $\mathcal{C}_n^y$ are unwinnable for any $n$. We've set up our notation suggestively: we'll delete $\Omega$ and show that $\mathcal{C}_n^x\backslash \Omega$-with-hint and $\mathcal{C}_n^y\backslash \Omega$-with-hint are unwinnable. This line of attack is enabled/suggested by Theorem \ref{FinalCzechHintTheorem}. (One could choose any vertex to delete, but $\Omega$ makes things simple.) We can accomplish things with a looser condition than Theorem \ref{FinalCzechHintTheorem}.

\begin{defn}\label{ExpedienceDefinition}
  If $f_0$ is a strategy on a Latvian path with $L(\{A,Z\},f_0,c)\subseteq \{c(Z)=1\}$ and $|L(\{A,Z\},f_0,c)|\leq 1$, call $f_0$ \textit{expedient.} If furthermore $f_1$ is a strategy that agrees with $f_0$ on $B...Y$ and has $L(\{A,Z\},f_1,c)\subseteq \{c(Z)=0\}$, we say that $f_0,f_1$ form an \textit{expedient pair}.
\end{defn}
\begin{lemma}\label{WhatSufficesToSuffice}
Let $\mathcal{P}$ be a game on a Latvian path with $h(A)=3$, $h(Z)=2$. Let $\mathcal{C}$ be the game obtained by attaching a vertex $\Omega$ to $A$ and $Z$ with $h(\Omega)\geq 4$. If $\mathcal{P}$ has no expedient pair, then $\mathcal{C}$ is unwinnable. In particular, to prove that $C_k^x,C_k^y$ are unwinnable for $k\geq 5$, it suffices to prove that $\mathcal{P}_n^x,\mathcal{P}_n^y$ have no expedient pair for $n\geq 4$. 
\end{lemma}
\begin{proof}
    Consider any partition of the $3\times 2$ rectangle $\mathbb{H}(\{A,Z\})$ into 4 parts. It has a part $S$ of size $\leq 1$. If there exist $f_{B...Y},\{f^i_{A,Z}\}$ satisfying Corollary \ref{LatvianHAH}.1, there must be some $i$ such that $L(\{A,Z\}, f_{B...Y} \cup f^i_{A,Z},c)\subseteq S$ with $S$ a singleton. If $S\subseteq \{c(Z)=1\}$, then $f_0:=f_{B...Y} \cup f^i_{A,Z}$ is expedient. If $S\subseteq \{c(Z)=0\}$, then the $f_0$ defined by 
    
    \[ f_{0v}=\begin{cases} 
        f_Y(c(X),1-c(Z)) & \text{ if } v=Y \\ 
        1-f^i_Z(c(Y)) & \text{ if } v=Z \\ 
        f_v & \text{ otherwise } \\ 
    \end{cases}\]
    is expedient. 
    
    By similar reasoning, there must be a part $S'$ in the opposite column from $S$, though it needn't be a singleton. We form $f_1$ from $f_{B...Y}\cup f_{A,Z}^j$ (for the $j$ pertaining to $S'$) in the same way $f_0$ is formed from $f\cup f_i$. 
    Ergo, if $\mathcal{C}_{n+1}^x$ and $\mathcal{C}_{n+1}^y$ are winnable, then there's an expedient pair for each of $\mathcal{P}_n^x=\mathcal{C}_{n+1}^x\backslash \Omega$, $\mathcal{P}_n^y=\mathcal{C}_{n+1}^y\backslash \Omega$. If there are no expedient pairs, they're unwinnable.
\end{proof}

\begin{lemma}\label{TrueForxandy}
Let $f$ be a strategy on $\mathcal{P}_n^x$ or $\mathcal{P}_n^y$.
    \begin{enumerate}
        \item There are distinct colors $p,q\in V_A$ such that $p$ and $q$ each have at least two colors in $V_B$ to which they can admissibly be extended.
        \item If there is an admissible path going left from $V_Z$ with 3 admissible leftward extensions, then $p$ and $q$ both extend to disprovers.
        \item If $f$ is expedient, the range of $f_Z$ is $V_Z$.
        \end{enumerate}
\end{lemma}
\begin{proof} If you find these arguments difficult to follow, I suggest you draw pictures.
   \begin{enumerate}
    \item There are nine total edges in $h*P_n$ between $V_A$ and $V_B$. For each vertex of $V_B$, one of these edges is forbidden by $f_A$.
    Delete these three edges. Six edges remain, so $\sum_{v\in V_A} \deg(v)\geq 6$. Also, $|V_A|=3$ and $\deg(v)\leq 3$ for all $v\in V_A$. Thus, there are at least two vertices in $V_A$ with degree $\geq 2$. (In the case of $P_4^x$, the numbers are different, but the conclusion holds.)
    \item Let $V'_v\subseteq V_v$ be the set of 3 elements into which the path has admissible leftward extensions. Then, by Lemma \ref{3PathLemma} and the above item, there is an admissible path originating at $p$ with two rightward extensions into $V'_v$, and similarly for $q$. Since $v$'s guess is already fixed, and the intersection of the two rightward extensions of the path originating from $p$ and the three leftward extensions of the path originating from $V_Z$ has order two, of which at most one is guessed by $v$, so at least one member of that intersection is a value for $c(v)$ that, together with these two paths, completes the disprover. 
    \item If $f$ is expedient and the range of $f_Z$ isn't $V_Z$, fix $c(Z)$ to be the color she never guesses. It suffices to show that $c_Z$ can be extended to disprovers $c,c'$ with $c(A)\neq c'(A)$; that will make $f$ fail Definition \ref{ExpedienceDefinition}. By the proof of \ref{prop:firstdeletable}, this amounts to finding disprovers on $P^x_n \backslash Z$ or $P^y_n \backslash Z$ that disagree on $A$, which is possible by e.g. Lemma 8 of \cite{BHKL09}. \qedhere
\end{enumerate} \end{proof}

\begin{lemma}\label{YNoExpedient}
    There is no expedient $f$ on $\mathcal{P}_n^y$.
\end{lemma}
\begin{proof}
    \textit{Claim 1.} If $f$ is an expedient strategy on $\mathcal{P}_n^y$, then $\mathbf{P}(f_Z(c(Y))=0)=3/4$. (I.e., under $f$, $Z$ guesses 0 for $3/4$ of the possible values of $c(Y)$. 
    
    \textit{Claim 2.} If a strategy $f$ on $\mathcal{P}_n^y$ has $\mathbf{P}(f_Z(c(Y))=0)=3/4$, then there exist distinct $p,q\in V(A)$ such that $(p,1)$ and $(q,1)$ are values of $(c(A),c(Z))$ extending to disprovers. 

    \textit{Conclusion.} By the claims and the third property of Definition \ref{ExpedienceDefinition}, we have ``If $f$ is expedient, then it is not expedient.'' Thus, there exists no expedient $f$ for $\mathcal{P}_n^y$.
    
    \textit{Proof of Claim 1.} We know by Lemma \ref{TrueForxandy}.2 that $\mathbf{P}(f_Z(c(Y))=0)<1$. Suppose for contradiction that $\geq 2$ colors (without loss of generality, $0$ and $1$) on $Y$ are mapped to $1$ by $f_Z$. So consider the edges $\overline{0_Y0_Z},\overline{1_Y0_Z}$ in $h*P_n$. By \ref{admissiblecontinuationlemma}, at least one of these two edges has $\geq 2$ admissible continuations into $X$. By \ref{3PathLemma}, we can extend to two admissible paths from $V_Z$ to $V_A$ that agree except in $V_A$.
    This gives us two colorings that agree on $B...Z$ and cause no sage but $A$ to guess correctly. $A$'s guess is fixed by $c(B)$, so in at most one of these colorings does $A$ guess correctly. Thus, there is a disprover $c$ with $c(Z)=0$, so by Definition \ref{ExpedienceDefinition}, $f$ is not expedient, a contradiction. Thus Claim 1 holds.

    \textit{Proof of Claim 2.} Let $V_Y'$ denote the subset of $V_Y$ such that $f_Z(v)=0$ for all $v\in V_Y'$. We have $|V_Y'|\geq 3$, so we finish by Lemma \ref{TrueForxandy}.2.
\end{proof}

\begin{lemma}\label{ExpedientPairMember0}
    Suppose $f_i$, $i\in \{0,1\}$  is a member of an expedient pair for $\mathcal{P}_n^x$. Then 
$\mathbf{P}({f_i}_Z(c(Y))=i)=2/3$. 
\end{lemma}
\begin{proof}
    If $\mathbf{P}(f_{iZ}(c(Y))=i)\leq 1/3$, then fix $c(Z)=i$. All vertices may assume that $f_{iZ}(c(Y))\neq i$, so there are $\geq 2$ remaining possible values of $c(Y)$. This reduces to the game on a path with hatnesses $(3,3,3,...,3,4,2)$, which we know is unwinnable by \cite{BHKL09}, \cite{BDFGM21}, or Theorem \ref{TreeAlgo2}.
      If $\mathbf{P}(f_{iZ}(c(Y))=i)=1$, then $1-i$ is never guessed, so when considering victory when $c(Z)=1-i$, we may as well delete $Z$. Then by e.g. \cite{BHKL09}, $f_i$ can lose for at least two values of $c(A)$, so $f_i$ is not part of an expedient pair. 
\end{proof}
\begin{lemma}\label{ExpedientPairMember1}
    Suppose $f_i$, $i\in \{0,1\}$  is a member of an expedient pair for $\mathcal{P}_n^x$. Then $\mathbf{P}({f_i}_Y(c(X),i)=a_i)=3/4$, where $a_i$ is the unique value such that  ${f_i}_Z(a_i)\neq i$. 
\end{lemma}
\begin{proof}
    Suppose that there are $\geq 2$ two values of $c(X)$ such that $f_{iY}(c(X),i)\neq a_i$. (The existence of $a_i$ follows from Lemma \ref{ExpedientPairMember0}.) Then the path in $h * P_n$ with vertices ${a_i}_Y,0_Z$ has $\geq 2$ admissible leftward continuations. Then apply Lemma \ref{3PathLemma} to find an admissible path from $Z$ to $B$ with two admissible continuations into $A$, at least one of which isn't $A$'s guess. So we have a disprover with $c(Z)=i$, a contradiction. 
\end{proof}

\begin{lmm}\label{CantHaveItBothWays}
    Let $f$ be a strategy on $\mathcal{P}_n^x$ such that $\mathbf{P}(f_Y(c(X),c(Z)= a_{c(Z)})=3/4)$, where $a_0,a_1$ are constants. Then there exist disprovers $c_0,c_1$ with $c_0(A)\neq c_1(A)$. 
\end{lmm}
\begin{proof}
    Set $c(Z)$ to be color $b$ is guessed $\leq 1/3$ of the time. There are at least two vertices of $V_Y$ we can choose without causing $Z$ to guess right. At least one of them is not $a_{c(Z)}$; choose that one and call it $d$. Then there are three colors $v_i\in V_X$ such that $(v_i,d,1)$ is admissible. Then we invoke Lemma \ref{TrueForxandy}.2. 
\end{proof}
\begin{lemma}\label{XNoExpedient}
    $P^x_n$ has no expedient pair. 
\end{lemma}
\begin{proof}
    Let $f_0,f_1$ be an expedient pair for $\mathcal{P}_n^x$. By Lemma \ref{ExpedientPairMember1}, we have $\mathbf{P}((f_Y(c(X),i))=a_{c(Z)})=3/4$. By Lemma \ref{CantHaveItBothWays}, there exist disprovers of $f_0$ that disagree on $A$, so $f_0$ is not expedient.
\end{proof}
We combine Lemmata \ref{XNoExpedient}, \ref{YNoExpedient}, and \ref{WhatSufficesToSuffice} to prove that $\mathcal{C}_n^x$ and $\mathcal{C}_n^y$ are unwinnable: i.e., to prove Lemma \ref{ConjectureSufficientForTwoToFour}. As stated above, this completes the proof of Theorem \ref{CategorizedTwoToFour}. 

\section{General Latvian cycles}
Here we prove Theorem \ref{CategorizedFiveAndUp} as follows. We establish the ``equivalently'' and show that the equivalent condition is implied by the unwinnability of ``bad cycles.'' Disprovers for bad cycles can be stitched together from certain families of quasi-admissible paths on ``atomic bad paths.'' Finally, we sketch the most tedious part: atomic bad paths have the desired families of quasi-admissible paths. 

We maintain our convention that $A$ is the leftmost vertex of a path, $B$ the next-leftmost, etc., $Z$ the rightmost, etc., regardless of the path's length. So if a path has five vertices, $C=X$, and so forth. This time, our talk of admissible paths will be flavored more in terms of admissible colorings than admissible graphs.   

\begin{lemma}\label{3.5.1equivalentlyfordeletable}
    Let $\mathcal{G}$ be a Latvian cycle with $v$ deletable. $\mathcal{G}$ is winnable if and only if it contains two vertices of hatness 2 between which neither $v$ nor any vertex of hatness $\geq 5$ lies.   
\end{lemma}
\begin{proof}
    This follows by the definition of ``deletable'' combined with Theorem \ref{thm: FirstTreeCategorization} as specialized to paths. 
\end{proof}
Let a \textit{bad cycle} be a Latvian cycle with $\forall v(h(v)\in \{2,3,5\})$ that contains exactly as many hatness-2 vertices as hatness-5 vertices and does not contain a winnable path subgraph.
\begin{lemma}\label{3.5.2badmonotone}
    Suppose all bad cycles are unwinnable. Then all Latvian cycles containing at least one vertex of hatness $\geq 5$ and no winnable path subgraph are unwinnable. 
\end{lemma}
\begin{proof}
    Consider such a Latvian cycle. Between every two vertices of hatness 2 (with no intervening vertex of hatness 2), pick exactly one vertex of hatness $\geq 5$ and mark it. Then lower the hatness of all marked vertices to $5$ and the hatness of all other non-hatness-2 vertices to 3. This is a bad cycle no harder than the original cycle; apply Proposition \ref{monotone}.
\end{proof}
Let an \textit{atomic bad path} be a Latvian path whose endpoints have hatness 2, with exactly one vertex of hatness $\geq 5$, and all other vertices of hatness $3$. 

\begin{lemma}\label{3.5.3atomicstitching}
    Suppose that for any strategy $f$ on an atomic bad path, there exists a partial coloring $c_{B,Y}$ of $B$ and $Y$ such that for any partial coloring $c_{A,Z}$ of $A$ and $Z$, $c_{B,Y}\cup c_{A,Z}$ extends to a quasi-admissible path. (Call such a $c_{B,Y}$ \textit{good}.) Then every bad cycle is unwinnable.
\end{lemma}
\begin{proof}
    Notice that bad cycles can be formed from atomic bad paths by gluing the ends together, i.e. by appropriately identifying vertices of hatness $2$. (If the bad cycle has but a single hatness-2 vertex, it is formed by identifying the endpoints of a single atomic bad path.) Now we construct a disprover for a given odd cycle using quasi-admissible paths. 
    Now fix the partial colorings $c_{B,Y}$ that are good in the atomic-bad-path subgraphs of the bad cycle. That is, we fix the color of all vertices of the bad cycle that are adjacent to hatness-2 vertices of the bad cycle. Thus, we've fixed a guess for each of the hatness-2 vertices. To each of those vertices, assign whatever color isn't being guessed. Relative to the atomic-bad-path subgraphs of the bad cycle, this is a choice of $c_{A,Z}$. By the lemma's assumption, the partial coloring chosen so far (which causes all of the hatness-2 vertices to guess wrong) can extend to a coloring that causes all other vertices to guess wrong, too. Thus we have a disprover for the bad cycle, proving the implication.
\end{proof}
Now comes the hard part: showing that the appropriate quasi-admissible paths exist. 

\begin{lmm}\label{23532lemma}
    For any plan $f_0$ on the middle three vertices of the Latvian path of hatnesses $(2,3,5,3,2)$, there exists a partial coloring $c_{B,Y}$ such that for any partial coloring $c_{A,Z}$, $c_{B,Y}\cup c_{A,Z}$ extends to a quasi-admissble path.
\end{lmm}
\begin{proof}[Proof sketch.] Casework can be somewhat mitigated by using symmetries and Lemma \ref{admissiblecontinuationlemma} appropriately, but it comes down to brute force. 
\end{proof}

\begin{lmm}\label{235lemma}
    Consider the Latvian game on $P_3$ with $c(A)=2$, $c(B)=3$, and $c(C)=5$. Fix a strategy $f$. There exists a color $p\in V_B$ such that for each color $q\in V_A$, there exists $\geq 3$ colors $r\in V_C$ such that $f_B(q,r)\neq p$. 
\end{lmm}
\begin{proof}[Proof sketch.]
    Essentially a counting argument in conjunction with Lemma \ref{admissiblecontinuationlemma}. 
\end{proof}

\begin{lmm}\label{endingwith2lemma}
    Consider a Latvian game on $P_k$ with $k\geq 4$, $c(Z)=2$, $c(Y)=5$, and $c(X)=3$. Fix a strategy $f$. Suppose there exist $\geq 3$ quasi-admissible paths $c_{A...Y}$ that are equal on $\{A,...,X\}$ but mutually unequal on $Y$. Then at least one of these extends to two quasi-admissible paths that agree on $\{A,...,Y\}$ and disagree on $Z$.
\end{lmm}
\begin{proof}
    Of the 3 different values considered for $c(Y)$ here, at least one of them is not guessed by $Y$ with our fixed $c(X)$ for either possible value of $c(Z)$. 
\end{proof}

\begin{lmm}\label{2335lemma}
    Consider a Latvian game on $P_k$ with $k\geq 4$ with $c(A)=2$, $c(Z)=5$, and $c(v)=3$ otherwise. There exists a color $p\in V_B$ such that for each color $q\in V_A$, there are distinct colors $r_1,r_2\in V_Y$ such that for every color $s\in V_Z$ there is a quasi-admissible path with $c(A)=q,c(B)=p,c(Y)\in \{r_1,r_2\}$ and $c(Z)=s$. (In particular, there is a color $r\in V_Y$ such that for $\geq 3$ colors $s\in V_Z$ there's a quasi-admissible path with $c(A)=q,c(B)=p,c(Y)=r$ and $c(Z)=s$.)
\end{lmm}
\begin{proof}[Proof sketch.]
    Applications of Lemmata \ref{admissiblecontinuationlemma} and \ref{3PathLemma}. 
\end{proof}
\begin{lmm}\label{interesectingat5lemma}
    Consider a Latvian game on a path of at least 6 vertices with $h(v)=5$ for some $v\notin\{A,B,Y,Z\}$. Suppose there are three quasi-admissible paths $c_{v...Z}$ from $v$ to $Z$ that all equal some $r_0$ on $v$'s rightward neighbor. Suppose also that there exist colors $r_1,r_2$ of $v$
    s leftward neighbor such that for every color $s\in V_v$, there's a quasi-admissible path $c_{A...v}$ whose value on $v$'s leftward neighbor is $r_1$ or $r_2$ and whose value on $v$ is $s$. Then of these, we can choose leftward and rightward quasi-admissible paths agreeing on $v$ whose union is quasi-admissible. 
\end{lmm}
\begin{proof}[Proof sketch.]
    Consider sets $\alpha,\beta,\gamma\subseteq \delta$ with $\alpha \cup \beta =\delta$, $|\gamma|\geq 3$, and $|\delta|=5$. Then $|\alpha \cap \gamma|\geq 2$ or $|\beta \cap \gamma| \geq 2$. 
\end{proof}
\begin{lmm}\label{thequasiadmissiblesexist}
    For any strategy $f$ on a basic path, there exists a partial coloring $c_{B,Y}$ of $B$ and $Y$ such that, for any partial coloring $c_{A,Z}$ of $A$ and $Z$, the partial coloring $c_{B,Z}\cup c_{A,Z}$ extends to a quasi-admissible path from $A$ to $Z$. 
    \end{lmm}
    \begin{proof}[Proof sketch.]
        We break it down by the hatness sequence of the path. For (2,5,2), it's obvious. For (2,3,5,3,2), it's Lemma \ref{23532lemma}. For (2,3,...,3,5,2), apply Lemmata \ref{2335lemma} and \ref{endingwith2lemma}. For (2,3,5,2), apply Lemmata \ref{235lemma} and \ref{endingwith2lemma}. For (2,3,...,3,5,3,2), apply Lemmata \ref{2335lemma}, \ref{235lemma}, and \ref{interesectingat5lemma}. For (2,3,...,3,5,3,...,3,2) apply Lemmata \ref{2335lemma} and \ref{interesectingat5lemma}. (Ellipses denote an arbitrary (possibly zero) number of threes.) This covers all cases up to reversal of direction.
    \end{proof}
Obviously, Lemmata \ref{3.5.1equivalentlyfordeletable}, \ref{3.5.2badmonotone}, \ref{3.5.3atomicstitching} and \ref{thequasiadmissiblesexist} prove Theorem \ref{CategorizedFiveAndUp}.

\chapter{Constructors}\label{ch:constructors}
\epigraph{When we build, \\let us think that we build forever.
}{John Ruskin}

Loosely, constructors are propositions of the following forms.
\begin{itemize}
    \item Let $\mathcal{G}$ be winnable. We alter it to yield $\mathcal{G}'$, which is winnable.
    \item Let $\mathcal{G}_1,...,\mathcal{G}_k$ be winnable. We combine them to yield $\mathcal{G}'$, which is winnable. 
    \item Either of the above, with ``unwinnable'' replacing ``winnable'' in each instance.
\end{itemize}
These propositions are typically constructive:\footnote{There's no way to avoid the pun. And besides, how come puns are the only stylistic feature you're supposed to apologize for? This is peculiar to contemporary English.} they build explicit winning strategies from earlier winning strategies, or build disprovers from earlier disprovers. All of ours are. Their algorithmic importance is clear; see Section \ref{OurApproach}. (Also, they're fun.) For their history, see Subsection \ref{Slavictechniques}.
Here we use a few simple constructors from \cite{KLR21b} to simplify parts of \cite{KL19} and (algorithmically) solve the Latvian game on trees, which also works as a relaxed introduction to the use of constructors.

Then we extend several specific constructors from \cite{KLR21b}––for instance, instead of ``adding a vertex of hatness 2 adjacent to two vertices and increasing their hatnesses by 1,'' we have a general constructor for altering hatnesses by attaching new vertices of arbitrary $g,h,$ and degree; it uses hats-as-hints. The changes are stronger if you attach a new vertex to a clique. \cite{KLR21b} wanted to attach paths but only managed with $P_2$; we show how to attach $P_n$'s, with a variety of attachment styles and hatness functions on the $P_n$. We also elucidate an implicit technique for attaching vertices of hatness 3. Finally, we examine the essential ``product'' or ``clique-join'' family of constructors, which have broadened and evolved from \cite{KLR21b} to \cite{BDO21} to \cite{LK23}. We broaden them further in three complicated directions.  
A few constructors emerge in other chapters from their special machinery––they happen to all be deletion constructors. See Corollary \ref{cor:redundantvisiondelete}, Theorem \ref{DeletableByPeerSets}, Corollary \ref{cor:deleteCzechleaves}, and Theorem \ref{ReplaceVertexWithArcs}.

Other constructors from \cite{KLR21b}––especially ``sewing'' (3.12), ``surgical intervention'' (3.13), ``fastening'' (3.14), ``cone'' (3.16), and ``multi-cone'' (3.16.1)––can be generalized to the Czech digraph case. However, this thesis is already long. See Subsection \ref{Constructorquestions}.

\textbf{Postscript.} Though we don't do anything with it, it's worth remarking that many constructor theorems can be understood as being about criticality––that is, as speaking of situations where a property is fragile to any slight changes. (For instance, Proposition \ref{Removing2Leaf} shows that no critical Latvian game has a leaf of hatness 3.) This is useful for describing the necessary structure of minimal (or maximal) counterexamples to a conjecture, thereby helping you find a counterexample or prove that none exists. The proof of the Strong Perfect Graph Theorem fruitfully used this apporach \cite{CRST09}, as did the proof of the four color theorem \cite{AH76}. 

\section{Applying extant constructors}

The first constructor in the first constructor paper \cite{KLR21b} was this product constructor; it and its generalizations have been the most important in subsequent research. (These constructors were originally stated for graphs, but we state them for digraphs.)

\begin{proposition}[Theorem 3.1 of \cite{KLR21b}]\label{SinglePointProduct}
Let $\mathcal{G}_1\equiv (D_1,h_1)$ and $\mathcal{G}_2\equiv (D_2,h_2)$ be two winnable Latvian games such that $V(D_1)\cap V(D_2)=\{a\}$. Now consider the Latvian game $\mathcal{G}\equiv (D_1 \cup D_2, h)$, where 

\[h(v)=
\begin{cases}
    h_1(v) \text{ for } v\in V(D_1)\backslash a \\
    h_2(v) \text{ for } v\in V(D_2)\backslash a \\
    h_1(v)h_2(v) \text{ for } v=a.
\end{cases}\]
$\mathcal{G}$ is winnable. 
\end{proposition}
\begin{proof}
    A special case of Proposition \ref{ManyCliquesProduct}, \ref{CliqueAndOneMoreThingProduct}, or \ref{TwoGeneralThingsProduct}. 
\end{proof}

\begin{corollary}[Corollary 3.1.1 of \cite{KLR21b}]\label{WinningTreeHatnesses}
    A Latvian game on a tree with $h(v)=2^{\deg(v)}$ is winnable. In particular, a Latvian path with hatnesses $(2,4,4,...,4,4,2)$ is winnable. 
\end{corollary}
\begin{proof}
    It follows from Proposition \ref{SinglePointProduct} and the winnableness of the Classic game $(P_2,2)$; construct the tree by adding one leaf at a time. 
\end{proof}

\subsection{Simplifying ``$\mu(G)\geq 3$ when connected $G$ has $\geq 2$ cycles''}\label{simplifyingclassic}
The hard part of \cite{KL19} was proving that every connected graph containing more than one cycle has $\mu$ at least 3. The proof required elaborate computational y. I haven't managed to simplify the cases in which the two cycles share $\geq 2$ vertices, but constructors from \cite{KLR21b} give stronger statements with less hassle in the cases of $0$ or $1$ shared vertex. 
\begin{prop}[Theorem 3.4 of \cite{KLR21b}]\label{AddHatness2}
    Consider a winnable Latvian game $\mathcal{G}\equiv (D,h)$. Consider $b,c\in V(D)$. Let $\mathcal{G}'$ be the game $(D',h')$, where $D'$ is derived from $D$ by adding vertex $a$ and edges $\overline{ab}, \overline{ac}$, and 
    \[h'(v)=
\begin{cases}
    h(v)+1 \text{ for } v\in \{b,c\}\\
    2 \text{ for } v=a \\
    h(v) \text{ otherwise. }
\end{cases}\]
Then $\mathcal{G}'$ is winnable. 
\end{prop}
\begin{proof}
    Special case of Proposition \ref{AddManyHatness2}. 
\end{proof}

\begin{corollary}[Corollary 3.4.1 of \cite{KLR21b}]\label{323Win}
    Consider a Latvian game $(C_k,h)$  with $a,b,c\in V(C_k)$, $a\sim b\sim c$. Set $h(a)=h(c)=3$, $h(b)=2$, and $h(v)=4$ on all other vertices $v$. That game is winnable. 
\end{corollary}
\begin{proof}
    Apply Proposition \ref{AddHatness2} to the path in Corollary \ref{WinningTreeHatnesses}.
\end{proof}

\begin{prop}
    Let $\mathcal{G}$ be a game on a ``kayak paddle graph'': i.e. two disjoint cycles and a path going from one cycle to the other. Let every vertex have hatness $4$ except for the four vertices that are part of a cycle and neighbor a degree-3 vertex; they have hatness $3$. $\mathcal{G}$ is winnable.

    Let $\mathcal{G}'$ be a game whose graph is two cycles that share exactly one vertex. Let every vertex have hatness $4$ except for the four vertices that neighbor a degree-4 vertex; they have hatness $3$. $\mathcal{G}'$ is winnable. 
\end{prop}
\begin{proof}
    In the first case, apply Proposition \ref{SinglePointProduct} with the degree-2 vertices of two cycles from Corollary \ref{323Win} and the endpoints of one path from Corollary \ref{WinningTreeHatnesses}. 
    In the second case, apply Proposition \ref{SinglePointProduct} with the degree-2 vertices of two cycles from Corollary \ref{323Win}. 
\end{proof}
So any connected graph with two edge-disjoint cycles has $\mu\geq 3$ and verging on $4$. This is quicker and stronger than the proofs in \cite{KL19}.

\subsection{Solving Latvian trees (take 1)}
Based on leaf deletion, we build an algorithm for Latvian trees (another one appears in Subsection \ref{subsec:SlickerTreeProof}). 

\begin{prop}\label{Removing2Leaf}
    A leaf of hatness $\geq 3$ is deletable in any Latvian game. 
    Consider a Latvian game $\mathcal{G}$ and a  leaf $\ell$ of hatness $2$. Let $u$ be the neighbor of $\ell$. Let $\mathcal{G}'$ be obtained by deleting $\ell$ and setting $h(u)$ to $\left\lceil \frac{h(u)}{2}\right\rceil$. $\mathcal{G}'$ has the same outcome as $\mathcal{G}$. 
\end{prop}
\begin{proof}
    The first sentence is simply Theorem 3.10 of \cite{KLR21b} (though it's also a special case of our Corollary \ref{cor:deleteCzechleaves}). For the rest, if $\mathcal{G}'$ is winnable, then by Proposition \ref{SinglePointProduct} with $(P_2,2)$, so is $\mathcal{G}''$, which is the same as $\mathcal{G}$ except $h''(u)=2  \left\lceil \frac{h(u)}{2}\right\rceil$. If $\mathcal{G}'$ is unwinnable, then Theorem 3.6 of \cite{KLR21b} (which follows from our Corollary \ref{LatvianHAH}.2) says that $\mathcal{G}'''$, which is the same as $\mathcal{G}$ except $h'''(u)=2  \left\lceil \frac{h(u)}{2}\right\rceil-1$, also is unwinnable. 
    Clearly, $\mathcal{G}''$ is no easier than $\mathcal{G}$, and $\mathcal{G}'''$ is no harder than $\mathcal{G}$. This means that when $\mathcal{G}'$ is (un)winnable, so is $\mathcal{G}$.  
\end{proof}
Together, these propositions mean that one never actually has to think about leaves in a Latvian game.  In particular, it solves trees, since they can be destroyed by successive leaf deletions. This gives us a nice polynomial algorithm,

Scan the leaves of $T$. Delete all leaves of hatness $\geq 3$. For every leaf of hatness $2$, delete it and alter its neighbor's hatness according to Proposition \ref{Removing2Leaf}. If that neighbor now has hatness $1$, halt;\footnote{We temporarily relax our convention against having any hatness-1 vertices.} the original game was winnable. If not, and there's only one vertex left, halt; the original game was unwinnable. Otherwise, repeat this paragraph with the remaining tree. 

\begin{thm}\label{thm: FirstTreeCategorization}
    This algorithm is correct. Moreover, it returns ``winnable'' if and only if $T$ contains a subtree $T'$ with $h(v)\leq 2^{\deg_{T'}(v)}$ for all $v\in V(T')$. 
\end{thm}
\begin{proof}
    As the algorithm runs, it preserves the game's outcome by Proposition \ref{Removing2Leaf}. Any game with a hatness-1 vertex is trivially winnable; any one-vertex game with hatness $>1$ is trivially unwinnable. 
    
    In the second part, ``if'' is easy to check. To get ``only if,'' assume we ended up at a hatness-1 vertex. Delete all other vertices, then build a game back up by reversing every Proposition \ref{Removing2Leaf} step. I.e., we add back all the vertices that we deleted when they had hatness $2$, adding them back successively by means of Proposition \ref{SinglePointProduct} and $(P_2,2)$. This is guaranteed to yield a tree game $(T',h')$ with $T'\subseteq T, h'(v)\geq h(v)$ and  $h'(v)=2^{\deg(v)}$, which completes the proof.
\end{proof}

\section{Generalizing some win-preserving constructors}

\subsection{Attaching a vertex}
We'd like to generalize Proposition \ref{AddHatness2}. 
\begin{proposition}\label{prop:GeneralAttaching}
    Suppose $f$ wins Czech game $\mathcal{G}\equiv (D,g,h)$. Specify a vertex set $Y=\{v_1,...v_n\}\subseteq V(D)$ and nonnegative integers $\{q_1,...,q_n\}, \{k_1,...,k_n\}$. Let $D'$ be $D$ with an additional vertex $x$ universal to the members of $Y$. Define $\mathcal{G'}\equiv (D', g',h')$ with $g'(v_i)=g(v_i)+q_i$ and $h'(v_i)=h(v_i)+k_i$ for all $v_i\in Y$, and $g'=g,h'=h$ on other vertices of $D$. 

    If $g(v_i)+q_i \geq (h'(x)-g'(x)) \lceil \frac{k_i}{h'(x)-1} \rceil$ and $q_i \geq (h'(x)-g'(x)-1)\lceil \frac{k_i}{h'(x)-1}\rceil$ for all $i$, then $f|_{G\backslash Y}$ extends to a winning strategy on $\mathcal{G}'$. 
\end{proposition}
\begin{proof}
    We use the hats-as-hints approach. For $n\in [h(x)]$, let $U_n$ denote $\prod [h(v_i) + \lfloor \frac{nh(v_i)}{h(x)-1} \rfloor ]$. Let $P_n$ equal $(U_n \backslash U_{n-1}) \cap \prod [h'(v_i)]$. (So $P_0=U_0=\mathbb{H}_h(Y)$, where $\mathbb{H}_h(Y)$ is specifically $\prod_{v\in Y} [h(v)]$.) From this $P_0,...,P_{h(x)-1}$, define $A_0,...,A_{h(x)-1}\subseteq \mathbb{H}(Y)$ according to Remark \ref{RemarkAboutSimpleHints} with $j=g(x)$ and $k=h(x)$. Let that be our $g(x),h(x)$-hint with respect to $Y,Y$. 
    Now we define the plan $f^n_Y$ (on $Y$) for each utterance ``$\neg A_n$.'' 
    
    If $P_0\subseteq \mathbb{H}_{h'}(Y)\backslash A_n$, then each $v_i$ uses her first $g(v_i)$ guesses according to $f$.  If $c(Y)\in \mathbb{H}_h(Y)$, it's correctly guessed by one of those ``original guesses'' because they win $\mathcal{G}$. 
    She uses her remaining $q_i$ guesses to guess those colors $d>h(v_i)$ such that $(0,0,...,d,...,0,0)\not\in A_n$, where $d$ is in the $i$'th place. By the definition of $A_n$, there are at most $(h(x)-g(x)-1)\lceil \frac{k_i}{h(x)-1}\rceil$ such colors for a given $i$, so (by the inequality on $q_i$) $v_i$ can guess all of them at once. By definition of these colors and $A_n$, there is no vector in $\mathbb{H}_{h'}(Y)\backslash \mathbb{H}_{h}(Y) \backslash A_n$ that doesn't have such a color in some coordinate. 
    We have $c(Y)\not \in A_n$, so if $c(Y)\not\in \mathbb{H}(Y)$, then some coordinate of it is correctly guessed by some ``extra guess,'' i.e. a guess afforded by the $+q_i$.

    If $P_0 \subseteq A_n$, then each each $v_i$ uses her $g(v_i)+q_i$ guesses to guess those colors $d$ such that $(0,0,...,d,...,0,0)\not\in A_n$. By the definition of $A_n$, there are at most $(h(x)-g(x)) \lceil \frac{k_i}{h(x)-1} \rceil$ such colors for a given $i$, so (by the inequality on $g(v_i)+q_i \equiv g'(v_i)$), $v_i$ can guess all of them at once. By definition of these colors and $A_n$, there is no vector in $\mathbb{H}_{h'}(Y)\backslash A_n$ that doesn't have such a color in some coordinate. So since $c(Y)\not\in A_n$, it's correctly guessed by some $v_i\in Y$.

    Thus, for any utterance ``$\neg A_n$,'' there's a winning strategy for ``$\mathcal{G}$-but-with-this-hint'' extending $f_{G\backslash Y}$. By \ref{HeadlineHints}.7, there's a winning strategy for $\mathcal{G}'$ extending $f_{D\backslash Y}$. 
\end{proof}

\begin{corollary}\label{AddManyHatness2}
    Let $\mathcal{G}$ be a winnable Czech game with $Y\subseteq V(D)$. Attach vertex $x$ universal to $Y$ with $g(x)=h(x)-1$. Let $\mathcal{G}'$ be obtained by replacing $g$ and $h$ for $Y$ with $g',h'$. If $(h'(v)-g'(v))-(h(v)-g(v))\leq g(v)g(x)$ for all $v\in Y$, then $\mathcal{G}'$ is winnable. In particular, if $h(x)=2$, we may set $h'(v)=h(v)+g(v)$ and $g'(v)=g(v)$. If $\mathcal{G}$ is Latvian, this is merely the alteration $h'(v)=h(v)+1$. 
\end{corollary}

If the vertex set to which you're attaching this vertex is a clique, you can do much better. Theorem 3.7 of \cite{KLR21b} says that if you attach a vertex of hatness 2 to vertices $b,c$ with $b\sim c$, you can double the hatnesses of $b$ and $c$ without losing winnableness. Let's generalize that. 

\begin{proposition}
    Suppose $f$ wins Czech game $\mathcal{G}\equiv (D,g,h)$, and specify a clique $Y=\{v_1,...v_n\}\subseteq V(D)$. Consider $\mathcal{G}'\equiv (D',g',h')$, where $D'$ is $D$ with a new vertex $x$ universal to $Y$, $g'(v_i)=g(v_i)(h(x)-g(x))$ and $h'(v_i)=h(v_i)h(x)$ for each $v_i\in Y$ and $h'=h,g'=g$ for all other vertices.  $f_{D\backslash Y}$ extends to win  $\mathcal{G}'$. 
\end{proposition}
\begin{proof}
    Consider the game on $P_2$ having vertices $x,z$. Set $h(z)=h(x)$ and $g(z)=h(x)-g(x)$. This is clearly winnable. Then apply \ref{CliqueAndOneMoreThingProduct} or \ref{ManyCliquesProduct} to the clique $Y$ and the vertex $Z$. (Really, Lemma 7 from \cite{BDO21} suffices.) 
\end{proof}

\begin{corollary}\label{cor:AddNearlyFullVertex}
    Suppose $f$ wins $\mathcal{G}$. Add vertex $x$ universal to clique $Y\subseteq V(D)$, with $g(x)=h(x)-1$. Multiply the hatness of each $v\in Y$ by $h(x)$. Call the resulting game $\mathcal{G}'$. Then $f_{D\backslash Y}$ extends to win $\mathcal{G}'$. 

    Suppose $f$ wins Latvian $\mathcal{G}$. Add a vertex $x$ universal to a clique $Y\subseteq V(D)$ with $h(x)=2$. Double the hatness of every vertex in $Y$. Call the resultant game $\mathcal{G}'$. $f_{D\backslash Y}$ extends to win $\mathcal{G}'$. 
\end{corollary}

\subsection{Attaching a path}
Here's \cite{KLR21b}'s quirky Theorem 3.8. 
Take a winnable game $\mathcal{G}$ and vertices $z,y$. Add new vertices $a,b$, with $a\sim z, a\sim b, b\sim y$. Set $h(a)=2$, $h(b)=3$. Double $h(z)$, and increase $h(y)$ by one. The resultant game is winnable.  ``Apparently it is hard to determine whether the graph obtained by attaching a new fragment via two independent `jumpers', is winning," they write. ``We are able to do this for very small fragment only'' [sic]. For Latvian games, however, we give versatile constructors of similar form: attach a path that links some vertices, and alter the hatnesses of those vertices. First, we explicitate and generalize an argument from \cite{KLR21b}.

\begin{proposition}\label{AttachHatness3}
    Suppose $f$ wins $\mathcal{G}$. Attach a vertex $x$ with $h(x)=3$ and $x$ complete to $Y\subseteq V(D)$, where $Y$ contains a vertex $b$ of hatness 2. For every other vertex in $Y$, increase its hatness by its guessness. $f$ extends to win the resultant game. 
\end{proposition}
\begin{proof}
We may assume $g(x)=1$. Extend $f$ as follows. If $x$ sees a new color on any member of $Y\backslash b$, she guesses $2$. Otherwise, she guesses $1-c(b)$. $b$ guesses $c(x)$, unless $c(x)=2$, in which case she guesses according to $f$. Each member of $Y\backslash b$ guesses her new colors unless $c(x)=2$, in which case she guesses according to $f$.

Assume $c(Y\backslash b)$ contains some new color: if $c(x)=2$, then $x$ guesses right; if $c(A)\neq 2$, then someone in $Y\backslash b$ guesses right. Now assume $c(Y\backslash b)$ contains no new colors. If $c(x)=2$, then everyone else plays according to $f$ and wins, since no new colors are being used. If $c(x)\neq 2$, then $b$ guesses $c(x)$ while $x$ guesses $1-c(b)$, so either $b$ or $x$ guesses right. 
\end{proof}

\begin{proposition}\label{WeakPath}
  Suppose $f$ wins $\mathcal{G}$. Let $a_1,a_2,...a_n$ be a path of hatnesses $2,3,4,4,...,4$ respectively or $3,2,4,4,...,4$ respectively, where an ellipsis indicates that the number of fours is arbitrary.\footnote{I.e., we can have paths (2,3) or (3,2).} 
  Attach $a_1$ to some $X\subseteq V(G)$, and attach $a_n$ to some $Z\subseteq V(G)$. (``Attach'' means ``make universal to.'') If $Z$ is a clique, double its hatnesses. Otherwise, increase its hatnesses by its guessnesses. Increase the hatnesses of $X$ by its guessnesses. $f$ extends to win the resultant game. 
\end{proposition}
\begin{proof}
    If the path is $P_2$, you can simply apply some combination of \ref{AddManyHatness2}/\ref{AttachHatness3}/\ref{cor:AddNearlyFullVertex}. Now assume the path is longer. Add a hatness-2 vertex to $Z$, applying Corollary \ref{AddManyHatness2} or Corollary \ref{cor:AddNearlyFullVertex} if $Z$ is a clique. Let that new vertex be $a_n$. Attach a hatness-2 leaf to $a_n$, altering $h(a_n)$ according to Corollary \ref{cor:AddNearlyFullVertex}. That new leaf is $a_{n-1}$. Continue this process until you add $a_2$. If you want a path of the first type, apply Corollary \ref{AddManyHatness2} with $x=a_1$, $Y=\{a_2\}\cup X$. If you want a path of the second type, apply Proposition \ref{AttachHatness3} with $x=a_1, Y\backslash b=X,b=a_2$.
\end{proof}
\begin{proposition}\label{StrongPath}
    Suppose $f$ wins Latvian game $\mathcal{G}$. Let $a_1,a_2,...a_n$ be a path of hatnesses $4,...,3,2,3,...,4$ or $4,...,2,3,3,...,4$, where an ellipsis indicates that the number of fours is arbitrary.\footnote{We require at least one on each side, since otherwise it reduces to Proposition \ref{WeakPath}.} Attach $a_1$ to some $X\subseteq V(D)$, and attach $a_n$ to some $Z\subseteq V(D)$. Increase the hatnesses of $X$ by its guessnesses; if $X$ is a clique, double the hatnesses instead. Do the same for $Z$.  
    $f$ extends to win the resultant game. 
\end{proposition}
\begin{proof}
    It's similar to the proof of \ref{WeakPath}. Add hatness-2 vertices $a_1$ and $a_n$ via \ref{AddManyHatness2} or \ref{cor:AddNearlyFullVertex} as appropriate. Add hatness-2 leaves via \ref{cor:AddNearlyFullVertex} as appropriate, working inward. For a path of the first type, finish with \ref{AddManyHatness2}. For a path of the second type, finish with \ref{AttachHatness3}.
\end{proof}
One can also devise versions of these constructors where the paths themselves are Czech, using \ref{AddManyHatness2}, \ref{cor:AddNearlyFullVertex}, and \ref{AttachHatness3} more liberally.

\subsection{Products of games}
From \cite{KLR21b} through \cite{BDO21,LK21,LK22} to \cite{LK23}, product constructors have been consistently and deeply useful. Here we extend that common reasoning, maybe to the limits of its usefulness.

\begin{defn}
    Let $D_1$, $D_2$ be digraphs with $H_1\subseteq V(D_1), H_2\subseteq V(D_2)$. Define a new digraph $D$ with $V(D)= (V(D_1) \backslash H_1) \coprod (V(D_2) \backslash H_2) \coprod (H_1 \times H_2)$. For $v\in V(D)$, let $\pi_i(v)$ denote the projection of $v$ to its $i$th coordinate if $v \in H_1 \times H_2$. If 
    $v\in V(D_j) \backslash H_j$, let $\pi_i(v)=v$ if $i=j$, and let it be undefined otherwise. Consider $u\in (V(D_j)\backslash H_i)$. We stipulate $\overrightarrow{uv}\in E(D)\iff \overrightarrow{u\pi_j(v)}\in E(D_j)$ and $\overrightarrow{vu}\in E(D)\iff \overrightarrow{\pi_j(v)u}\in E(D_j)$. 

    Now, suppose $u,v\in H_1\times H_2$. There is an arrow from $u$ to $v$ if and only if, for (exactly/at least) one $i$, we have $\overrightarrow{\pi_i(u)\pi_i(v)}\in E(D_i)$, and $\pi_j(u)=\pi_j(v)$ for $j$ not satisfying that condition. If we say ``exactly,'' we call the resulting digraph the \emph{Cartesian product of $D_1$ and $D_2$ with respect to $H_1$ and $H_2$} and denote it $D=D_1 \square_{H_1,H_2} D_2$. If we say ``at least,'' we call the resulting digraph the \emph{strong product of $D_1$ and $D_2$ with respect to $H_1$ and $H_2$} and denote it $D=D_1 \boxtimes_{H_1,H_2} D_2$. We extend these binary products to $n$-ary thus. 
    
    \[\bigsquare^{n}_{i=1} (D_i,H_i):= \;D_n\; \square_{H_n, \prod_{i=1}^{n-1} H_i}\; \bigsquare^{n-1}_{i=1} (D_i,H_i)\]
    
    \[\bigxsquare^{n}_{i=1} (D_i,H_i):= D_n \boxtimes_{H_n, \prod_{i=1}^{n-1} H_i} \bigxsquare^{n-1}_{i=1} (D_i,H_i)\]

    Notice that $\square$ and $\boxtimes$ are associative and commutative, so these are well-defined.
\end{defn}

The only constructor \cite{BDO21} used was a generalization of Proposition \ref{SinglePointProduct} that allowed you to join one vertex of a graph to a clique of another, for undirected Czech games. We extend it even further.\footnote{We don't include the ``reduction'' introduced in \cite{LK23}. Although it seems compatible, it would complicate already unwieldy statements.} We allow digraphs, and instead of taking the product of two games with respect to a clique and a vertex, we'll consider variations where we're taking the product of several graphs with respect to subgraphs, some of which may be cliques. It seems (unsurprisingly) that cliques are much ``easier'' than general subgraphs in two senses. You don't need requirements on the winning behavior of the subgraph (compare Propositions \ref{CliqueAndOneMoreThingProduct} and \ref{TwoGeneralThingsProduct}), and you don't need to loosen the guessness on the vertices in $D_i \backslash H_i$. Single points are easier still because they don't cause $\prod H_i$ to grow. The \cite{BDO21} version remains easier to work with, but we think our statements can be useful. If nothing else, we've pushed against the contours and requirements of graph-product reasoning. Also, the contrast between cliques and other digraphs is a counterintuitive contribution to the contrast between linear and general strategies: in these constructions, linearity is associated with ease and efficacy, though linear strategies can drastically underperform general ones (e.g. \cite{ABST20}).

The proofs all use the same idea, which is most easily seen in the two-game case. We join $\mathcal{G}_1$ to $\mathcal{G}_2$ by multiplying specified subgraphs in some way. We think of every ``middle vertex'' $v$ as having $h_1(v)h_2(v)$ composite colors, and they play in a ``split-brain'' fashion, guessing their ``left coordinate'' according to $f_1$ and guessing their ``right coordinate'' according to $f_2$, where $f_i$ wins $\mathcal{G}_i$. Vertices on the left only heed the left coordinate of composite colors, and vertices on the right only heed the right. They guess according to $f_1$ and $f_2$ respectively, sometimes with additional guesses to compensate for additional uncertainty in the middle. 


\subsubsection*{The constructors themselves}
The most obvious extension is the clique join with respect to several cliques. We use the same idea as in \cite{KLR21b}, \cite{BDO21}, or Proposition \ref{CzechOnComplete}, where cliques are thought of as having a collective color obtained by summing. 

\begin{prop}\label{ManyCliquesProduct}
    Consider general games $\mathcal{G}_i \equiv (D_i, g_i, h_i)$ with $H_i\subseteq V(D_i)$ inducing cliques in each. Let $f_i$ be winning strategies on each $\mathcal{G}_i$. Then the $f_i$ combine to a winning strategy $f$ on the game $\mathcal{G}=(D,g,h)$, where
     \[ D= \bigxsquare^{n}_{i=1} (D_i,H_i)\]  
 \[ g(v)=\begin{cases}
    g_i(v) &\text{ if } v\in V(D_i)\backslash H_i  \\
    \prod_{i=1}^n g_i(\pi_i(v)) &\text{ if } v\in \prod_{i=1}^n H_i 
    \end{cases}\]
\[h(v)=\begin{cases}
     h_i(v) &\text{ if } v\in V(D_i)\backslash H_i \\
    \prod_{i=1}^n h_i(\pi_i(v)) &\text{ if } v\in \prod_{i=1}^n H_i
    \end{cases}.\]
\end{prop}
\begin{proof}We'll treat any coloring $c$ on $D$ as arising from/inducing an $n$-tuple $(c_1,...,c_n)$ of colorings on each $D_i$. This association is clearly bijective. In particular, each sage treats a color $c(a)$ on $a\in \prod_{i=1}^n H_i$ as a \textit{composite color} $(c_1(\pi_1(a)),...,c_n(\pi_n(a)))$.
    Sages in $V(D_i)\backslash H_i$ won't vary their guess as the various visible $c_j$'s vary for $i\neq j$. For each $H_i$ and each $u\in H_i$ we define an \emph{imaginary color} $s_u:=\left( \sum_{\{u\}\times \prod_{j\neq i} H_j} c_i(u)\right)\mod h_i(u)$. By the definition of $D$, $v$ can discern $s_u$ whenever $\overrightarrow{\pi_i(v)u}\in E(D_i)$.
    
    Each $v\in V(D_i) \backslash H_i$ guesses according to $f_i$, except that she feeds $s_u$ into the input for $c_i(u)$ for $u\in H_i$ and feeds $c(u)=c_i(\pi_i(u))=c_i(u)$ for $u\in V(D_i)\backslash H_i$. 

    Each $a\in \prod_{i=1}^n H_i$ plays in a ``split-brain'' fashion. She guesses her composite color  by deciding the guess for each $c_i(\pi_i(a))$ separately for each $i$ as follows. She imagines she is simply in $D_i$. Like the vertices in $D_i\backslash H_i$, she reads off the imaginary colors $s_v$ for each $v\in H_i$ that she would have been able to see in $D_i$, and for the hats she can see in $D_i \backslash H_i$, she simply reads off their colors. She runs $f_i$ on those inputs, but instead of making that result her raw guess, she makes that her assumption for the value of $s_u$, where $u\in H_i$ is such that $a\in \{u\} \times \prod_{j\neq i} H_j$. Then she guesses $c_i(a)$ accordingly, since she can see the $c_i$ for all other vertices in  $\{u\} \times \prod_{j\neq i} H_j$. (Notice that she guesses the Cartesian product of the guess-sets dictated by each $f_i$ for the various $\pi_i(a)$; this is why we must have $g(a)=\prod g(\pi_i(a))$.)

    Because each $f_i$ won $\mathcal{G}_i$, either some vertex of $D_i\backslash H_i$ has guessed right for some $i$ (in which case $\mathcal{G}$ has been won), or for every $i$, there is some set of the form $H_1 \times ... \times H_{i-1} \times \{u\} \times H_{i+1} ... \times H_n$ all of whose members correctly guess $s_u$ and therefore correctly guess their own $c_i$. We need only take the intersection of all such sets to find a sage who guesses every coordinate of her composite color correctly, and therefore guesses correctly overall. Thus, $f$ wins. 
\end{proof}

We can start to incorporate general digraphs instead of cliques by replacing the single point in \cite{BDO21}'s formulation. So we multiply two digraphs with respect to any subgraph in one and a clique in another.
\begin{prop}\label{CliqueAndOneMoreThingProduct}
    Consider general games $\mathcal{G}_1\equiv (D_1, g_1, h_1)$, $\mathcal{G}_2\equiv(D_2, g_2, h_2)$  with $H_1\subseteq V(D_1)$ inducing a clique and $H_2\subseteq V(D_2)$ arbitrary. Winning strategies $f_1,f_2$ for $\mathcal{G}_1, \mathcal{G}_2$ combine to yield a winning strategy $f$ on $\mathcal{G}\equiv (D,g,h)$, where

    \[D=D_1 \square_{H_1,H_2} D_2\]

    \[g(v)=\begin{cases}
        |H_2|\cdot g_1(v) &\text{ if } v\in V(D_1)\backslash H_1 \text{ and } \exists u\in H_1(\overrightarrow{vu}\in E(D_1)) \\
         g_1(v) &\text{ if } v\in V(D_1)\backslash H_1 \text{ and } \nexists u\in S(\overrightarrow{vu}\in E(D_1)) \\
        g_1(\pi_1(v)) \cdot g_2(\pi_2(v)) &\text{ if } v\in H_1\times H_2\\
        g_2(v) &\text{ if } v\in V(D_2) \backslash H_2
        \end{cases}\]

    \[h(v)=\begin{cases}
         h_1(v) &\text{ if } v\in V(D_1)\backslash H_1 \\
        h_1(\pi_1(v)) \cdot h_2(\pi_2(v)) &\text{ if } v\in H_1\times H_2\\
        h_2(v) &\text{ if } v\in V(D_2) \backslash H_2
        \end{cases}\]\end{prop}

\begin{proof}
    Again, we describe $f$, using composite color $c(a)=(c_1(\pi_1(a)), c_2(\pi_2(a)))$. For each $u\in H_2$ whom $v$ sees in $D_2$, we define imaginary color $s_u:=\left( \sum_{H_1 \times \{u\}} c_2(u)\right)\mod h_2(u)$. Each $v\in V(D_1)\backslash H_1$ guesses according to $f_1$ as follows. If in $D_1$ she sees some $H'_1\subseteq H_1$, she sees $|H_2|$ copies of $H'_1$ in $G$. She gives one guess-set for each such copy, as if she were in $D_1$, seeing only that copy and only the ``left part'' of the composite colors. (If there are guesses left over, they may be used however.) Each $v\in V(D_2) \backslash H_2$ calculates the imaginary colors she can see and guesses as though she were playing $f_2$ in $D_1$ seeing color $s_u$ on $u$.  

    As before, $a\in H_1\times H_2$ guesses $c(a)$ by guessing $ c_1(\pi_1(a)), c_2(\pi_2(a))$ separately. She guesses $c_1$ in accordance with $f_1$, seeing only the ``left part'' of anyone's colors, ignoring all vertices of $H_1 \times H_2$ except those in her own copy of $H_1$. For $c_2$, she plays according to $f_2$, except that her inputs are $s_u$ where appropriate and her output is a guess for $s_{\pi_2(a)}$, which she then translates into a guess for $c_2(a)$ by linearity, since she can see all other vertices in $H_1 \times \{u\}$. 

    Assume that no sage in $D_1\backslash H_1$ guesses right. Because $f_1$ wins (and the sages seeing $H_1\times H_2$ get $|H_2|$ tries), we have that for every coloring and for each set of the form $H_1 \times \{v\}$ for $v\in H_2$, there's a vertex $a\in \{H_1\} \times v$ that guesses $c_1(\pi_1(a))$ correctly. Assume that no sage in $D_2 \backslash H_2$ wins. Because $f_2$ wins, we have the for a given coloring, some set  $S$ of the form $H_1 \times \{v\}$ for $v\in H_2$, every vertex in $S$ guesses right. Combining these, we have that either a vertex in $V(D_1)\backslash H_1$ wins, a vertex $V(D_2)\backslash H_2$ wins, or a vertex in $H_1 \times H_2$ guesses both her $c_1$ and $c_2$ correctly and therefore wins. 
\end{proof}

If we want to use \textit{only} general subgraphs, we can no longer use the main virtue of cliques––that in each clique, someone will guess a given side of her composite color correctly. Thus, we have to rely on brute numbers to guarantee that some correct left guess intersects with some correct right guess.\footnote{This directs our attention to a theme we (and most recent writers on deterministic games) have ignored: \textit{how many} correct guesses can we guarantee? See Question \ref{ques:howmanycanweget}.} (There are other conditions one could imagine that guarantee an intersection, but this is the simplest and can easily be replaced with another.) 

\begin{prop}\label{TwoGeneralThingsProduct}
    Consider Czech games $\mathcal{G}_i\equiv (D_i, g_i, h_i)$ with $H_i\subseteq V(D_i)$. Let $f_i$ be strategies on each $\mathcal{G}_i$ such that, if every sage in $V(D_i)\backslash H_i$ guesses wrong, then at most $r_i$ members of $H_i$ guess wrong. Presume that $\sum_{i=1}^n \frac{r_i}{|H_i|}<1$. Then the $f_i$ extend to a winning strategy on the game $\mathcal{G}\equiv (G,g,h)$, where 

    \[D=\bigsquare^{n}_{i=1} (D_i,H_i)\]
   \[ g(v)=\begin{cases}
        \prod_{j\neq i}|H_j| \cdot g_i(v) &\text{ if } v\in V(D_i)\backslash H_i \text{ and } \exists u\in \prod_{i=1}^n H_i (\overrightarrow{vu}\in E(G)) \\
         g_i(v) &\text{ if } v\in V(D_i)\backslash H_i \text{ and } \nexists u\in \prod_{i=1}^n H_i (\overrightarrow{vu}\in E(G)) \\
        \prod_{i=1}^n g_i(\pi_i(v)) &\text{ if } v\in \prod_{i=1}^n H_i 
        \end{cases}\]
    \[h(v)=\begin{cases}
         h_i(v) &\text{ if } v\in V(D_i)\backslash H_i \\
        \prod_{i=1}^n h_i(\pi_i(v)) &\text{ if } v\in \prod_{i=1}^n H_i
        \end{cases}.\]
\end{prop}

\begin{proof} Again, we construct $f$, considering composite colors. Each $v\in V(D_i) \backslash H_i$ guesses according to $f_i$. (If in $D_i$ she sees nonempty $H'_i \subseteq H_i$, then  she sees $\prod_{j\neq i}|H_i|$ copies of $H'_i$. She makes one guess-set for each such copy, perceiving only the $c_i$ part of the composite colors. Each $a=(a_1,...,a_n)\in \prod_{i=1}^n H_i$ guesses her composite color by guessing each separate $c_i$ in accordance with $f_i$, as if she were in $D_i$. For these purposes, she perceives only the values of $c$ among the visible members of $V(D_i) \backslash H_i$ and the values of $c_i$ among the visible members of $\prod_{i=1}^n H_i$. (Notice that by definition, if $a_i \rightarrow b_i$ in $D_i$, we have a vertex of the form $b=(a_1,...,a_{i-1},b_i,a_{i+1},...,a_n)$ with $a\rightarrow b$. So the plan can be executed.)

    Now presume that no vertex of $D_i\backslash H_i$ wins in the game $\mathcal{G}$. Then at most $r_i\prod_{j\neq i} |H_i|$ of the vertices in  in $\prod_{i=1}^n H_i$ guess their $c_i$ incorrectly. By the problem's condition that $\sum_{i=1}^n \frac{r_i}{|H_i|}<1$, someone in $\prod_{i=1}^n H_i$ guesses all her $c_i$ correctly, so $f$ wins $\mathcal{G}$.
\end{proof}
We hope that these three propositions and proofs give some intuition for how this works, enabling the reader to combine them and extend the logic if she so desires. 

\chapter{Extending plans to metapeer and independent sets}

Say you're building a strategy, and so far you have a plan $f_0$ for everyone except for the sages in some independent set $S$. The task remaining to $S$ has a nice geometric interpretation. $\mathbb{H}(S)$, the possible colorings of $S$, form a combinatorial prism of measurements $\{h(v)\}$ for $v\in S$. Each dimension corresponds to one sage. A single guess of a single sage corresponds to a hyperplane perpendicular to that sage's axis. The sages in $S$ must coordinate to cover $L'(S,f_0,c)$ with their guesses. It's especially nice when $S$ is metapeer; we can treat it as a single reasoning entity. 

This chapter mostly considers the situation of a metapeer set lacking a plan, though this has implications for an independent set lacking a plan. We'll discuss how the size and shape of $L'(S,f_0,c)$ determine whether $f_0$ can extend to win. We'll discuss which additional vertices, universal to $S$, are helpful or not. Some constructors result: vertices with $N(v)$ independent and $r(v)$ too low can be replaced by enough arcs to make $N(v)$ metapeer; certain Czech leaves can be deleted wholesale; certain cut-vertices with low $r(v)$ can too. These give an even slicker algorithm for Latvian trees (Subsection \ref{subsec:SlickerTreeProof}).

Also, every vertex's neighborhood is metapeer in complete bipartite graphs, which have attracted attention in nearly every deterministic hat guessing paper. Using our techniques, we can completely characterize the winnable Czech games on complete bipartite graphs––the caveat being that our characterization has unknown time-complexity, even for $K_{1,2}$!

\section{Covering and winning}
\epigraph{I don't give a ––– what happens. I want you all to stonewall it, let them plead the Fifth Amendment, cover up or anything else, if it'll save it, save this plan. That's the whole point.}{Richard Nixon}

\subsection{Planes and sprawls}
Let's codify our intuitions about covering. Let $A$ be a $d$-dimensional vector of positive integers $(a_1,...,a_d)$.
Consider $Q\subseteq \mathbb{Z}^d$. An \emph{$A$-cover of $Q$} is a list $B\equiv \{b_{ij}, 1\leq i \leq d, 1\leq j \leq a_i\}$ of integers such that, for every vector $q\in Q$, there are $i,j$ such that $q_i=b_{ij}$. We say that $Q$ is \emph{$A$-coverable}. If $Q$ has no $A$-cover, we call it \emph{$A$-sprawling} or an \emph{$A$-sprawl.} (The definition works fine with ``nonnegative'' in place of ``positive'' and any set in place of $\mathbb{Z}$. We present it this way for ease.) For lists like this, let $\cup$ denote concatenation. 
\begin{lemma}\label{ProductOfCovering}
    Consider vectors $A,A'$ of dimensions $d,d'$. Consider sets $Q\subseteq \mathbb{Z}^d$ and $Q'\subseteq \mathbb{Z}^{d'}$. $Q\times Q'$ is an $A\cup A'$-sprawl if and only if $Q$ and $Q'$ are $A$- and $A'-$sprawls respectively. 
\end{lemma}
\begin{proof}
    If $Q$ is $A$-coverable, then $Q\times \mathbb{Z}^n$ is $A\cup (0,...,0)$-coverable for any $n$: the same list of $b_{ij}$ still has the desired property. It's clear that increasing the entries of the covering vector and shrinking the set to be covered preserve covering. So if $Q$ is $A$-coverable, $Q\times Q'$ is $A\cup A'$-coverable. Symmetrically, if $Q'$ is $A'$-coverable, $Q\times Q'$ is $A\cup A'$-coverable. 
\end{proof}

We allow $g(S)$ to denote ``the vector of $g(v)$ for $v\in S$, in whatever order they're listed for the sake of $\mathbb{H}(S)$.'' $h(S)$ has the analogous meaning. When we speak of $g(S)$-covers, we'll implicitly require $b_{ij}\in [h(s_i)]$. We'll also speak of a $g(S)$-cover as a geometric object given by the set $\{x\in \mathbb{H}(S) \;\mid\; \exists i,j (x_i=b_{ij})\}$. 

\begin{proposition}\label{ifyouloseyouloseonsprawls}
    Consider a game $\mathcal{G}$ and a plan $f_0$ on $V(D)\backslash S$ for independent $\{s_1,...,s_n\}\equiv S$.
    \begin{enumerate}
        \item A plan $f_S$ for $S$ is a choice of a $g(S)$-cover $B$ for each value of $c_{N^+(S)}$. 
        \item $c_{N^+(S)}$ extends to disprove $f_0\cup f_S$ if and only if $f_S(c_{N^+(S)})$ is not a cover for $L'(S,f_0,c_{N^+(S)})$. 
        \item If $L'(S,f_0,c)$ is a $g(S)$-sprawl for some $c$, then $f_0$ does not extend to win. 
        \item If $S$ is metapeer and $L'(S,f_0,c)$ is $g(S)$-coverable for all $c$, then $f_0$ extends to win. 
    \end{enumerate}
  \end{proposition}
\begin{proof}
    Each item builds on those before. 
    \begin{enumerate}
        \item Obviously, $c_{N^+(S)}$ is the input of $f_S$. That its output is a $g(S)$-cover is just to say that its output is a selection of $g(s_i)$ hyperplanes in $\mathbb{H}(S)$ perpendicular to the $i$th axis for all $i$. That's specified by $g(s_i)$ distinct values in $[h(s_i)]$ (for each $i$), which is precisely the output of a plan $f_S$.
        \item By definition, $L'(S,f_0,c_{N^+(S)})$ consists of possible partial colorings $c_S$ such that $c_{N^+(S)}\cup c_S$ extends to a coloring $c$ causing nobody in $V(D)\backslash S$ to guess right. By 1, a value $c_S\in \mathbb{H}(S)$ is guessed if and only if it's in the cover $f_S(c_{N^+(S)})$. So if all points of $L'(S,f_0,c_{N^+(S)})$ are covered, then we have (for this $c_{N^+(S)}$) that if everyone in $V(D)\backslash S$ guesses wrong, it means someone in $S$ guesses right. Thus, $c_{N^+(S)}$ does not extend to a disprover.
    
        If some point of $L'(S,f_0,c_{N^+(S)})$ is not covered by $f_S(c_{N^+(S)})$, then we have some value of $c_{N^+(S)}\cup c_S$ such that $c_S$ is not covered by $f_S(c_{N^+(S)})$, and a way to color $V(D)\backslash S \backslash N^+(S)$ so that nobody in $V(D)\backslash S$ guesses right. I.e., we have a value of $c_{N^+(S)}\cup c_S$ that extends to a disprover.
        \item If $L'(S,f_0,c)$ is a $g(S)$-sprawl for some $c$, then no matter what $f_S$ is, we have (by 2) a value of $c_{N^+(S)} \cup c_S$ that extends to a disprover for $f_0\cup f_S$. Thus, $f_0$ does not extend to win. 
        \item By Proposition \ref{PropertiesofLL'}.4, Lemma \ref{ProductOfCovering}, and the definition of metapeer, it suffices to prove this with ``peer'' in place of ``metapeer.'' That's easy. In a peer set $S$, each sage knows $c_{N^+(S)}$ precisely, so the sages collectively can choose any particular $g(S)$-cover they like to be the value of $f_S(c_{N^+(S)})$. Since we assume $L'(S,f_0,c_{N^+(S)})$ is always $g(S)$-coverable, they simply cover it, thereby winning for all extensions of $c_{N^+(S)}$ that would cause all members of $V(D)\backslash S$ to guess wrong, by 2. If it is a $g(S)$-sprawl, they lose by 3. \qedhere
    \end{enumerate} 
\end{proof}
\begin{prop}\label{prop:winnableifcomplementhasprism}
    Consider $\{s_1,...,s_n\}=S\subseteq V(D)$. Consider $Q\subseteq \mathbb{H}(S)$. $Q$ is $g(S)$-coverable if and only if $\mathbb{H}(S)\backslash Q$ contains a $(h(s_1)-g(s_1))\times ... \times (h(s_n)-g(s_n))$ combinatorial prism. 
\end{prop}
\begin{proof}
    The $(h(s_1)-g(s_1))\times ... \times (h(s_n)-g(s_n))$ combinatorial prisms $P\subseteq \mathbb{H}(S)$ are in bijection by the complementation operation ``$\mathbb{H}(S)\backslash $'' with the $g(S)$-covers.  
\end{proof}
\begin{prop}\label{prop:TooBigThenSprawl}
    Let $S$ be an independent set and $f_0$ a plan on $V(D)\backslash S$. If $\exists c |L'(S,f_0,c)|>\left(\prod_{v\in S} h(v)\right) - \left(\prod_{v\in S} (h(v)-g(v))\right)$, then $f_0$ does not extend to win. 
\end{prop}
\begin{proof}
    Apply Propositions \ref{prop:winnableifcomplementhasprism} and \ref{ifyouloseyouloseonsprawls}.3.
\end{proof}

\subsection{Aside: Hamming balls and random covering}
 Hat guessing research has used the notion of Hamming distance since at least \cite{GG15}. The Hamming distance for two vectors is the number of places on which they disagree. Let a Hamming ball $H_r(a)$ be the set of vectors whose Hamming distance from vector $a$ is $\leq r$. Assume for the moment that $g$ and $h$ are constant for $S$. A plan $f_S$ for $S$ picks, for each $c$, vectors $a_1,...,a_g$ of pairwise Hamming distance $|S|$. (The vector $a_i$ is composed of the $i$th guess from $v_n$ in the $n$th place. It depends on $c$.) Any coloring $c$ such that $c(S)$ lies within a Hamming ball centered at $H_{|S|-1}(a_i)$ for some $i$ is (by definition) guessed correctly by $f_S$. That is, in order to win, you want an $f_S$ such that $L'(S,f_0,c) \subseteq \bigcup_{i=1}^g H_{|S|-1}(a_i)$ for any $c$. The choice of $\{a_i\}$ is totally free when $S$ is a peer set. Thus, it's easy to see (cf. Lemma 3.2 of \cite{ABST20}) that if $|L'(S,f_0,c)|\leq g|S|$, then there's a plan $f_S$ where $f_S\cup f_0$ is a winning strategy. You can show (cf. Lemma 3.3 of \cite{ABST20}) that this can be accomplished when $|L'(S,f_0,c)|<ge^{|S|/h}$. 

However, the formalism of Hamming balls is cumbersome when $g,h$ vary. We can still show, though, that small-enough sets are coverable when we switch the the formalism of planes. 
\begin{prop}\label{Hamminghowtowin}
    Let $S$ be metapeer and $f_0$ a plan on $V(D)\backslash S$. If $|L'(S,f_0,c_{N^+(S)})| <\prod_{v\in S} (1-r(v))^{-1}$ for all $c$, then $f_0$ extends to win. 
\end{prop}
\begin{proof}
    By Proposition \ref{ifyouloseyouloseonsprawls}.4, it suffices to show that 
    $L'(S,f_0,c_{N^+(S)})$ is always $g(S)$-coverable. To do so, choose a $g(S)$-cover uniformly at random. It misses any given member of $L'(S,f_0,c_{N^+(S)})$ with probability $\prod_{v\in S} \frac{h(v)-g(v)}{h(v)}$, so if $|L'(S,f_0,c_{N^+(S)})|\prod_{v\in S} \frac{h(v)-g(v)}{h(v)}= |L'(S,f_0,c_{N^+(S)})|\prod_{v\in S} (1-r(v))<1$, then with positive probability it misses no member, so $L'(S,f_0,c_{N^+(S)})$ is $g(S)$-coverable.
\end{proof}

\subsection{Vertices universal to independent sets}
\epigraph{Then the due time arrived\\
For Halfdane’s son to proceed to the hall.\\
The king himself would sit down to feast.\\
No group ever gathered in greater numbers \\
Or better order around their [hint]-giver.}
{\textit{Beowulf} 1007-11, trans. Seamus Heaney}

We combine plane covering with the hats-as-hints approach. This spurs us to numerical questions about sprawl sizes, which ultimately yield constructors.     
\begin{proposition}\label{HAHplusSprawls}
    Consider a game $\mathcal{G}$ and $v\in V(D)$ such that $N(v)$ is independent. Fix a plan $f_0$ on $V(D) \backslash N(v) \backslash v$ and $f_v$ on $v$. Let  $P=\{P_1,...,P_{h(v)}\}\subseteq 2^{\mathbb{H}(S)}$ be the set family representing $H_{f_v}$ (according to Definitions \ref{def:derived} and \ref{def:hints}). Let $W\subseteq L'_{\mathcal{G}\backslash v}(N(v),f_0,c)$ represent a $g(N(v))$-sprawl. 
    \begin{enumerate}
        \item If for some $P_i$ and $c$, $L'(N(v),f_0,c)\backslash P_i$ is a $g(N(v))$-sprawl, then $f_0$ does not extend to win. 
        \item If $v$ is Lyonic and $\forall P_i,c$ we can $g(N(v))$-cover the set $L'(N(v),f_0,c)\backslash P_i$, then $f_0$ extends to win. 
        \item Suppose there is at least one sprawl in $L'(N(v),f_0,c)$ and $r(v)^{-1} > \min_{c,W} |W|$. Then $f_0$ does not extend to win $\mathcal{G}$. 
    \end{enumerate}
\end{proposition}
\begin{proof}
     For the first point, combine Proposition \ref{HintWinning}.1 with Proposition \ref{ifyouloseyouloseonsprawls}.3. For the second, combine Proposition \ref{HintWinning}.1 with Proposition \ref{ifyouloseyouloseonsprawls}.4. For the third, notice that for a given $W$, we have $\sum_{i=1}^{h(v)} |P_i \cap W|=g(v)|W|$ by definition of $P$. So $\min_{P_i} P_i\cap W\leq \lfloor \frac{g(v)|W|}{h(v)}\rfloor=\lfloor \frac{|W|}{r(v)}\rfloor$, which if $r(v)>|W|$ is $0$. So for some $P_i,W$, we have $P_i\cap W=\emptyset$. Thus, $W\subseteq L'(N(v),f_0,c)\backslash P_i$. By the above, $f_0$ does not extend to win. 
\end{proof}
We wonder, then: what do these sprawls look like, and how big are they? In particular, how can we bound $\min |W|$? It's lower-bounded by Proposition \ref{Hamminghowtowin}. An upper bound follows from Proposition \ref{prop:TooBigThenSprawl}: any set $Q\subseteq \mathbb{H}(S)$ of order $\geq 1+\left(\prod_{u\in S} h(u)\right) - \left(\prod_{u\in S} (h(u)-g(u))\right)$ is a $g(S)$-sprawl. 
We can improve this by using a lemma proven cleverly by Noga Alon in communication labeled ``remark175.''
\begin{lemma}\label{AlonCoveringLemma}
    Let $A$ be a $d$-tuple of positive integers, and let $Q\subseteq \mathbb{Z}^d$ be a set with no $A$-cover such that for every $q\in Q$, $Q \backslash q$ has a $A$-cover. Then $|Q|\leq \prod_{i=1}^d(a_i+1)$. (This is tight for any $A$, as shown by the set $Q=\prod_{i=1}^d[a_i+1]$.)
\end{lemma}
\begin{proof}
    For any $q\in Q$, there is an $A$-cover $b^{q} =\{b^q_{ij} :1\leq i\leq d,1\leq j \leq g_i\}$ of $Q\backslash q$ which is not an $A$-cover of $Q$. Therefore 
\begin{equation*}\label{firstforcoverproof}
    \text{for every $q' \in Q$, $q' \neq q$ there are $1\leq i\leq d,1\leq j\leq a_i$ such that $q'_i=b_{ij}^q$.}\tag*{(*)}
\end{equation*}
Also, for every $1\leq i \leq d, 1\leq j \leq a_i$, we have $q_{i}\neq b_{ij}$.
Now for every $q\in Q$, we define a real polynomial in the variables $x_1,...,x_d$.
\begin{equation*}
    P^q(x_1,...,x_d):=\prod^d_{i=1} \prod_{j=1}^{g_i} (x_i-b_{ij})
\end{equation*}
Consider any $q,q'\in Q$, interpreted as vectors. By \ref{firstforcoverproof}, $P^q(q')\neq 0$ if and only if $q'=q$. 

Thus, the $|Q|$ different polynomials $P^q$ are linearly independent. Note that they lie within the vector space of polynomials in $x_1,...,x_d$ where the degree of $x_i$ is $\leq a_i$. That space has dimension $\prod_{i=1}^d(a_i+1)$, so $|Q|\leq \prod_{i=1}^d(a_i+1)$. 
\end{proof}

\section{Applications}
\subsection{Deleting vertices (and replacing them with arcs)}
\epigraph{The waste of life occasioned by trying to do too many things at once is appalling.}{Orisen Swett Marden}

Some constructors (``destructors?'') take vertices away. So far, we've seen Corollary \ref{cor:redundantvisiondelete} and Proposition \ref{Removing2Leaf}. They help you simplify games either for specific analysis or to narrow the requirements to prove a theorem. We'd like to have more. In this case, we also get the first known instance of a constructor that can exchange vertices for arcs.

\begin{thm}\label{DeletableByPeerSets}
Let $v$ be Lyonic. If $r(v)^{-1}> \prod_{u\in N(v)} (g(u)+1)$, then $v$ is deletable.  
\end{thm}
\begin{proof}
    If $\mathcal{G}\backslash v$ is unwinnable, then by \ref{ifyouloseyouloseonsprawls}.4, for any $f$, $L'(N(v),f,c)$ includes a sprawl for some $c$. By \ref{AlonCoveringLemma}, $\min_{c,W}|W|\leq \prod_{u\in N(v)} (g(u)+1))$. If $h(v)/g(v)=r(v)^{-1}$ is greater than that, $v$ is deletable by \ref{HAHplusSprawls}.2. 
\end{proof}
The following three corollaries hold because (respectively) every vertex in a bipartite graph is Lyonic, every leaf is Lyonic, and trees can be constructed by adding successive leaves.
\begin{coro}\label{cor:deleteinLatvianbip}
    Consider a Latvian game on $K_{n,m}$. If $v$ is a sage in the left partition with $h(v)>2^{m}$ (or in the right partition with $h(v)>2^{n}$), then $v$ is deletable. 
\end{coro}
\begin{coro}\label{cor:deleteCzechleaves}
    Let $\ell$ be a leaf adjacent to $u$. If $1/r(\ell)>g(u)+1$, then $\ell$ is deletable. (If $1/r(\ell)\leq g(u)+1$, this may not hold: consider $D=P_2$, $h(\ell)=h(u)=k, g(\ell)=1, g(u)=k-1$.)
\end{coro}
\begin{coro}[Theorem 5 of \cite{BDFGM21}]
     For a tree $T$, $\mu_s(T)\leq s(s+1)$. 
\end{coro}

Given a digraph $D$ and independent set $S\subseteq V(D)$, let $D_{\overrightarrow{S}}$ be $D$ with arcs added to make $S$ metapeer. Given a game $\mathcal{G} \equiv (D,g,h)$, let $\mathcal{G}_{\overrightarrow{S}}$ be $(D_{\overrightarrow{S}},g,h)$. 
\begin{thm}\label{ReplaceVertexWithArcs}
    If $\mathcal{G}$ is winnable and $r(v)^{-1}>\prod_{u\in N(v)} (g(u)+1)$ for some directionless $v$, then $\mathcal{G}_{\overrightarrow{N(v)}}\backslash v$ is winnable.
\end{thm}
\begin{proof}
   Because $\mathcal{G}\preceq \mathcal{G}_{\overrightarrow{N(v)}}$, the latter is winnable. By \ref{DeletableByPeerSets}, that implies $\mathcal{G}_{\overrightarrow{N(v)}}\backslash v$ is winnable.
\end{proof}
\begin{remark}
    Getting from a winning $f$ on $\mathcal{G}$ to a winning $f'$ on $\mathcal{G}_{\overrightarrow{N(v)}}\backslash v$ is simple. $f'=f$ for $V(D)\backslash N(v)\backslash v$, and $f'_{N(v)}$ is coordinated to cover $L'_{\mathcal{G}\backslash v}(N(v),f_0, c)$ along the lines of \ref{ifyouloseyouloseonsprawls}.4.
\end{remark}

\subsection{Solving Latvian trees (take 2)}\label{subsec:SlickerTreeProof}
\begin{lmm}\label{DeletableInTrees}
    For a vertex $v$ of a Latvian forest, if $h(v)>2^{\deg(v)}$, then $v$ is deletable.
\end{lmm}
\begin{proof}
    By \ref{DeletableByPeerSets}, it suffices to prove that every vertex $v$ of a forest $F$ is Lyonic. That's obvious because $v$ is on the unique path between any two of its neighbors, so each of its neighbors lies in a different connected component of $F\backslash v$.
\end{proof}

This immediately suggests an algorithm: scan the vertices of $F$, looking for any $v$ with $h(v)>2^{\deg(v)}$. If you find none, stop and output ``winnable.'' If you find one, delete it and repeat. If you delete your last vertex, output ``unwinnable.'' 
\begin{thm}\label{TreeAlgo2}
    This algorithm is correct. Moreover, it returns ``winnable'' if and only if the forest $F$ contains a tree $T$ (not necessarily a whole component) with $h(v)\leq 2^{\deg_{T}(v)}$. 
\end{thm}
\begin{proof}
    Lemma \ref{DeletableInTrees} guarantees that the algorithm's deletions preserve winnableness. Corollary \ref{WinningTreeHatnesses} guarantees that its outputs of ``winnable'' are true, while outputs of ``unwinnable'' are obviously true.
\end{proof}

\subsection{Czech stars as packing problems}
Stars provide the purest example of the discrete-geometry persepective. Consider $v$ to be the middle vertex and $S$ the leaves. The hardest winnable Latvian game on a star $K_{1,n}$ and the only (up to symmetries) winning strategy for it illustrate our covering concepts nicely. Recall that this is $h(\ell)=2$ for every leaf and $h(v)=2^{n}$. We have $\mathbb{H}(S)=\{0,1\}^n$. Each of the $2^n$ possible ``clues'' to be given to the leaves (i.e., $2^n$ possible colors of $v$) corresponds to the elimination of exactly one member of $\{0,1\}^n$, leaving a $(1,1,...,1)$-coverable set. (More generally, consider a Czech game on $K_{1,n}$ where $g(\ell)=h(\ell)-1$ for every leaf and $r(v)=\prod_{\ell\in S} h(\ell)$.) This is a good guessing-centric view of the tightness of Lemma \ref{AlonCoveringLemma}. General Czech stars are more elaborate.

An instance of the \textit{combinatorial prism packing problem} is a list of positive integers $a_1,...,a_n,$ $d_1,...,d_n,x,y$. A \textit{solution} for an instance of the combinatorial prism packing problem consists of a list of $x$ combinatorial prisms $B_k \subseteq \prod_{i=1}^n [d_i]$, each of which has measurements $a_1,...,a_n$, such that the intersection of any $y+1$ of them is empty.

\begin{prop}\label{lmm:WinOnCzechStars}
   A Czech star is winnable if and only if the combinatorial prism packing problem with $d_i=h(\ell_i)$, $a_i=h(\ell_i)-g(\ell_i)$, $x=h(v)$, and $y=g(v)$ has a solution. 
\end{prop}
\begin{proof}[Proof sketch.]
    Apply Propositions \ref{HAHplusSprawls}.1 and \ref{prop:winnableifcomplementhasprism}.
\end{proof}
This isn't a wholly satisfactory solution in the sense of Section \ref{OurApproach}. A solution for the combinatorial prism packing problem is relatively easy to check, but we don't know if it's hard to find. See Section \ref{sec:computation}.

\subsection{Complete bipartite graphs as packing/covering problems}
The same idea holds for bigger complete bipartite graphs. It's just more complicated. 
We'll arbitrarily designate one side to be the ``hint-giving side'' and the other to be the ``covering side.'' (We're finally giving life to the suggestion in the Hats as Hints section that we could use multiple hints.)
Let $a_1,...,a_n$ be the vertices on the left (the coverers) and let $b_1,...,b_m$ be the vertices on the right (the hint-givers). You may also think of $a_1,...,a_n$ as our peer set and $b_1,...,b_m$ as various vertices whose neighborhood is peer.

\begin{thm}\label{ref:CzechCompleteBipartite}
    Consider a complete bipartite Czech game  with vertices $A\equiv \{a_1,...,a_n\}$ on the left and $B\equiv \{b_1,...,b_m\}$ on the right. It is winnable if and only if there exists a set $\mathfrak{P}=\{\mathcal{P}_1,...,\mathcal{P}_m\}$ with the following properties: 
    \begin{itemize}
        \item for all $i$, $\mathcal{P}_i$ is a list of subsets $P_{i,q}\subseteq \mathbb{H}(A)$,  
        \item for all $i$, $\mathcal{P}_i$ has length $h(b_i)$,
        \item for all $i$ and for each $x\in \mathbb{H}(A)$, there exist exactly $g(b_i)$ sets $P_{i,q}\in \mathcal{P}_i$ such that $x\in P_{i,q}$,
        \item and for any $(q_1,q_2,...,q_m)\in \mathbb{H}(B)$, $\bigcup P_{i,q_i}$ contains a combinatorial prism of measurements $\{h(a_i)-g(a_i)\}$.
    \end{itemize}
\end{thm}
\begin{proof}[Proof sketch.]
    \ref{HAHplusSprawls}.1 extends somewhat obviously from the case of one $v$ to the case of multiple $b_1,...,b_m$ with $\forall (N(b_i)=N(b_j))$. Then imitate the proof of \ref{lmm:WinOnCzechStars}. 
\end{proof}
\begin{corollary}\label{cor:ClassicBipartite}
    $\mu(K_{m,n})\geq k$ if and only if there exist $n$ partitions $\mathcal{P}_i$ of $[k]^m$ into $k$ parts such that, whenever we select one member $P_{i,q_i}$ from each $\mathcal{P}_i$, the set $\bigcup P_{i,q_i}$ contains a $(k-1)\times (k-1) \times ... \times (k-1)$ combinatorial prism.
\end{corollary}
On an aesthetic note, it's not obvious (to me at least) without the hat guessing formalism that $m$ and $n$, or $\{a_i\}$ and $\{b_i\}$, can be exchanged in the above statements. 

\chapter{Randomizing hats}\label{chap:LLL}
Here we study restrictions on $r(v)$ that arise when we consider the uniform probability space of colorings and investigate the probability of the event that no sage guesses right. Mostly we use the Lovász Local Lemma (LLL) and its variants, which we'll state. These deal with dependency (di)graphs, structures whose vertices are events and whose arcs/edges express possible dependence. They say, of certain probability vectors over the vertices, that any set of events with these probabilities and this dependency (di)graph, that with positive probability (w.p.p.) none of the events occur. 

We discuss this approach's history and previous results. Then we define the ``Ratio game,'' which abstracts away from the $g,h$ of the Czech game to $r$ alone. We explain how all facts about Czech games in this paradigm are actually facts about Ratio games. 

The headline fact, which enables all applications, is: ``ratio is probability; visibility is dual to dependency.'' That is to say: the probability that $v$ guesses right is $r(v)$, and the dependency graph for the events ``sage $v$ guesses right'' is simply the dual of the visibility graph, where the vertex corresponding to a sage then corresponds to the event that she guesses right. (\cite{HV15} remarks that ``nearly all applications [of the Lovász Local Lemma] of which we are aware involve an undirected graph.'' This is one that does not!) These applications include a bound in $\Delta^-$ of $\hat{\mu}$ for a digraph and a necessary condition for a Ratio game to be winning. Also, we consider the hope of Scott and Sokal that Shearer's result on independence polynomials could extend to directed graphs and give a disappointing answer using the example of the Classic game $(\overrightarrow{C}_k, 2)$.

We should also note that LLL-related methods needn't always be negative. Certain variants, like Shearer's lemma for (undirected) dependency graphs, are sharp. That is to say, they give sufficient \textit{and necessary} conditions for a probability vector to guarantee that w.p.p., none of the events occur. (This is for generic events, however––for Ratio games, it's only been extended to chordal graphs \cite{BDO21}. See Question \ref{PolynomialWorksForAllRatioGames}.) 

Nor need they be nonconstructive (and so out of line with our algorithmic orientation). In recent years, there's been much work (e.g. \cite{HV15}) in constructive methods––algorithms that can find a specific member of the probability space such that none of the events happen. 

Finally, there are other notions of dependency (di)graphs. One might fruitfully consider the ``soft-core dependencies'' of \cite{SS06} (see especially Theorem 4.8). The ``d-dependency'' graph of \cite{KLP16} can be strictly sparser than our dependency digraph, though on first inspection it seems to be usually denser for hat guessing. In general, the ``variable-LLL'' approach \cite{HLLWX17} is appropriate for hat guessing, but it's not clear to me how useful. 

\textbf{History.} Nobody before has considered this method in conjunction with directed graphs. An ancient result of deterministic hat guessing––already ``folklore'' by \cite{Far16}––is that $\mu(G)\in O(\Delta(G))$. Different sources give different precise bounds. Reasoning by entropy compression, \cite{Far16} obtained $\mu\leq 4\Delta+4$. Using the Lovász Local Lemma, Farnik obtained $\mu \leq e\Delta+e$, but he remarked that by using \cite{She85}, one could obtain a bound of $<e\Delta$. \cite{ABST20}, using the same paradigm, shows a complex proposition that relates $\mu$ for subgraphs on vertices of high degree to $\mu$ for the whole graph. \cite{BDO21} brought in Shearer's Lemma to prove some necessary conditions on an undirected Ratio game, which they prove are sufficient when the graph is chordal. (They don't speak of ``Ratio games,'' but this is essentially what they prove.) This polynomial condition was picked up by \cite{LK22} and \cite{LK23} and combined with constructors.

\section{Defining Ratio games}
\epigraph{Natura non facit saltum.}{Leibiz, Lucretius, and others}

Throughout, ``measure'' means the usual Lebesgue measure on $[0,1)^n$ for any $n$. That's our probability space. I.e., we draw colorings $c$ uniformly at random from $[0,1)^{|V(D)|}$. For $m\in [0,1]$, let $\mathcal{F}_m$ denote the collection of subsets of $[0,1)$ of measure $m$. 

\begin{defn}
    Let $A, A_0,...,A_{k-1}$ be events. We say that $A$ is \textit{mutually independent of} the set of events $\{A_0,...,A_{k-1}\}$ if 
    
    \[\mathbf{P}\left(A\cap \left(\bigcap_{i\in J} A_i\right)\cap \left(\bigcap_{i\in [k]\backslash J} A_i^c\right)\right)=\mathbf{P}(A)\mathbf{P}\left(\left(\bigcap_{i\in J} A_i\right)\cap \left(\bigcap_{i\in [k]\backslash J} A_i^c\right)\right)\]

    for every $J\subseteq [k]$. Equivalently, for every $J\subseteq [k]$, we have 
    \[\mathbf{P}\left(A\, \mid \, \bigcap_{i\in J} A_i\cap \bigcap_{i\in [k]\backslash J} A_i^c \right)=\mathbf{P}(A).\] 
    Consider a set of events $\mathcal{A}=\{A_1,...,A_n\}$. A directed graph $D$ with $V(D)=\mathcal{A}$ is called a \textit{dependency digraph for $\mathcal{A}$} if for all $A_i\in \mathcal{A}$, $A_i$ is mutually independent of the events $\{A_j \; \mid \; \overrightarrow{A_iA_j}\not\in E(D)\}$. An undirected dependency digraph is called a \textit{dependency graph.} 
\end{defn}

\begin{defn}
    Fix a digraph $D$ with vertex set $X$. If a sequence $\mathbf{p}=\{p_x\}_{x\in X}\in [0,1]^X$ has the property that, for every collection of events $\{A_x\}_{x\in X}$ with dependency digraph $D$ with $\forall x (\mathbf{P}(A_x)\leq p_x)$ we have $\mathbf{P}(\bigcap_{x\in X} \neg A_x)>0$, then we say that $\mathbf{p}$ is \textit{good for $D$}. 
\end{defn}

The essence of the Lovász Local Lemma and its variants is to say, of sequences $\mathbf{p}$ and digraphs $D$, whether $\mathbf{p}$ is good for $D$ or not. Sufficient and necessary (albeit NP-hard) conditions are known for dependency graphs; for dependency digraphs less is known. 

\begin{defn}\label{def:RatioGames}
    A \textit{Ratio game} is a pair $(D, r)$, where $D$ is a digraph and $r$ is a function from $V(D)$ to $(0,1)$. We call $r(v)$ the \textit{ratio of $v$}. For each sage $v$, the set of possible hat colors $V_v$ is $[0,1)$;  a coloring $c$ is a function $v: V(D)\rightarrow [0,1)$. A plan for $v$ is a function $f_v: \prod_{u\in N^+(v)} V_u \rightarrow \mathcal{F}_{r(v)}$. That is, for every possible sight $v$ could have, she has an associated measure-$r(v)$ set that she will guess. We say she guesses right if $c(v)\in f_v(N^+(v))$. The rest of the nomenclature extends in the obvious fashion. As usual, $R_v$ denotes the event that sage $v$ wins. (For discussion of alternative definitions, limit properties, etc., see Subsection \ref{subsec:RatioGames}.)
\end{defn}  
\begin{defn}
    Let $\mathcal{G}\equiv (D,r)$ and $\mathcal{G}'\equiv (D',r')$ be Ratio games. The ease relation, $\preceq$, is the minimal partial order such that $\mathcal{G} \preceq \mathcal{G'}$ if either of the following hold: 
    \begin{itemize}
        \item $D\subseteq D'$ and $r=r'$ on $V(D)$.
        \item $D=D'$ and $r'\geq r$.
    \end{itemize}
    If $\mathcal{G}$ is a Czech game and $\mathcal{G}'$ is a Ratio game, with $D=D'$ and $r=r'$, we call $\mathcal{G}'$ the \textit{associated Ratio game of $\mathcal{G}$}. We confidently write $\mathcal{G}\preceq \mathcal{G}'$  and take the transitive closure under this additional stipulation. As before, we say $\mathcal{G}'$ is easier than $\mathcal{G}$.  
\end{defn}
Clearly, all probabilistically derived results about Czech games that use only $r$ without considering $g$ or $h$ actually pertain to the associated Ratio game of that Czech game.
\begin{prop}\label{RatioMonotone}
    If $\mathcal{G}\preceq \mathcal{G}'$ and $\mathcal{G}$ is winnable, then $\mathcal{G}'$ is winnable.    
\end{prop}
\begin{proof}
    If both games are Czech, this is Proposition \ref{monotone}. If both games are Ratio games, it suffices to check it for each bullet point, which is trivial. If $\mathcal{G}$ is Czech and $\mathcal{G}'$ a Ratio game, it suffices (by the previous two sentences and the definition of $\preceq$ as a transitive closure)
    to prove it for $\mathcal{G}'$ the associated Ratio game. For that, it's easy to define a strategy for $\mathcal{G}'$ that imitates any strategy from $\mathcal{G}$ by segmenting the unit interval into $h(v)$ equal-measure segments, having the guess functions ignore any data besides which interval $c(u)$ is in (for $v\rightarrow u$), and have them output only sets from the the intersection of $\sigma$-algebra of that segmentation. Clearly, if the strategy from $\mathcal{G}$ wins, so does the strategy we derived for $\mathcal{G}'$. 
\end{proof}
\begin{prop}\label{prop:fractionalandratio}
    $\hat{\mu}(D)=1/r_s$, where $r_s$ is the supremum of constants $r_0$ such that $(D,r_0)$ is unwinnable.
\end{prop}
\begin{proof}[Proof sketch.]
    By the definition of $\hat{\mu}$ and the density of the rationals.
\end{proof}

\section{Visibility-dependency duality and its consequences}\label{sec:thmsofLLL}
\begin{thm}\label{thm:mainfact}
Consider a game $\mathcal{G}$ (it can be Ratio or Czech) and a strategy $f$. Then $\mathbf{P}(R_v)=r(v)$. Also, the digraph $D'=(\{R_v\},E)$ that has $\overrightarrow{R_vR_u}\in E$ if and only if $\overrightarrow{uv}\in D$ is a dependency digraph for the events $\{R_v\}$.
\end{thm}
\begin{proof} It suffices to prove that $\mathbf{P}(R_v \, \mid \, (\bigcap_{u\in J} R_u)\cap (\bigcap_{w\in V(D)\backslash N^-(v)\backslash v \backslash J)} R_w))=r(v)$ for any $J\subseteq V(D)\backslash N^-(v)\backslash v$.  Denote by $\mathcal{J}$ the set of $c$ such that the event $\left((\bigcap_{u\in J} R_u)\cap (\bigcap_{w\in V(D)\backslash N^-(v)\backslash v \backslash J} R_w)\right)$ occurs. So the sufficient condition rephrases as $\mathbf{P}(R_v \; \mid \; c\in \mathcal{J})=r(v)$. 
    
For any fixed partial coloring $c_0$ on $V(D)\backslash v$, because the coloring of each vertex of $c$ is drawn independently and uniformly at random, we have $\mathbf{P}(R_v \; \mid \; c_{D\backslash v}=c_0)=\mathbf{P}(c(v) \in f(c_0(N^+(v))))=r(v)$. Since $c\in \mathcal{J}$ depends only on $c(V(D)\backslash v)$, we can integrate $\mathbf{P}(R_v \; \mid \; c_{D\backslash v}=c_0)$ over all $c_0$ such that $c_0$ extends to a coloring in $\mathcal{J}$, and we obtain $\mathbf{P}(R_v \; \mid \; c\in \mathcal{J})=r(v)$, completing the proof.
\end{proof}

\begin{remark}
    Theorem \ref{thm:mainfact} isn't even slightly true without reversing the arrows' directions. You can concoct counterexamples with Classic games of hatness $2$ on rather small digraphs. These can be strongly connected or not, the strategies can win or not, the directed girth can be whatever you like, and so on.  
\end{remark}

The previous theorem can be used with any variant of the Lovász Local Lemma––i.e., any result that says, ``if the dependency (di)graph and probabilities look a particular way, then with positive probability none of the events occur." Here are some applications. 

\begin{prop}\label{prop:application1}
    Consider a Ratio game $\mathcal{G}$. If there exists a function $x:V(D)\rightarrow \mathbb{R}$ such that $r(v)\leq x(v)\prod_{u\rightarrow v}(1-x(u))$ for all $v$, then $\mathcal{G}$ is unwinnable. 
\end{prop}
\begin{proof}
    Immediate from \ref{thm:mainfact} and the directed Lovász Local Lemma (e.g. Lemma 5.1.1 of \cite{AS08}). 
\end{proof}
\begin{prop}\label{prop:application2}
    For a digraph $D$, $\hat{\mu}(D)\leq e(\Delta^-+1)$.
\end{prop}
\begin{proof}
   Set $x(v)=1/(\Delta^-+1)$ for all $v$. Apply \ref{prop:application1} and \ref{prop:fractionalandratio}.
\end{proof}
\begin{prop}\label{prop:application3}
    For a graph $G$, $\hat{\mu}(G)\leq (\Delta-1)^{1-\Delta} \Delta^{\Delta}<e\Delta$.
\end{prop}
\begin{proof}
    Use Theorem 2 from \cite{She85} and Theorem \ref{thm:mainfact}. 
\end{proof}

\section{The dream of a directed Shearer's Lemma}\label{sec:DirectedShearersLemma}

    Shearer's Lemma is a sharp version of the Lovász Local Lemma for undirected graphs. Here we discuss the problem of generalizing it to digraphs and conclude that it probably cannot be done slickly, at least for any digraph with an odd directed cycle.

    \begin{prop}[Theorem 1 of \cite{She85}, essentially]\label{prop:originalshearer}
        Consider $S\subseteq V(G)$. Define 
    \[Y_S(\mathbf{w})\equiv \sum_{\substack{S\subseteq I \\ I \text{ is independent}}}(-1)^{|T|-|S|} \prod_{x\in T} w_x. \]
        If $Y_S(\mathbf{w})\geq 0$ for all $S$ and $Y_\emptyset(\mathbf{w})>0$, then $\mathbf{w}$ is good for $G$. 
    \end{prop}
    It was elegantly rephrased by Scott and Sokal \cite{SS06}, with a tougher and deeper proof. 
    \begin{defn}
        For a graph $G$, define the \textit{multivariate independence polynomial} as the sum of monomials where each monomial is the product of variables representing members of an independent set. That is, \[Z_G(\mathbf{w})=\sum_{\substack{I\subseteq X \\I \text{ is independent}}}\prod_{x\in I} w_x. \]
        For a set $X$ and a sequence of radii $\mathbf{r}=\{r_x\}_{x\in X}$, let us define the \textit{closed polydisc} $\overline{D}_\mathbf{r}$ as the set of complex-valued vectors indexed over $X$ whose $x$ place has absolute value $\leq r_x$, for each $x$. I.e., 
        \[\overline{D}_\mathbf{r}=\{\mathbf{w}\in \mathbb{C}^X \;:\; \forall x(|w_x| \leq r_x)\}.\] 
    \end{defn}
    \begin{prop}[Theorem 1.3 of \cite{SS06}]\label{prop:shearerslemma}
        A sequence $\mathbf{p}$ is good for dependency graph $G$ if and only if $Z_G$ has no zeros in $\overline{D}_\mathbf{p}$. 
    \end{prop}
    
    \begin{prop}[Proposition 9 of \cite{BDO21}, essentially]\label{prop:shearerforratio}
        Consider undirected Ratio game $\mathcal{G}\equiv (G,r)$. If $Z_G$ has no zeroes in $\overline{D}_r$, then $\mathcal{G}$ is unwinnable.  
    \end{prop}
    \begin{proof}
        Proposition \ref{prop:shearerslemma} and Theorem \ref{thm:mainfact}.
    \end{proof}

    Proposition \ref{prop:shearerforratio} was deployed to great effect in \cite{BDO21}. It makes me wonder––what about directed games? What I'm really asking here is, what about directed dependency digraphs? Can we find necessary and sufficient conditions for a sequence to be good on a digraph $D$?
    Scott and Sokal wondered the same thing. 
    On page 268 they said, 
    \begin{quote}
        It would be interesting to have a digraph analogue of Theorem 3.1, but we do not know how to do this.  
    \end{quote}
And on page 273 they reiterated,
\begin{quote}
    Let us remark that we have been able to relate the Lovász Local Lemma to a combinatorial polynomial (namely, the independent-set polynomial) only in the case of an \textit{undirected} dependency graph G. Although the Local Lemma can be formulated quite naturally for a dependency \textit{digraph} [2, 5, 23], we do not know whether the digraph Lovász problem can be related to any combinatorial polynomial. (Clearly the independent- set polynomial cannot be the right object in the digraph context, since exclusion of simultaneous occupation is manifestly a symmetric condition.) 
\end{quote}
    
    So the task is: find a combinatorial digraph polynomial $Q_D$ such that $\mathbf{p}$ is good for $D$ if and only if $Q_D$ and $\mathbf{p}$ stand in some relation to one another. Furthermore, whatever $Q_D$ and relevant theorem we obtain about it should restrict to $Z_G$ and Proposition \ref{prop:shearerslemma} or Proposition \ref{prop:originalshearer}
    for undirected graphs.\footnote{I find perverse and improbable the notion that there might exist an analogue for digraphs that breaks if the digraph ceases to be a proper digraph––especially since undirectedness is a global property, and we're dealing with the Lovász \textit{Local} Lemma here.} That immediately narrows our search. I reviewed much literature and found no digraph polynomials that generalize the independence polynomial. There's one relatively obvious candidate we can define, however. 

    \begin{defn}
        For a digraph $D$, define the \textit{multivariate ayclicity polynomial} $Q_D$ as the sum of monomials where each monomial is the product of variables representing members of an acyclic set. That is, \[Q_D(\mathbf{w})=\sum_{\substack{A\subseteq X}\\ A \text{ is acyclic}}\prod_{x\in A} w_x. \]
    \end{defn}

    Clearly, $Q_G=Z_G$ for undirected $G$, $Q_D(0)=1$, all events in a given monomial are mutually independent, and various other nice things check out. Furthermore, the correspondence between the versions persists.  
    
    Let $Y'$ be defined like $Y$ in Proposition \ref{prop:originalshearer}, except that ``acylic'' replaces ``independent.'' For $S\subseteq V(D)$, let $(-\mathbf{w1}_S)_x$ equal $-w_x$ for $x\in S$ and $0$ otherwise.  
  
    \begin{prop}\label{differentformsequiv}
        $Q_D(-\mathbf{w 1}_S)>0$ for all $S\subseteq V(D)$ if and only if $Q_D$ has no zeroes in $\overline{D}_{\mathbf{w}}$ if and only if $Y'_S(\mathbf{w})\geq 0$ for all $S$ and $>0$ for $S=\emptyset$.
    \end{prop}
    \begin{proof}[Proof sketch.]
        Imitate the proof of $(c)\iff(b') \iff (f)$ in Theorem 2.2 of \cite{SS06}.
    \end{proof}
    Proposition \ref{differentformsequiv} establishes that the various forms of Shearer's Lemma, if they extend to acyclicity, all extend together.
    You might reasonably ask: ``why are we reusing the conditions from $Z_G$? Couldn't it be some more complex condition on $Q_D$ that restricts to `no zeroes in the polydisc' when $D$ is a graph? And sure! It could be. But I can't think of any candidates. 
    \begin{lemma}\label{lmm:calcpropdirectedcycle}
        Consider a collection $X$ of events $X_v$ having probability vector $\mathbf{r}=(r_v)_{v\in X}$ with dependency digraph $\overrightarrow{C}_{k}$ for $k$ odd. Then \[ 
            Q_D(-\mathbf{r1}_S)= 
            \begin{cases}
                \prod_{v\in V(D)}(1-r_v)+\prod_{v\in V(D)}r_v &\text{if $S=V(D)$}\\
                \prod_{v\in S} (1-r_v) &\text{otherwise $S=V(D)$}
            \end{cases}\]
    \end{lemma}
    \begin{proof}
        Every set except $V(D)$ in $\overrightarrow{C}_{k}$ is acyclic. The acyclicity polynomial of an acyclic set, evaluated on $-\mathbf{r}$, is \[1-\sum_{v\in S} r_v + \sum_{\substack{T\subseteq S  \\ |T|=2}} \prod_{v\in T} r_v -  \sum_{\substack{T\subseteq S  \\ |T|=3}} \prod_{v\in T} r_v...\] 
        which factors as $\prod_{v\in S} (1-r_v)$. $V(D)$, however, is not acyclic, so the monomial $-\prod_{v\in V(D)}r_v$ does not appear, so we add $\prod_{v\in V(D)}r_v$.
    \end{proof}

    \begin{prop}\label{thm:nodirectedshearers}
        Let $D$ be a digraph containing an odd directed cycle subgraph. There exist events $\{R_v\}$ having dependency digraph $D$ and probability vector $\mathbf{p}$ with $Q_D(-\mathbf{p1}_S)>0$ for all $S$ but $\mathbf{P}(\bigcap \neg R_v)=0$.
    \end{prop}
    \begin{proof}
        We can set $p_x$ to $0$ for all events outside the odd directed cycle subgraph of $D$, so it suffices to prove this for $D=\overrightarrow{C}_k$. Consider the Classic game $(\overrightarrow{C}_k, 2)$ with a winning strategy as in Theorem \ref{LatvianDirectedCycles}, and let $R_v$ be the event that sage $v$ wins. The dual of a directed cycle is a directed cycle, so the dependency digraph of $\{R_v\}$ is $\overrightarrow{C}_k$, and $\mathbf{P}(R_v)=1/2$ for all $v\in V(\overrightarrow{C}_k)$ by Theorem \ref{thm:mainfact}. In this setup, by Theorem \ref{LatvianDirectedCycles}, we have $\mathbf{P}(\bigcap \neg R_v)=0$. However, by \ref{lmm:calcpropdirectedcycle} $Q_D(-\mathbf{r1}_S)=2^{-k+1}>0$ if $S=V(D)$ and $Q_D(-\mathbf{r1}_S)=2^{-|S|}$ otherwise.
    \end{proof}
Thus, if we had hoped to write a theorem like ``if $Q_D$ is nonvanishing in $\overline{D}_\mathbf{p}$, then $\mathbf{p}$ is good for $D$,'' our hopes are frustrated. See Subsection \ref{subsec:RatioGames} for other investigations on these methods. 

\chapter{Open problems and partial results}
I'll resist the temptation to write a final chapter as fanciful as the zeroth. Instead, this will merely point the way for further research by collating open problems. This includes all that have been stated in the literature (and remain unsolved), along with an indiscriminate menu of new ones. For some, we casually present preliminary results that didn't fit in any earlier chapter. We investigate Czech games on directed cycles and translate it into a Ramsey problem. We prove basic analytic facts about ratio games. We gesture toward complexity theory. Less propositionally, we'll discuss the relation of $\mu/\mu_s/\hat{\mu}$ to parameters, classes, and operations of (di)graphs, along with random (di)graphs, some constructors, and miscellany.
    
Many conjectures and questions on hat guessing have been disproved or answered. The very first question\footnote{Which includes Question 2.7.2 of \cite{HT13}.} of the first paper \cite{BHKL09}––concerning the order of $\mu(K_{n,n})$ and the independence of $\mu$ from $\omega$ more broadly––has been answered; see \cite{ABST20} and \cite{GG15} respectively. We still don't know whether there exists an upper bound for $\mu$ of planar graphs, but \cite{BDFGM21} conjectured that planar graphs have hat guessing number $\leq 4$, and \cite{GIP20} conjectured that it was at most 6. By \cite{LK23}, it is at least 22.  \cite{BDFGM21} also conjectured that $\mu$ was at most the Hadwiger number, which is quite false:\footnote{The easiest counterexample, however, is windmill graphs, for which the simplest proof is by constructors in \cite{LK21}. Indeed, by examining the Latvian triangle of hatnesses 4,4,2 and taking the product, we see that the bowtie graph (whose Hadwiger number is 3) has hat guessing number $\geq 4$ (and in fact $=4$ becuase it's a non-clique graph on 5 vertices).} the Hadwiger number of $D_d(N)$ from \cite{HL20} is $d+1$, while the hat guessing number is doubly-exponential in $d$. \cite{GIP20, HIP22, ABST20} and \cite{ABST20} guessed that $\mu\leq \Delta+1$. \cite{LK22} solidly disproved this, also answering (affirmitively) the question from \cite{HIP22} as to whether there exist edge-critical graphs of arbitrarily high diameter and hat guessing number. Though \cite{Bra21} guessed that there exist outerplanar graphs of $\mu$ at least 1000, \cite{KMS21} showed that $\mu \leq 40$ for outerplanar graphs. 

I have seen no conjecture proven, only disproven. Perhaps this is because finding counterexamples is easier than proving general theorems––or perhaps this is simply because this problem confounds our intuitions. Out of caution, we write ``question'' in lieu of ``conjecture'' even when representing others' conjectures. (Also of note: none of these counterexamples were found (to my knowledge) using computer search. Most used the constructors approach of Kokhas, Latyshev, and Retinskiy.)

We present all stated problems that remain open pertaining to the games and parameters defined in Section \ref{DefiningGamesAndParameters}. Questions without citation are original. We are not especially discerning in our choice of original questions; the broad menu is for inspiration. For the definition of any unfamiliar parameters or types of graphs, see \cite{deR14}. Also, whenever we pose a question concerning $\mu$, realize that the same could be asked of $\mu_s$ or $\hat{\mu}$. Similarly, when we speak of graphs, it often could be replaced by digraphs.

In this chapter, we sometimes omit or abridge proofs and definitions. When we speak of algorithms, the word ``efficient'' should be taken loosely. We're not committing to a formalism. 

\section{Games on particular (di)graph families}
\subsection{Extending Latvian games}
Originally, this subsection concerned Latvian cycles of $h(v)\geq 5$ and partial results thereupon, but then I figured out Theorem \ref{CategorizedFiveAndUp}. As mentioned in the introduction, this solves Latvian unicyclic graphs. So we wonder:
\begin{ques}
What Latvian games on graphs with exactly two cycles are winnable?    
\end{ques}
Because leaves can be worked out by Proposition \ref{Removing2Leaf}, it suffices to study kayak paddle graphs and graphs that consist of two cycles intersected at one vertex. The cleanest answer would be that all winnable Latvian games on those graphs can be attained from applying Proposition \ref{SinglePointProduct} to winnable unicyclic graphs. Those were the graphs on which we studied Classic games for Subsection \ref{simplifyingclassic}. There, we found that things were harder when two cycles shared one or more edges; see \cite{KL19}. Probably this represents an uptick in Latvian games' challenge as well.

\begin{ques}
    Consider graphs with maximum degree 3, no leaves, and exactly two vertices of maximum degree. What Latvian games on these graphs are winnable? How hard is this, computationally?
\end{ques}

\subsection{Nontrivial Czech games}\label{subsec:nontrivczechgames}
Czech games are tricky. It was easy enough to solve them for complete graphs, yet we have not solved Czech games for any other digraph!\footnote{Essentially. Of course, \ref{prop:strongcomponly} enables us to solve digraphs whose strong components are cliques, and Corollary \ref{cor:redundantvisiondelete} allows us to delete certain structurally redundant vertices, but this isn't especially interesting.} That seems worth remedying. 

The minimal connected non-complete graph is $P_3$. Label its vertices $A,B,C$. $P_3$ is isomorphic to $K_{1,2}$, so by Proposition \ref{lmm:WinOnCzechStars}, we have the following characterization of winnable $P_3$s. 
\begin{prop}
    A Czech game $\mathcal{G}$ on $P_3$ is winnable if and only if a combinatorial $h(A)\times h(C)$ rectangle can fit $h(B)$ combinatorial $(h(A)-g(A))\times (h(C)-g(C))$ rectangles, such that no $g(B)+1$ of the smaller rectangles have nonempty intersection. 
\end{prop}
\begin{ques}
    Is there an efficient algorithm for determining whether a Czech  $P_3$ is winnable? I.e., is the 2-dimensional combinatorial prism packing problem in $P$?
\end{ques}

The minimal strongly connected non-clique proper digraph is $\overrightarrow{C}_3$. Lemma \ref{DirectedCyclesLemma} gives us a necessary condition for a Czech $\overrightarrow{C}_3$ to be winnable, but it's not sufficient.
\begin{prop}\label{prop:directedtrianglesweird}
    Let $\mathcal{G}$ be a Czech game on $\overrightarrow{C}_3$ with $k:=h(A)=h(B)=h(C)$ and $\ell:=g(C)=k-g(A)=k-g(B)$. $\mathcal{G}$ is winnable if and only if $\ell \mid k$. 
\end{prop}
\begin{proof}[Proof sketch.]
    ``If'' is easy. It's winnable with $\ell=1$: $A$ and $B$ guess whatever they don't see, and $C$ guesses what she sees. Then apply Proposition \ref{monotone}. For ``only if,'' we read $c(A)$ and $c(B)$ as the forbidding of certain values of $c(C)$ (with certain of these forbiddings incompatible) and consider the numerical requirements for $B$'s strategy.
\end{proof}
At least on the margin, determining winnable Czech $\overrightarrow{C}_3$s is strange indeed. 
\begin{ques}
    Is there an efficient algorithm for determining whether a Czech $\overrightarrow{C}_3$ is winnable?
\end{ques}
For both questions in this subsection, having an answer for $k$ would be lovely, but $3$ seems like a hard enough step already. 
Still, we've had more success with $\overrightarrow{C}_k$ than with any other nontrivial class of graphs. For nothing else have we solved the Latvian and Polish games. The following Ramseyish angle might be helpful. 
\begin{prop}
    Let $\mathcal{G}$ be a Czech game on $\overrightarrow{C}_k$ with hatnesses $h_0,...,h_{k-1}$ and guessnesses $g_0,...,g_{k-1}$. 
    Construct digraph $H_\mathcal{G}$ as follows. Start with the complete $k$-partite directed cycle $\overline{C}_{h_0,...,h_{k-1}}$ (as in \cite{ABST20}). For each vertex in the $i$th component, remove $g_{i-1}$ of its in-arrows. (As usual, $0-1=k-1$.)
    $\mathcal{G}$ is winnable if and only if there is some $H_\mathcal{G}$ with no $\overrightarrow{C}_k$ subgraph. 
\end{prop}

Finally, it seems worth remarking that we're close to solving Czech cycles with $h\leq 4$. 
The cutoff $h\leq 4$ is qualitatively interesting for Latvian cycle because (if our guesses above are true) the winnability criterion when $\exists v(h(v)\geq 5)$ is simple: ``contains two vertices of hatness 2 between which there is no vertex of hatness $\geq 5$.'' That cutoff is less interesting for Czech games. Nevertheless, we're close to solving it. (Ignore $C_3\equiv K_3$.)

Proposition \ref{CzechOnComplete} implies that any low-hatness ($h\leq 4$) cycle with a guessness-3 vertex is winnable, or a guessness-2 and hatness-3 vertex adjacent to a hatness-3 vertex. Proposition \ref{monotone} and Corollary \ref{WinningTreeHatnesses} imply that any low-hatness cycle with two guessness-2 vertices, or any low-hatness cycle with a guessness-2 vertex and a hatness-2 vertex is winnable. Theorem \ref{CategorizedTwoToFour} forbids certain configurations with a guessness-2, hatness-4 vertex. So we've reduced it to the following cases. They're probably solvable by hats-as-hints or admissible paths; I haven't tried. 
\begin{ques}
    Consider a Czech cycle $\mathcal{G}$, with $h(v)\in \{3,4\}$ for all $v$. $g(A)=2$, and $g(v)=1$ for other $v$. 
    \begin{enumerate}
        \item Suppose $h(Z)=h(B)=4$. Can we say that $\mathcal{G}$ is winnable or not? Does it depend on $h(A)$? Does it depend on the other vertices' hatnesses?
        \item Suppose $h(Z)=h(A)=h(C)=4$, and $h(B)=3$. Can we say that $\mathcal{G}$ is winnable or not? Does it depend on the other vertices' hatnesses?
    \end{enumerate}
\end{ques}

We expect general Czech cycles to be prohibitively difficult.

\subsection{Symmetrical strategies}
As Corollary 6 of \cite{Szc17} showed, there's a winning strategy for $(C_{3k},3)$ in which every vertex's strategy ``looks the same.'' We showed in the proof of \ref{LatvianDirectedCycles} that the same is true for $(\overrightarrow{C}_{2k+1}, 2)$. This feature seems impressive and delicate; when else can it be achieved? 

We might formalize that by considering (di)graph automorphisms. An automorphism of a digraph $D$ is a bijection $b: V(D)\rightarrow V(D)$ such that $v\rightarrow u \iff b(v) \rightarrow b(u)$. If we have a game $\mathcal{G}\equiv (D,g,h)$, and additionally $\forall v((h(b(v)),g(b(v)))=(h(v),g(v)))$, we say that $b$ is an automorphism of $\mathcal{G}$. 
Automorphisms form a group $Aut(\mathcal{G})$. Let $b$ be an automorphism and $f$ a strategy. There are many reasonable criteria by which one could call $f$ ``invariant under $b$.'' Pick your favorite. 

\begin{ques}\label{SameStratQuestion}
    What games (with nontrivial automorphism groups) are winnable via strategies invariant under $Aut(\mathcal{G})$? 
\end{ques}
\begin{ques}
    What groups $G$ have the property that there exists a game $\mathcal{G}$ and winning strategy $f$ invariant under $G$ as a subgroup of $Aut(\mathcal{G})$?
\end{ques}
There are many lovely graphs with rich automorphism groups. E.g. strongly regular graphs, Johnson graphs, cages, snarks, and so forth. This isn't a rigorous question, but we wonder whether there mightn't be some nice algebra involved in hat games on such graphs, or on similarly well-structured digraphs.

\section{The parameters $\mu/\mu_s/\hat{\mu}$ and (di)graph structure}
Although we've said little about the specific parameters $\mu,\mu_s,\hat{\mu}$, there are plenty of interesting open questions pertaining to them. (Again, most references to $\mu$ can be profitably replaced by $\mu_s$ or $\hat{\mu}$.)

\subsection{Bounding on classes}
A common form for hat guessing conjectures, ``the hat guessing number of a class,'' was made notationally precise by \cite{LK23}. For $p$ a (di)graph parameter and $\mathcal{C}$ a class of (di)graphs, we define $p(\mathcal{C})\equiv \sup_{D\in \mathcal{C}} p(D)$. Here are some open questions (on some of which conflicting conjectures have been made.)
\begin{ques}[\cite{BDFGM21, HIP22, LK21, Bra21}]
    What is $\mu(\mathcal{P}lanar)$? Is it $\infty$?
\end{ques}
\begin{ques}[\cite{Bra22, HL20}]
    What is $\mu(\mathcal{O}uterplanar)$?
\end{ques}
\begin{ques}[\cite{HL20, Bra21}]
    What is $\mu(\mathcal{N}o\mathcal{C}_4\mathcal{S}ubgraph)$? Is it $\infty$?
\end{ques}
\begin{ques}[\cite{Bra21}]
    Is it true that $\mu(\mathcal{N}o\mathcal{HS}ubgraph)=\infty$ if and only if $H$ includes a cycle?\footnote{One might also want to study families with multiple forbidden subgraphs, such as arise in treewidth studies. By \cite{Bra21}, we know that if the family of forbidden subgraphs contains a tree (or forest), the bound is finite, and if it contains all cycles $C_k$ for $k>N$ for some $N$, it's also finite. Other cases are open. The question of whether $\mu$ is necessarily unbounded for finite tree-free families is essentially the same as Question \ref{girthquestion}.
    }
\end{ques}
Forbidden \textit{induced} subgraphs (i.e., those achieved by deleting vertices alone) are fascinating in e.g. treewidth studies. For $\mu$, they are not quite as immediately interesting.
\begin{prop}
    Given a graph $H$, $\mu(\mathcal{H}-\mathcal{I}nduced\mathcal{S}ubgraph\mathcal{F}ree)=\infty$ unless $H\in \{P_1,P_2\}$. Given any $t$, $\mu((\mathcal{H},K_t)-\mathcal{I}nduced\mathcal{S}ubgraph\mathcal{F}ree)=\infty$ if $H$ is not complete bipartite. 
\end{prop}
\begin{proof}
    For the first statement, if $H$ is not a clique, observe that $\mu(K_t)$ is unbounded. If $H$ is a clique, or for the second statement, recall that $\mu(K_{n,n})$ is unbounded. 
\end{proof}
\begin{ques}
    If $H$ is complete bipartite, is $\mu((\mathcal{H},K_t)-\mathcal{I}nduced\mathcal{S}ubgraph\mathcal{F}ree)<\infty$?
\end{ques}
Minors are interesting, though. (For minors, we're allowed to delete edges, delete vertices, and contract edges.) By \cite{Bra21}, we know that $\mu(\mathcal{N}o\mathcal{HM}inor)$ is finite if $\mathcal{H}$ is a tree or a cycle.
\begin{ques}
    Given a graph $H$, what is $\mu(\mathcal{N}o\mathcal{HM}inor)?$ In particular, is it always finite?
\end{ques}

\begin{ques}
    Is $\mu(\mathcal{W}heels)>4$? Is $\mu(\mathcal{C}actuses)>4$? Is $\mu(\mathcal{K}ayak\mathcal{P}addles)>4$?\footnote{It's easy to show that they're all at least 4 by (respectively) $K_4$, the bowtie, and $K\mathbf{P}(3,3,1)$. In fact, I'd guess that $\mu(x,y,z)=3$ if either $x$ or $y$ is $>3$.} What is $\mu(\mathcal{S}eries\mathcal{P}arallel)$? 
\end{ques}

If you want to ask these questions with $\mu_{s>1}$ or $\hat{\mu}$, be aware that $\hat{\mu}$ is unbounded for any family with unbounded $\Delta$.

Some other families it could be interesting to consider might be Apollonian networks, unit-distance graphs,\footnote{``What is the hat guessing number of the plane?''} grids,\footnote{``What is the hat guessing number of $\mathbb{Z}^{2}$?''} hypergrids,\footnote{Cartesian products of paths––``What is the hat guessing number of $\mathbb{Z}^{k}$?''} triangular grids, cubical, Mycielski, and cubic graphs seem noteworth. Laman, modular, CPG, and perfectly orderable graphs could be interesting. 
Relative neighborhood, Urquhart, Gabriel, and Delaunay graphs could move us toward planar; from the other direction one could try apex graphs. Perhaps there's something to be achieved from considering the $k$ in $k$-outerplanar graphs in combination with Bradshaw's result on ``layered outerplanar'' graphs. 

Of course, for classes $\mathcal{C}$ such that $\mu(\mathcal{C})=\infty$, we can still wonder about the order of growth of $\mu$ in terms of $|V(G)|$. This transforms the question into, generally, ``for graph class $\mathcal{C}$, find the least $f_\mathcal{C}$ such that $\mu(G)\leq f_\mathcal{C}(|V(G)|)$ for all $G\in \mathcal{C}$.'' 
Asymptotic behaviors of $\mu$ have been an object of curiosity since \cite{BHKL09}.

\begin{ques}
    For what $b$ is $\mu(K_{n,n})\in \Theta(b(n))$? For what $b_r$ is $\mu(K_{n,n,...,n})\in \Theta(b_r(n))$, where the graph is $r$-partite?
    \end{ques}

\begin{ques}[\cite{GIP20, HIP22}]
    Is it true that for $n\geq 3$, $\mu(K_{n,n})\leq n$?
\end{ques}

\begin{ques}
    Let $G$ be a complete $r$-partite graph on $k$ vertices. Is $\mu(G)$ maximized by $G=K_{\lceil k/r \rceil,...,\lfloor k/r \rfloor}$?
\end{ques}

\begin{ques}
    For which families $\mathcal{C}$ is $\mu(G)\in O(\log |V(G)|)$ for $G\in \mathcal{C}$? In particular, for which $H,k$ is $(\mathcal{H},K_t)-\mathcal{I}nduced\mathcal{S}ubgraph\mathcal{F}ree$ such a family?
\end{ques}

The idea of considering lower bounds for graph classes is dull much of the time––most families include boring things like discrete graphs, i.e. those with $V=[n],E=\emptyset$. 

\subsection{Bounding with functions of other parameters}

Another common form is, ``what does $\mu$ have to do with parameter $p$?''
Some provocative examples (mentioned above) were disproved. Here are some unresolved matters.
\begin{ques}[\cite{BDFGM21, HL20}]
    Is $\mu(G)\geq \chi(G)-1$ for all $G$?
\end{ques}
\begin{ques}[\cite{BDFGM21,HIP22,Bra21,HL20}]\label{girthquestion}
    Is it true that for all integers $m,n$, there exists a graph with $\mu\geq m$ and girth $\geq n$?
\end{ques}
\begin{ques}
    For what $\epsilon$, if any, can we get $\mu <(e-\epsilon)\Delta$? (As a reminder, $e$ is Euler's number. See Proposition \ref{prop:application3}.)
\end{ques}

Here is a general template. 
\begin{ques}
    Does there exist a function $f$ such that $f(p(G))\leq \mu(G)$? How big can $f$ be? Does there exist a function $g$ such that  $\mu(G)\leq g(p(G))$? How small can $g$ be?
\end{ques}
Popular parameters for consideration include degeneracy \cite{ABST20, HL20, BDFGM21, KMS21}, minimal degree \cite{ABST20, Szc17, BDFGM21}, chromatic number \cite{BDFGM21,ABST20}, and girth \cite{BDFGM21,HIP22,Bra21,HL20}. (Citations here include both question-posing and partial results.) We still find these the most interesting and challenging. We also suggest considering genus, thickness, edge-density, page number, fractional chromatic number, list chromatic number, circuit rank, arboricity, boxicity, strength, Colin de Verdière number, and any form of -width or -depth. 

Some may only become interesting or tractable when combined with other parameters. For instance, \cite{BDFGM21} proved that girth and genus together, or order and chromatic number together, bound $\mu$. Further results of that form answer to this question, for parameters $p_1,...,p_n$ of $G$.
\begin{ques}
    Do there exist $f,g$ such that $f(p_1(G),...,p_n(G)) \leq \mu(G)\leq g(p_1(G),...,p_n(G))$? What are the restrictions on the functions $f,g$?
\end{ques}

Perfect graphs are an interesting role model. Just as $\omega(G)\leq \chi(G)$ immediately, we have $\omega(G)\leq \mu(G)$, and equality sometimes holds (for instance, in cliques). In both cases, we want to know: when are they equal? To make sure the answer is interesting––and can't just be spoiled by adding huge cliques to any graph––we specify that it must hold for any induced subgraph. 

\begin{ques}[``Hat-perfect graphs'']
    Let $\mathcal{F}$ be the graph family where $G\in \mathcal{F}$ if and only if $\mu(H)=\chi(H)$ for every induced subgraph $H$ of $G$.
\end{ques}
Finally, even for parameters that are formally shown to be independent from $\mu$, one might still wonder if they still ``generally have something to do with $\mu$.''
\begin{ques}
    If you're trying to maximize $\mu(G)$, and you're offered two graphs of equal order and size, do you (on average) prefer the one with higher diameter, lower diameter, or does it not matter? 
\end{ques}

\subsection{Changing $\mu$ via graph operations}\label{AddingEdgesAndVertices}
Most graph parameters don't change drastically by the addition or removal of a single edge or vertex. $\mu$ doesn't either, not usually. The ``typical'' edge-addition doesn't make much of a difference. Consider the evolution from $\overline{K_n}$ to $K_n$ by adding one edge at a time (in whatever order). We add $\binom{n}{2}$ edges for an increase from $\mu(\overline{K_n})=1$ to $\mu(K_n)=n$, so on average, a new edge increases the hat guessing number by $\frac{2}{n}$. So most of them don't increase $\mu$ at all. 

Yet edge-additions can be arbitrarily dramatic. It was shown in \cite{HIP22} that for fixed $d$ and sufficiently large $n$, $\mu(B_{d,n})=1+\sum_{i=1}^d i^i$. (It's even nicer aesthetically if we allow $0^0=1$.) It was shown in \cite{HT13} that $\mu(K_{d,n})\leq d+1$. You need add only $\binom{d}{2}$ edges to turn  $K_{d,n}$ into $B_{d,n}$. In that process, the average increase in $\mu$ induced by a new edge is at least $\frac{ -d+\sum_{i=1}^d i^i}{\binom{d}{2}}$, which is arbitrarily large for arbitrarily large $d$. Thus: 

\begin{prop}\label{prop:rightherebro}
    For any integer $k\geq 0$, there exist graphs $G,G'$ with $G'\backslash e =G$ for some $e\in E(G')$ and $\mu(G')-\mu(G)>k$.
\end{prop}

We can use books to find extreme examples for vertices, as well! Let $B_{d,n}$ have $n$ sufficiently large that $\mu(B_{d,n})=1+\sum_{i=1}^d i^i$. Then, by deleting a single vertex in the central clique, we have $B_{d-1,n}$, with $\mu(B_{d,n})-\mu(B_{d-1,n})=d^d$, which is arbitrarily large for large enough $d$. So similarly:
\begin{prop}\label{prop:alsorightherebro}
    For any integer $k\geq 0$, there exist graphs $G,G'$ with $G'\backslash v =G$ for some $v\in V(G')$ and $\mu(G')-\mu(G)>k$.
\end{prop}

Our constructions are somewhat brutish, however. Could they be refined?
\begin{ques}
    Does Proposition \ref{prop:rightherebro} hold if we replace $>$ with $=$? How about Proposition \ref{prop:alsorightherebro}?
\end{ques}
We'd guess ``yes,'' and it can probably be attained by controlling the number of pages in a book. 
The constructions also require big graphs. Perhaps that requirement is intrinsic. If so, why?
\begin{ques}
    Suppose $G$ is obtained from $G'$ by deleting a single vertex. What is the maximum value of $\mu(G')-\mu(G)$ in terms of $|V(G)|$? What about in terms of $|E(G)|$ or $|E(G')|$? What if $G$ is obtained by deleting an edge? What if $G$ is obtained by deleting an edge, and we care in terms of $|E(G)|$?
\end{ques}

Finally, it seems like adding a universal vertex is quite powerful. Is it always?

\begin{ques}
    Let $G$ be a graph, and let $G'$ be obtained from $G$ by adding a vertex universal to every vertex of $G$. Is it possible that $\mu(G)=\mu(G')$? Let $G'$ be obtained from $G$ by adding $k$ nonadjacent universal vertices. What is the minimum possible value of $\mu(G')-\mu(G)$? What if the new vertices are adjacent?
\end{ques}

Similar propositions and questions work for digraphs, replacing ``edge'' with ``arc.'' And, of course, these aren't the only intereesting operations. For instance:

\begin{ques}
    Let $D_1$ and $D_2$ be digraphs. For a given notion of product, how does $\mu$ of the product of $D_1$ and $D_2$ relate to $\mu(D_1)$ and $\mu(D_2)$?
\end{ques}
Let $\overline{D}$ denote the complement of $D$, i.e. the digraph defined by $V(\overline{D})=V(D)$ and $E(\overline{D})=V(D)\times V(D)\backslash E(D)$ (except for loops, which we exclude).
\begin{ques}
    For what $f,g$ is $f(|V(G)|)\leq \mu(G)+\mu(\overline{G})\leq g(|V(G)|)$ (for undirected\footnote{There exists an infinite family of digraphs $D_n$ with $|V(D_n)|=n$, $\mu(D_n)=1$, and $D_n$ isomorphic to $\overline{D_n}$: set $V(D_n)=[n]$ and $i\rightarrow j$ if and only if $i<j$} $G$).
\end{ques}

Clearly $g(n)\leq 2n-1$, since $G$ and $\overline{G}$ can't both be cliques. Also, $g(n)\geq n+1$, as shown by $G$ a clique. That's insufficient, though. Consider the bowtie graph $B$. $|V(B)|=5$, and $\overline{B}$ is $C_4$ plus a stray vertex, so $\mu(\overline{B})=3$, so $g(5)\geq 7$. 
We have $f(n)\leq n+1$ by considering $D$ a clique, but that's also insufficient: $\overline{C_5}=C_5$ and $\mu(C_5)=2$, so $f(5)\leq 4$. By \cite{AC22}, $f\in \Omega(n^{1-o(1)})$. We have $g\in \Theta(n)$ and would guess that $f\in \Theta(n)$ too. 

 Another fundamental operation is edge contraction. 
 Contracting an edge of $C_{3k+7}$ increases $\mu$ by 1, but contracting an edge of $C_{3k+6}$ decreases $\mu$ by 1, and contracting an edge of $C_{3k+5}$ does nothing. We can decrease $\mu$ arbitrarily by contracting an edge in the main clique of a book. 
\begin{ques}
    In terms of $|V(G)|$ or $|E(G)|$, what is the maximum value of $\mu(G')-\mu(G)$, where $G$ is obtained from $G'$ by contracting an edge?
\end{ques}
\begin{ques}
    What is the maximum value for $\mu(H)-\mu(G)$, where $H$ is a minor of $G$? 
\end{ques}
\begin{ques}
    If $H$ is a minor of $G$ and $H\neq G$, is $\hat{\mu}(H)<\hat{\mu}(G)$? 
\end{ques}

\begin{ques}
    What are extremal values of $\mu(G)-\mu(H)$, where $G$ and $H$ are related by $Y-\Delta$ transformation?
\end{ques}

\begin{ques}
    What is the relationship between $\mu(G)$, $\mu(L(G))$, and $\mu(L(L(G)))$? (This seems especially tractable for $\hat{\mu}(G)$.)
\end{ques}

\section{Computational complexity aspects}\label{sec:computation}
This very unrigorous section speculates how one might fulfill some aims of Section \ref{OurApproach}. 

In Chapter 0, we alluded to the fact that you can't solve $(K_7,7)$ by brute force, that there are just too many strategies and colorings to check. You might wonder: how many? 
Let $\mathcal{G}\equiv (D, g, h)$ be a Czech game. There are 
\[\prod_{v\in V(D)} h(v)\] 
different colorings. Each sage $v$ has 
\[\binom{h(v)}{g(v)}^{\prod_{\overrightarrow{vw}\in E(D)} h(w)}\]
possible plans, so the group has a whole has 
\[\prod_{v\in V(D)} \binom{h(v)}{g(v)}^{\prod_{\overrightarrow{vw}\in E(D)} h(w)}\]
possible strategies. 

Speaking loosely, for $n=|V(D)|$, the number of colors is on the order of $h^n$, and the number of possible plans is on the order of $h^{gnh^{\Delta^+}}$. 

\begin{ques}
    Roughly circumscribe the family of games that can be practically checked by brute force. 
\end{ques}

These numbers get big fast. For instance, suppose we're playing a rather ordinary Polish game $(K_4, 2,8)$. Then we'd have to check $8^{8^3 \cdot 2\cdot 4}=2^{12,288}$ different strategies. (And check them against $8^4=2^{12}$ colorings each, but at this point, another factor of $2^{12}$ isn't especially impressive.) Yet we know at a glance how to win it. We'd like to develop that power for other games. When is that possible?

\begin{ques}
    Consider the language consisting of winnable Czech games. What's its complexity? How about the language of winnable Polish games, or of winnable Latvian games?
\end{ques}
We might also want to know how hard it is to evaluate a strategy––even if for no other reason than the fact that a winning strategy is a certificate for a winnable game. 
\begin{ques}
    Consider the language consisting of pairs $(\mathcal{G},f)$, where $\mathcal{G}$ is Czech game and $f$ wins $\mathcal{G}$. What's its complexity? How about Polish, Latvian, etc.?
\end{ques}
We'll discuss these questions here. Rigor is held in abeyance: in particular, we don't specify what the encoding scheme is for a game, so we avoid terms like ``2EXPTIME'' or ``NP.''

One might consider the performance of brute-force algorithms. The previous subsection guarantees a certain time-bound, and it's also not too hard to compute a space-bound. 
\begin{ques}
    For your favorite encoding scheme, what complexity classes do the brute-force algorithms restrict our recognition languages to?
\end{ques}

For specific cases, certain sub-problems have received algorithms that solve them efficiently. Current examples include ``Is $(G,2)$ winnable?'' (obvious), ``Is $(G,3)$ winnable?'' (\cite{KL19}), ``Is $(G,r)$ a winnable Ratio game, where $G$ is chordal?'' (\cite{BDO21}), ``Is $(K_n,g,h)$ winnable?" (\ref{CzechOnComplete}), and so forth. We've added a few to this list: ``Is $(\overrightarrow{C}_k,h)$ winnable?'' (\ref{LatvianDirectedCycles}), ``Is $(\overrightarrow{C}_k,\star s,\star h)$ winnable?'' (\ref{LatvianPolish}), ``Is $(T,h)$ winnable?" (\ref{thm: FirstTreeCategorization}/\ref{TreeAlgo2}), and ``Is $(C_k,h)$ winnable?'' (\ref{CategorizedTwoToFour} and \ref{CategorizedFiveAndUp}). 

As discussed above, we'd like to expand the territory of efficiently solved games and/or circumscribe that territory, ideally by showing that some relatively simple games are computationally hard. I suggest Czech games on $P_3$ or $\overrightarrow{C}_3$. If these turn out to be efficiently solvable, perhaps $K_{1,n}$ or $\overrightarrow{C}_n$ is thorny enough. See Subsection \ref{subsec:nontrivczechgames} for rephrasings that might be amenable to computational argument. I have no particular suggestions for Latvian games. 
\begin{ques}
    What's the computational hardness of Czech games on $P_3$? On $\overrightarrow{C}_3$? On $K_{1,n}$? How about $\overrightarrow{C}_n$?
\end{ques}
$K_{1,n}$ might be particularly fruitful. Professor Tanya Berger-Wolf points out that we've translated Czech stars to a packing problem on an integer lattice, which is related by duality to hitting and integer programming.
\begin{ques}
    What are the associated hitting problem and integer programming problem for the combinatorial prism packing problem?
\end{ques}
More generally, note that ``does this strategy win?'' is equivalently ``can I find a disprover?'' which is a specialized graph-coloring problem. There are many computationally hard graph-coloring problems (traditionally dealing with mere proper coloring). 
\begin{ques}
    Is there a hard graph-coloring problem that can be rephrased as a ``does this strategy win?" for a Czech game? How about Polish, Latvian, or Classic? What implications does this have (or not) for the ``is there a winning strategy'' question?
\end{ques}

We should note that \cite{KL19} reduces ``does this strategy work?'' to SAT for certain undirected Classic games. We have not examined it closely.
\begin{ques}
    Is the reduction of \cite{KL19} efficient? Can it be extended to directed Czech games? What does it imply about the complexity classes of problems we're interested in?
\end{ques}

Even without knowing what other languages might reduce to hat games, it's useful to stack and bundle the questions about hat-game complexity by knowing how the different questions relate to one another. 

First off, one might wonder about the complexity of the problem ``given $D$, output $\mu(D)$.'' However, any question of the form ``What is $\mu_s$ for $D$?'' can be answered by repeating the question ``Is $(D,\star s,\star h)$ winnable?'' for every constant $h\leq s|V(D)|$. Thus, the problem ``determine $\mu_s$'' has, for any fixed $s$, an efficient translation into problems of the form ``is this Polish game winnable?''

Now we'll consider which variants of the game can be reduced to one another. 
Obviously, the Classic game is a special case of both the Latvian and fractional games, which are both a special case of the Czech game. Undirected games are special cases of directed games. 
One version we didn't mention requires the addition of a ``guessing graph'' \cite{BHKL09}, a directed multigraph whose vertices are sages and whose arcs $\overrightarrow{uv}$ each indicate ``a chance for $u$ to guess the hat color of $v$.'' As usual, a single correct guess is victory. The Czech game is just the special case of this game where all the guessing graph's arcs are loops.\footnote{Strictly speaking, the text doesn't suggest variable hatness or make it clear that the guessing graph is a multidigraph, but it's an obvious generalization to make.} 

This game is in turn a special case of a variation considered in \cite{KL19} that allows variable hatness, a single guess, and the constraint that $c$ must be equal among certain preannounced sets of vertices.\footnote{Again, it wasn't considered quite this explicitly.} The equivalence from the previous game to this one is as follows. You have a vertex $A$ in the new game for every arc $\overrightarrow{uv}$ of the guessing graph. Suppose $B$ came from $\overrightarrow{xy}$. If there's an arc in the visibility graph from $u$ to $y$, add an arc from $A$ to $B$. If $y=v$, announce the stipulation $c(A)=c(B)$. The equivalence can easily be reversed.

So, allowing $\implies$ to denote ``can be polynomially reduced to,'' we have Classic $\implies$ Polish, Latvian $\implies$ Czech $\implies$ guessing-graph $\iff$ equality-stipulation, and for each version we have undirected $\implies$ directed. (Mostly, this reduction is accomplished by ``inclusion as a special case.'') 
\begin{ques}
   Can any of these $\implies$ relations be reversed?
\end{ques}

Finally, some authors \cite{GIP20} have considered arbitrary restrictions on $c$. Even if we restrict to a single guess, a constant hatness $k$, and graphs consisting only of disjoint copies $P_2$, having \textit{arbitrary} restrictions on $c$ is enough to encode most anything. 
A \emph{distributed guaranteed approximation game} is a quadruple $(S,A,K,G)$, for States of the world, Agents, Knowledge, and Guessing targets. Formally, $S$ and $A$ are sets; $K$ and $G$ are functions $A\rightarrow \mathsf{Part}(S)$; $W$ is a function $S\rightarrow 2^{2^A}$. A \emph{strategy} for a d.g.a.s. is a function $P$ that takes each $a\in A$ to an element of $\mathsf{Hom}(K(a), G(a))$. The game is winnable if and only if, $\forall_S s\exists_A a\forall_{K(a)}k(s\in k \implies s\in \mathbf{P}(k))$. 
A d.g.a.s. can be efficiently encoded as a Classical hat game with arbitrary restrictions on $c$. Details are left to the reader.

All the above makes it seem likely that evaluating the winnableness of a hat game is computationally very hard. Perhaps there are some practical fudges/compromises to make.
\begin{ques}
    Is there an efficient random algorithm to test for winnableness with (e.g.) 3/4 confidence? Can it be iterated to $1-\epsilon$ confidence?
\end{ques}
\begin{ques}
    Over some distribution of games, is there an algorithm whose average-case running time is good?
\end{ques}

\section{Ratio games, probability, and $\hat{\mu}$}\label{subsec:RatioGames}
We have some foundational and definitional queries about Ratio games. For instance, how ``obvious'' or ``natural'' is our particular definition?
One might have \textit{defined} the Ratio game by \ref{RatioMonotone}––that it's winnable if and only if it's no harder than a winnable Czech game. This seems perverse: consider the Ratio game $\mathcal{G}\equiv (K_2,r)$ with $r(v_1)=\pi^{-1}, r(v_2)=1-\pi^{-1}$. There's no winnable Czech game $(K_2,r')$ with $r'\leq r$. Yet $\mathcal{G}$ is a game on a clique with $\sum r(v)\geq 1$. That was the condition for Czech games to win on cliques, and we'd like to be the case for Ratio games, instead of ``either $\sum r(v)>1$ or $\sum r(v)=1$ but they're all rational.'' Definition \ref{def:RatioGames} avoids that nastiness.

Another proposal could be that a Ratio game is winnable if and only if some sequence of winnable Czech games approaches it. (One aesthetic objection is that it feels asymmetrical.)
\begin{defn}\label{winnableseq}
    For a Ratio game $\mathcal{G}\equiv (D,r)$, consider a sequence of Czech games $\mathcal{G}_1\equiv (D,g_1,h_1),\mathcal{G}_2\equiv (D,g_2,h_2),...$ such that $\lim_{n\rightarrow \infty} r_n(v)=r(v)$ for all $v\in V(D)$. If all $\mathcal{G}_i$ are unwinnable, we call this an \textit{unwinnable sequence}. If all $\mathcal{G}_i$ are winnable, we call it a \textit{winnable sequence}.
\end{defn}
 The following easy propositions (whose proofs we leave as an exercise) elucidate the connection between Definitions \ref{def:RatioGames} and \ref{winnableseq}.
\begin{prop}
    If $\mathcal{G}$ is unwinnable, it has an unwinnable sequence. If $\mathcal{G}$ is winnable, it has a winnable sequence. 
\end{prop} 
\begin{prop}
    Let $\mathcal{G}$ be a Ratio game. Suppose there exists winnable Czech $\mathcal{G}'$ such that $\mathcal{G}'\preceq \mathcal{G}$. Then $\mathcal{G}$ has a winnable sequence. If there does not exist such a $\mathcal{G}'$, then $\mathcal{G}$ has an unwinnable sequence. 
\end{prop}

\begin{prop}
    There exist winnable Ratio games $\mathcal{G}$ that have both winnable and unwinnable sequences. Some have winnable Czech $\mathcal{G}'$ such that $\mathcal{G}'\preceq \mathcal{G}$, and some do not.
\end{prop}
The main question, that of equivalence, thus comes down to:
\begin{ques}
    Does there exist an unwinnable game with a winnable sequence?
\end{ques}
It seems also that these problems only arise on the border between winnableness and unwinnableness. 
\begin{ques}
    What, precisely, delineates that border?
\end{ques}

Further questions arise when this thinking about Ratio games as limits of Czech games. 
\begin{ques}
    Let $\mathcal{G}\equiv(D,r)$ be a Ratio game with $r(v)\in \mathbb{Q}$ for all $v$. Suppose every Czech game of which $\mathcal{G}$ is the associated Ratio game is unwinnable. Must $\mathcal{G}$ be unwinnable? 
\end{ques}   

\begin{ques}\label{MeasureEpsilon}
  Let $\mathcal{G}$ be a Ratio game such that, for any $\epsilon>0$, there is a strategy $f_\epsilon$ such that the probability of losing is $<\epsilon$. Must there exist a strategy $f$ such that the probability of losing is $0$? 
\end{ques}
\begin{ques}\label{MeasureZero}
    Let $\mathcal{G}$ be a Ratio game with a strategy $f$ such that the probability of losing is $0$. Must $\mathcal{G}$ be winnable?
\end{ques}

It also behooves us to consider the question implicitly posed by \cite{BDO21}: whether an undirected Ratio game's winnableness (modulo \ref{MeasureZero}) can be totally determined by evaluating $r$ in the multivariate independence polynomial of $G$.
\begin{ques}
    Let $\mathcal{G}$ be a Ratio game on a graph $G$. Prove or disprove: if $Z_G$ is nonvanishing in $\overline{D}_r$, must $\mathcal{G}$ be winnable?
\end{ques}

If that is proven, then we may fairly say that hat-guessing rightness is a generic phenomenon for anticorrelating events on a dependency graph. Perhaps this is true for dependency digraphs as well, though we lack the clean criterion of Shearer's lemma. 
\begin{ques}\label{PolynomialWorksForAllRatioGames}
    Let $\mathcal{G}\equiv (D,r)$ be a Ratio game. If the probability vector $r$ is not good for $D$, must $\mathcal{G}$ be winnable?
\end{ques}

By figuring out Ratio games, one could distinguish ``Czech games that are unwinnable because there isn't enough $r$ spread around in the right places'' from ``Czech games that are unwinnable simply because the hats are too discrete.''

\begin{ques}
    Let $\mathcal{G}$ be a Czech game for which the associated Ratio game is winnable, which nevertheless is unwinnable. What can we say about $\mathcal{G}$?
\end{ques}

The question isn't strictly pertinent to hat guessing, but we should acknowledge that \ref{thm:nodirectedshearers} only eliminates odd directed cycles. 
\begin{ques}
    If $D$ has no odd directed cycles, do we have ($Q_D$ nonvanishing in $\overline{D}_r$ $\iff$ $r$ is good for $D$)? 
\end{ques}
Finally, $\hat{\mu}$ intrigues us. 

\begin{ques}
    For $\Delta(G) \leq k$, what is $\sup \frac{\hat{\mu}(G)}{\mu(G)}$? Is it attained? Can it be attained or approached by a sequence of trees?
\end{ques}

\begin{ques}
    Does adding an edge or a leaf to a graph always increase $\hat{\mu}$?
\end{ques}

\begin{ques}
    Are cliques the only connected graphs\footnote{We can't say ``digraphs'' here, since we have $\mu(\overrightarrow{C}_k)=\hat{\mu}(\overrightarrow{C}_k)=2$.} for which $\hat{\mu}(G)=\mu(G)$?
\end{ques}

\begin{ques}
    What is $\sup \frac{\hat{\mu}}{\Delta}$? $e$?
\end{ques}
\begin{ques}
    Let $D_n$ be an infinite sequence of graphs and $g(n)$ a function. Is it true that, if 
    $\underset{n\rightarrow \infty }{\lim \inf} \mu(D_n) \geq g(n)$ and $|V(D_n)|\rightarrow \infty$, then $\underset{n\rightarrow \infty }{\lim \inf} \max_f \mathbf{P}(f \text{ wins } (1,g(n)))=1$? Does it hold if we replace $\inf$ with $\sup$? What if we eliminate $\inf$?
    \end{ques}
    
    \begin{ques}
        Fix $k\in \mathbb{N}$, and suppose that a (strongly connected, nontrivial) Latvian game $(D,h)$ is winnable with $h\geq k$ $h\neq \star k$. Must $\hat{\mu}(D)$ be $>k$?
    \end{ques}

    (One could formulate similar questions involving $\mu_s$ for various $s$.)

\section{Constructors}\label{Constructorquestions}
The fact that our constructors  (as we've said) allow us to polynomially build new winning strategies on composite graphs. This intuitively warns me that the constructors can be loose: probably we could get better strategies if we spent more time on them. In \cite{KLR21b}, they showed that the single-point product constructor wasn't tight. Specifically, they showed that although the games on $K_5$ with hatnesses $(4,5,5,5,6)$ couldn't be made any more difficult, their product on the hatness-6 vertices \textit{could} be made more difficult: specifically, the hatness of that central vertex could be increased from 36 to 37. That's not much, but we don't have a good sense of how big these discrepancies could get. 
\begin{ques}
    How tight/loose are our constructors? Can we tighten them?
\end{ques}
\begin{ques}
    Is there an analogue to Theorem 2.3 of \cite{LK23} for some of our constructors?
\end{ques}
\begin{ques}
    How may the constructors presented in this thesis be altered to include reduction, and is that useful at all?
\end{ques}
\begin{ques}
    How might the other constructors of \cite{KLR21b} generalize to directed Czech games?
\end{ques}
\begin{ques}
    If we're attaching a universal vertex $x$, the changes we can make are much more drastic if $N(x)$ is a clique. Can we formulate any gains if $N(x)$ is dense but not quite a clique?
\end{ques}

\begin{ques}\label{tantamounttodegen}
    Does there exist a $b$ such that, if $\deg(v)\leq k$ and $h(v)>b(k)$, $v$ can be deleted without changing the outcome of any undirected Latvian game?
\end{ques}
If the answer to Question \ref{tantamounttodegen} is ``yes,'' then $\mu(G)\leq b(d)$, where $d$ is the degeneracy. So if $b$ exists, $b(d)\geq 2^{2^{d-1}}$ by \cite{HL20}. 
\begin{ques}\label{trianglefree}
    Does there exist an $f$ such that, if $\deg(v)\leq k$, $h(v)>f(k)$, and $N(v)$ is independent, $v$ can be deleted without changing the outcome of any undirected Latvian game?
\end{ques}
Question \ref{trianglefree} would have the same effect, but only for triangle-free graphs. 

\begin{ques}
    What gains are there to be made by analyzing undirected games (especially Latvian games) with the following scheme?
    \begin{itemize}
        \item Find a tree decomposition $(T,\chi)$ for $G$ that somehow suits your purposes.
        \item Determine how $\mathcal{G}$ (or a harder game) can be constructed from the maximal Latvian game $h(v)=2^{\deg(v)}$ on $T$ and several smaller games $\mathcal{G}_i$.
        \item Profit.
    \end{itemize}
\end{ques}

\begin{ques}
    Let $\mathcal{G},\mathcal{H}$ be winnable games on digraphs $D,F$ with $V(D)=V(F)$ and $E(D)\cap V(F)=\emptyset$. Is there a constructor that allows us to create a winnable game $\mathcal{J}$ played on $D\cup F$ with higher $h$ and/or lower $g$ than either $\mathcal{G}$ or $\mathcal{H}$?
\end{ques}
If such a thing were possible, our understanding of trees and cycles would be very useful, since the study of decomposing edge-sets into graphs of simple families is quite rich. 

\begin{ques}
    In general, does there exist a winnability-preserving constructor $\mathcal{G}\mapsto \mathcal{G}'$ with $D=D'$ and $r\not \leq r' \not \leq r$?
\end{ques}

\section{Miscellaneous}

\begin{ques}[\cite{BHKL09}]
    For an integer $k>2$, are there infinitely many edge-critical graphs with $\mu=k$?
\end{ques} 
\begin{ques}[\cite{HT13}]
    ``Given cardinals (possibly finite, treated as sets) $c,m,k$, with $k<c$, what is the smallest size of a family $F$ of functions from $m$ to $c$ such that, for  every subset $A$ of $m$ of size $k$, $f|_A$ is one-to-one for some $f \in F$?''
\end{ques}

Hardin and Taylor define $P_{k}(n,m)$ as the set of digraphs on $n$ vertices with a strategy for the $k$-hats Classic game guaranteeing that at least $m$ of the sages guess right. (We've concerned ourselves with $m=1$.) The very general family of questions is: 
\begin{ques}[\cite{HT13}]\label{ques:howmanycanweget}
    What is $P_{k}(n,m)$ for various values of $k,n,m$?
\end{ques}

Grytczuk conjectured the following: 
\begin{ques}[\cite{Szc17,BDFGM21}]
    Prove or disprove: if $h(v)\leq \deg^-(v)$ for all $v\in V(G)$, then $\mathcal{G}$ is winnable. 
\end{ques}

The $\mu$ of random graphs has also received some attention \cite{BDFGM21, AC22, WC23}.
\begin{ques}[\cite{HL20}]
    Let $p(n)$ be such that $np(n)\rightarrow \infty$ as $n\rightarrow \infty$. Is it true that for every constant $C$, we have $\mu(\mathcal{G}(n,p(n)))>C$ with high probability?
\end{ques}
\begin{ques}
    Consider a function $b(n)$. (Even a constant would be quite interesting.) What is the threshhold value for the event $\mu(\mathcal{G}(n,p(n)))\geq b(n)$?
\end{ques}
\begin{ques}[Posed by Noga Alon]
    Can we find a better upper bound for $\mu(\mathcal{G}(n,p))$? In particular, one that's not based on the chromatic number?
\end{ques}
Perhaps the above questions can be attacked using graph martingales as in Chapter 7 of \cite{AS08}.
Broadly, stochastic processes of $\mu(G)$ seem interesting. 
\begin{ques}
    The results on Latvian cycles seem to suggest that ``smoothness'' matters. You don't want your hatness-2 vertex sandwiched too closely between hatness-4 vertices; that's more important than the average hatness of the graph. Hatness-5 vertices are especially destructive. Or, for paths, (2,4,...,4,2) is winnable, but (2,3,...,3,5,3,...,3,2) is unwinnable. Can this be formalized? What's going on? Does it generalize? 
\end{ques}

That's all I have. Thanks for reading.  Go in peace.


\begin{thebibliography}{99} \addcontentsline{toc}{chapter}{Bibliography}
   
    \bibitem[ABST20]{ABST20} N. Alon,  O. Ben-Eliezer, C. Shangguan, and I. Tamo. The hat guessing number of graphs. Journal of Combinatorial Theory, Series B,
    144:119–149, 2020.
    
    \bibitem[AC22]{AC22} N. Alon and J. Chizewer. On the hat guessing number of graphs.
    Discrete Mathematics, 345(4). 9pp. 2022. 
    
    \bibitem[ACLY00]{ACLY00} R. Ahlswede, N. Cai, S.-Y. R. Li, and R. W. Yeung. Network information flow. IEEE Transactions on Information Theory, 46(4):1204-1216. July 2000.

    \bibitem[Ada16]{Ada16} P. Adamson. Philosophy in the Islamic World. Oxford: Oxford University Press. 2016.

    \bibitem[AFGHIS05]{AFGHIS05} G. Aggarwal, A. Fiat, A.V. Goldberg, J.D. Hartline, N. Immorlica, and M. Sudan. Derandomization of Auctions. In Proceedings of the Thirty-Seventh Annual ACM Symposium on Theory of Computing. 619-625. Baltimore, MD. 2005.
  

    \bibitem[AH76]{AH76} K. Appel and W. Haken. Every Planar Map is Four Colorable. Bulletin of the American Mathematical Society, 82(5): 711-712. 1976. 

    \bibitem[Alo08]{Alo08} N. Alon. Problems and Results in Extremal Combinatorics - II. Discrete Mathematics, 308(19): 4460–4472. 2008.

    \bibitem[AS08]{AS08} N. Alon and J.H. Spencer. The Probabilistic Method. John Wiley \& Sons. Hoboken NJ. 2008.

    \bibitem[BDFGM21]{BDFGM21} B. Bosek, A. Dudek, M. Farnik, J. Grytczuk and P. Mazur.
    Hat chromatic number of
    graphs. Discrete Mathematics 344(12). 10pp. 2021.

    \bibitem[BDO21]{BDO21} V. Bla\v{z}ej, P. Dvo\v{r}ák, and M. Opler. Bears with hats and independence polynomials. In M. Pilipczuk, \L. Kowalik, and P. Rzążewski, 
    editors, Graph-Theoretic Concepts in Computer Science. 283-295. 2021.  
   
    \bibitem[Ber01]{Ber01} M. Bernstein. The hat problem and hamming codes. FOCUS. 21(8): 4-7. 2001.

    \bibitem[BGK02]{BGK02} E. Burke, S. Gustafson, and G. Kendall. In Proceedings of the 5th European Conference on Genetic Programming. 238-247. 2002.

    \bibitem[BHKL09]{BHKL09} S. Butler, M.T. Hajiaghayi, R.D. Kleinberg, and T. Leighton. Hat guessing games. SIAM Review, 51(2):399-413. 2009. 

    \bibitem[Blu01]{Blu01} W. Blum. Denksport für Hutträger. Die Zeit. May 3, 2001. 
    
    \bibitem[BNW15]{BNW15} O. Ben-Zwi, I. Newman, and G. Wolfovitz. Hats, auctions and derandomization. Random Structures and Algorithms. 46(3): 478-493. 2015. 
    
    \bibitem[Bra21]{Bra21} P. Bradshaw. On the hat guessing number and guaranteed subgraphs. ArXiV: 2109.13422. 8pp. 
    2021. 
    
    \bibitem[Bra22]{Bra22}P. Bradshaw. On the hat guessing number of a planar graph class. Journal of Combinatorial Theory, Series B, 156: 174-193. 2022. 
    
    \bibitem[Buh02]{Buh02} J.P. Buhler. Hat tricks. The Mathematical Intelligencer, 24(4): 44-49. 2002. 
    
    \bibitem[CHLL97]{CHLL97} G. Cohen, I. Honkala, S. Litsyn, and A. Lobstein. Covering codes. Amsterdam. 1997. 

    \bibitem[CRST09]{CRST09} M. Chudnovsky, N. Robertson, P. Seymour, and  R. Thomas. The strong perfect graph theorem. The Annals of Mathematics, 164(1): 51-229. 2006.

    \bibitem[deR14]{deR14} H.N. de Ridder et al. 
    Information System on Graph Classes and their Inclusions (ISGCI). https://www.graphclasses.org. 2001-2014.

    \bibitem[Ebe98]{Ebe98} T. Ebert. Applications of recursive operators to randomness and complexity. PhD Thesis, University of California at Santa Barbara. 1998. 

    \bibitem[EMV03]{EMV03} T. Ebert, W. Merkle, and H. Vollmer. On the autoreducibility of random sequences. SIAM Journal on Computing. 32(6): 1542-1569. 2003.

    \bibitem[Far16]{Far16} M. Farnik. A hat guessing game. PhD thesis, Jagiellonian University. 2016.

    \bibitem[Fin94]{Fin94} S. Finger. Origins of Neuroscience: A History of Explorations into Brain Function. New York: Oxford University Press. 1994. 

    \bibitem[Gad15]{Gad15}  M. Gadouleau. Finite dynamical systems, hat games, and coding theory. SIAM Journal on Discrete Mathematics, 32(3):1922-1945. 2015.

    \bibitem[Gad17]{Gad17} M. Gadouleau. On the Stability and Instability of Finite Dynamical Systems with Prescribed Interaction Graphs. The Electronic Journal of Combinatorics, 26(3). 18pp. 2017.

    \bibitem[Gar61]{Gar61} Gardner, M.: The 2nd Scientific American Book of Mathematical Puzzles \& Diversions. Simon and Schuster, New York. 1961.

    \bibitem[GG15]{GG15} M. Gadouleau and N. Gerogiou. New constructions and bounds
    for Winkler’s hat game. SIAM Journal on Discrete Mathematics,
    29(2):823–834. 2015.
    

    \bibitem[GIP20]{GIP20}  Z. Gu, Y. Ido, and B. Przybocki. Hat-guessing on graphs. https://surim.stanford.edu/publications/combinatorics/hat-guessing-graphs. 
    17pp. 2020. 
   
    \bibitem[GKRT06]{GKRT06} W. Guo, S. Kasala, M.B. Rao, and B. Tucker. The Hat Problem and Some Variations. In 
    Advances in Distribution Theory, Order Statistics, and Inference. Birkhäuser Boston, Boston, MA. 459-479. 2006. 

    \bibitem[GM90]{GM90} E. Goles and S. Martínez. Neural and automata networks: dynamical behavior and applications. Kluwer Academic Publishers, Norwell, MA. 1990.

    \bibitem[GR11]{GR11} M. Gadouleau and S. Riis. Graph-theoretical constructions for graph entropy and network coding based communications. IEEE Transactions on Information Theory, 57(10): 6703-6717. 2011. 

    \bibitem[Gui14]{Gui14} Z. Guido. ``Of Course!'': the greatest collection of riddles \& brain teasers for expanding your mind! Seemingly self-published. 2014.

    
    \bibitem[HIP22]{HIP22} X. He, Y. Ido, and B. Przybocki. Hat guessing on books and windmills. Electronic Journal of Combinatorics, 29(1). 2022.
    
    \bibitem[HL20]{HL20} X. He and R. Li. Hat guessing numbers of degenerate graphs. Electronic Journal of Combinatorics, 27(3), 18pp. 2020.
     
    \bibitem[HLLWX17]{HLLWX17} K. He, L. Li, X. Liu, Y. Wang, and M. Xia. Variable version Lovász Local Lemma: beyond Shearer's bound. In 2017 IEEE 58th Annual Symposium on Foundations of Computer Science (FOCS), 451-462. 2017. 

    \bibitem[HT13]{HT13} C. Hardin and A. Taylor. The mathematics of coordinated inference:
    a study of generalized hat problems. Springer International, New York,
    NY. 2013. 

    \bibitem[HV15]{HV15} N.J.A. Harvey and J. Vondrák. An algorithmic proof of the Lovász Local Lemma
    via resampling oracles. In In 2015 IEEE 56th Annual Symposium on Foundations of Computer Science (FOCS), 1327-1345. 2015. 

    \bibitem[JJG19]{JJG19} K. Jin, C. Jin, Z. Gu. Cooperation via codes in restricted hat guessing games. In AAMAS '19: Proceedings of the 18th International Conference on Autonomous Agents and MultiAgent Systems, 547-555. 2019. 
    

    \bibitem[KMS21]{KMS21} C. Knierim, A. Martinsson, and R. Steiner. Hat guessing numbers of strongly degenerate graphs. ArXiV: 2112.09619. 14pp. 2021. 

    \bibitem[KL19]{KL19} K.P. Kokhas and A.S. Latyshev. For which graphs the sages can guess correctly the color of at least one hat. Journal of Mathematical
    Sciences, 236:503–520, 2019. 
    
    \bibitem[KL21]{KL21} K.P. Kokhas and A.S. Latyshev. Cliques and constructors in ‘hats’
    game. i. Journal of Mathematical Sciences. 255:39-57, 2021. 
    
    \bibitem[KLP16]{KLP16} L. Kirousis, J. Livieratos, and K.I. Psaromiligkos.Directed Lovász Local Lemma and Shearer's Lemma. Annals of Mathematics and Artificial Intelligence, 88:133-155. 2016.
    
    \bibitem[KLR21a]{KLR21a} K.P. Kokhas, A.S. Latyshev,  and V.I. Retinskiy. Cliques and constructors in ‘hats’ game. ii. Journal of Mathematical Sciences, 255:58-70,
    2021. 
    
    \bibitem[KLR21b]{KLR21b}
    K.P. Kokhas, A.S. Latyshev.  and V.I. Retinskiy. Cliques and constructors in ‘hats’ game. 28pp. arxiv:2004.09605. 

    \bibitem[Krz10]{Krz10} M. Krzywkowski. On the hat problem, its variations, and their applications. Annales Universitatis Paedagogicae Cracoviensis, Studia Mathematica, 84(1): 55-67. 2010.

    \bibitem[LK21]{LK21} A.S. Latyshev and K.P. Kokhas. The hats game. the power of constructors [sic]. Journal of Mathematical Sciences, 255:124-131. 2021.

    \bibitem[LK22]{LK22} A.S. Latyshev and K.P. Kokhas. The hats game. On maximum degree and diameter [sic]. Discrete Mathematics, 345(7), 23pp. 2022. 

    \bibitem[LK23]{LK23} A. Latyshev and K. Kokhas. Hat guessing number of planar graphs is at least 22. 23pp. ArXiV: 2301.10305. 2023.



    \bibitem[LS02]{LS02} H. Lenstra and G. Seroussi. On hats and other covers. In Proceedings IEEE International Symposium on Information Theory, 342-342. 2002. 

    \bibitem[Mil23]{Mil23} E.S.V. Millay. The harp-weaver and other poems. Harper \& Brothers: New York, USA. 1923.

    \bibitem[She85]{She85} J.B. Shearer. On a problem of Spencer. Combinatorica, 5(3):241–245. 1985.

    \bibitem[SS06]{SS06}
    A.D. Scott and A.D. Sokal. On dependency graphs and the lattice gas. Combinatorics, Probability and Computing, 15: 253–279. 2006.



    \bibitem[Szc17]{Szc17}
    W. Szczechla. The three colour hat guessing game on cycle graphs.
    Electronic Journal of Combinatorics, 24(1). 13pp. 2017. 

    \bibitem[TD90]{TD90} R. Thomas and R. D'Ari. Biological Feedback. CRC Press, Boca Raton, FL. 1990.

    \bibitem[Pla60]{Pla60} Plato. Timaeus. Self-published: Athens, Greece. c. 360 B.C.

    \bibitem[Pou01]{Pou01} J. Poulos. Could you solve this \$1 million hat track? abcNews. November 29, 2002.

    \bibitem[PWWZ18]{PWWZ18}  R. Pratt, S. Wagon, M. Wiener, and P. Zielinkski. Too many hats. 10pp. ArXiV: 1810.08263. 2018.

    \bibitem[Rii07a]{Rii07a} S. Riis. Information flows, graphs and their guessing numbers. The Electronic Journal of Combinatorics, 14(1): 1-17. 2007. 

    \bibitem[Rii07b]{Rii07b}
    S. Riis. Graph entropy, network coding and guessing games. ArXiV: 0711.4175. 30pp. 2007.  

    \bibitem[Rob01]{Rob01} S. Robinson. Why mathematicians now care about their hat color. The New York Times, Science Times Section, page D5, April 10, 2001. 

    \bibitem[Uem16]{Uem16} T. van Uem. The hat game and covering codes. ArXiV: 1612.00276. 2016.

    \bibitem[Uem18]{Uem18} T. van Uem. Hats: all or nothing. ArXiV: 1801.01512v6. 29pp. 2018.

    \bibitem[WC23]{WC23} L. Wang, Y. Chen. The hat guessing number of random graphs with constant edge-chosen probability. arXiv:2302.04122. 5pp. 2023.

    \bibitem[WCGFB02]{WCGFB02} G.A. Winer, J.E. Cottrell, V.R. Gregg, J.S. Fournier, and  L.A. Bica. Fundamentally misunderstanding visual perception. Adults' belief in visual emissions. American Psychologist, 57(6-7): 417-24. 2002. 

   \bibitem[WCKC96]{WCKC96} G.A. Winer, J.E. Cottrell, K.D. Karefilaki, and M. Chronister. Conditions affecting beliefs about visual perception among children and adults. Journal of Experimental Child Psychology, 61(2):93-115. 1996. 
   
   \bibitem[WCKG96]{WCKG96} G.A. Winer, J.E. Cottrell, K.D. Karefilaki, and V.R. Gregg. Images, words, and questions: variables that influence beliefs about vision in children and adults. Journal of Experimental Child Psychology, 63(3):499-525. 1996. 


   \bibitem[Win01]{Win01} P. Winkler. Games people don’t play. In Puzzlers’ Tribute. 301–313. 

   \bibitem[WRC03]{WRC03} G.A. Winer, A.W. Rader, and J.E. Cottrell. Testing different interpretations for the mistaken belief that rays exit the eyes during vision. Journal of Psychology, 137(3):243-61. 2003.  



    \end{thebibliography}
\end{document}